\documentclass[a4paper,10pt]{article}
\usepackage{amscd}
\usepackage{amsmath}
\usepackage{bbding}
\usepackage{mathrsfs}
\usepackage{amssymb}
\usepackage{graphicx,latexsym,amsmath,amssymb,amsthm,enumerate}
\usepackage{color}
\usepackage{stmaryrd}
\allowdisplaybreaks[2]   %ÔÊÐí¹«Êœ·ÖÒ³ÏÔÊŸ
9.8in\textwidth 6.9in \hoffset -2.5cm \DeclareFontEncoding{OT2}{}{}
\DeclareTextSymbol{\cyrsftsn}{OT2}{126}
\DeclareTextSymbol{\textnumero}{OT2}{125}

\theoremstyle{plain}
\newtheorem{theorem}{Theorem}[section]
\newtheorem{corollary}[theorem]{Corollary}
\newtheorem{proposition}[theorem]{Proposition}
\newtheorem{definition}[theorem]{Definition}
\newtheorem{remark}[theorem]{Remark}
\newtheorem{example}[theorem]{Example}

\newtheorem{lemma}[theorem]{Lemma}
\begin{document}
\title{{\LARGE\bf{Stochastic Integrals on Predictable Sets of Interval Type with Financial Applications}\thanks{This work was supported by the National Natural Science Foundation of China (12171339).}}}
\author{Jia Yue$^a$, Ming-Hui Wang$^a$, Nan-Jing Huang$^b$\footnote{Corresponding author.  E-mail addresses: nanjinghuang@hotmail.com, njhuang@scu.edu.cn} \\
{\small\it a. School of Mathematics, Southwestern University of Finance and Economics,}\\
{\small\it Chengdu, Sichuan 610074, P.R. China}\\{\small\it b. Department of Mathematics, Sichuan University, Chengdu,
Sichuan 610064, P.R. China}}
\date{ }
\maketitle
\begin{flushleft}
\hrulefill\\
\end{flushleft}
 {\bf Abstract}.
In this paper, by extending the classic stochastic integrals, we investigate three kinds of more general stochastic integrals: Lebesgue-Stieltjes integrals on predictable sets of interval type (in short: PSITs), stochastic integrals on PSITs of predictable processes with respect to local martingales, and stochastic integrals on PSITs of predictable processes with respect to semimartingales. Such stochastic integrals on PSITs are defined only on restricted stochastic subsets, and their values outside the subsets do not matter. Our study reveals that a stochastic integral on a PSIT can be characterized by a coupled sequence of classic stochastic integrals. Furthermore, the It\^{o}'s formula for semimartingales on PSITs is developed for stochastic calculus, and stochastic integrals on PSITs can be applied to more general problems in mathematical finance.


%\begin{flushleft} \hrulefill \end{flushleft}

\section{Introduction}\label{section1} \noindent
\setcounter{equation}{0}
Stochastic integrals allow for some randomness in more realistic mathematical models, and play a significant role in studying a tremendous range of problems in finance, engineering, physics, and other fields. Consequently, developing different stochastic integrals is of much importance for providing general tools in stochastic calculus and solving general problems in practical applications.

Stochastic integrals have a long history, and one can refer to \cite{Jarrow,Kuo,Meyer2009} for details. Among various stochastic integrals developed in the existing literature, we just sketch the development of those relative to our study.
%%%%%%%%%%%%%%%%%%%%%%%%%%%%%%%%
In 1944, It\^{o} \cite{Ito} first constructed the stochastic integrals of adapted measurable processes with respect to (in short, w.r.t.) a Brownian Motion, and used them to develop a change of variables formula, i.e., the famous ``It\^{o}'s formula". The It\^{o} integrals permit stochastic processes as integrands, and there is a key character that the processes produced by such integration are still martingales (or, more generally, local martingales).
%%%%%%%%%%%%%%%%%%%%%%%%%%%%%%%%
In 1967, Kunita and Watanabe \cite{Kunita} defined the stochastic integrals of a class of adapted measurable processes w.r.t. square integrable martingales, and used them to develop a general change of variables formula through the idea of quadratic variation as a pseudo inner product.
%%%%%%%%%%%%%%%%%%%%%%%%%%%%%%%%
In 1970, Dol\'{e}ans-Dade and Meyer \cite{Doleans-dade} defined the stochastic integrals of locally bounded predictable processes w.r.t. local martingales or semimartingales, and in 1979, Jacod \cite{Jacod1979} defined the stochastic integrals of non-bounded predictable processes w.r.t. semimartingales. Such stochastic integrals relative to semimartingales are used to develop a unified theory of stochastic differential equations.
%%%%%%%%%%%%%%%%%%%%%%%%%%%%%%%%

% There are two main features in above development of stochastic integrals: stochastic processes as integrands or integrators, and applications in theory or practice.

Fundamental properties of stochastic integrals mainly grow out of stochastic processes as integrands or integrators, and more general stochastic processes as integrands or integrators can be used to define more general stochastic integrals with desirable properties. For example, It\^{o} integrals being martingales is a consequence of their integrators (i.e., Brownian motions which are martingales) and their adapted integrands, and the stochastic integrals in \cite{Kunita} profoundly extend It\^{o} integrals and their fundamental properties though a more general class of stochastic processes (i.e., square integrable martingales which include Brownian motions). In classic stochastic integration, stochastic processes as integrands or integrators are considered in the deterministic time interval $[0,+\infty[$ or $[0,T]$ (see the notations in Subsection 2.1) for a constant $T>0$. On the other hand, there are more general classes of stochastic processes studied in stochastic calculus. Such stochastic processes are only defined on some stochastic sets of interval type, while their values outside the stochastic sets of interval type do not matter (see \cite{He} or Definition \ref{processB}). More importantly, the stochastic process on a stochastic set of interval type shares the similar properties with its coupled sequence of classic stochastic processes, and plays a role in theoretical applications.
Jacod \cite{Jacod} defined a local martingale on a special PSIT by a sequence of stopped local martingales, and then used it to study semimartingales' characteristics and exponential formula. He et al. \cite{He} defined general classes of stochastic processes on stochastic sets of interval type (where PSITs and optional sets of interval type were mainly considered), and applied them to the study of Girsanov's theorems for local martingales and semimartingales.
Furthermore, the stochastic process on a stochastic set of interval type could also meet the need for some practical applications. When a problem is dealt with before a random time, the situations after the random time are not necessarily considered. Here are several scenarios:
\begin{itemize}
  \item [(1)] In a financial market, a risky asset with default is traded, and an agent invest in the asset. Assume that the default time is a random variable in the credit risk setting (see, e.g., \cite{Bielecki}). The asset price after the default time dose not matter, and it is sufficient for the agent to make a strategy strictly before the default time.
  \item [(2)] In the game theory (see, e.g., \cite{Dockner}), decision makers  (individuals, organizations, or governments) are designated as players. The final time of many intertemporal decision problems can be assumed to be a random variable (see, e.g., \cite{Marin-Solano}). The information after the final time dose not matter, and the players in the games just need to make decisions strictly before the final time.
  \item [(3)]  In the extraction of a non-renewable resource (for example, an oil field), the exact time when the resource is completely depleted is uncertain (see, e.g., \cite{Kostyunin}). The amount of the resource after the stochastic horizon does not matter, and the consumption of the resource is necessarily determined strictly before the stochastic horizon.
\end{itemize}
Generally, these scenarios are discussed in available literature by supplementing stochastic processes and integrals with extra information after random times (see, e.g. \cite{Blanchet,Kostyunin,Landriault,Lin,Lopez-Barrientos,Marin-Solano}), which imposes strict restriction on associated stochastic processes and integrals.
On the other hand, scenarios with deterministic horizons usually are allowed to take no account of the information after terminal times (see, e.g. \cite{Carassus,Karatzas,Oksendal,Yan}).
Therefore, instead of adding dispensable values to stochastic processes, utilizing stochastic processes on a stochastic set of interval type could be an alternative choice, and this also enable us to apply more general stochastic processes to investigate these scenarios in the following way: providing that classic stochastic processes can be used to solve the problem in a sequence of time intervals not exceeding the random time, stochastic processes on a stochastic set of interval type can be constructed to solve the problem. As far as we know, stochastic integrals relative to stochastic processes on stochastic sets of interval type do not exist in available researches, and problems on stochastic sets of interval type are not solved efficiently.

In this paper, we investigate stochastic integrals on a PSIT by using stochastic processes on the PSIT as integrands and integrators, and then develop a more general tool of the study on stochastic calculus and practical applications on the PSIT. More precisely, our study offers several key contributions to the literature, which are threefold as follows.

%process
To start with, we define several stochastic processes on a PSIT based on Definition 8.19 in \cite{He}, and then study their fundamental properties which play a crucial role in developing stochastic integrals on the PSIT.
Each of these stochastic processes on the PSIT can be characterized by a coupled sequence of stochastic processes, and degenerates into a stochastic process having analogous properties under some conditions.
For one thing, the jump process of a c\`{a}dl\`{a}g process (i.e., the process whose all paths are right-continuous with finite left-limits) on the PSIT can be well defined by introducing left-limit processes into more general stochastic processes on the PSIT. Just like the jump process in classic stochastic calculus (see, e.g. \cite{Cohen,He,Jacod}), the jump process on the PSIT is of much significance to study local martingales and semimartingales on the PSIT. For another, the quadratic covariation of two local martingales (resp. semimartingales) on the PSIT are defined, and its relationship with the classic quadratic covariation of two local martingales (resp. semimartingales) has been uncovered in our investigation. The quadratic covariation on the PSIT is the key process to consider stochastic integrals on the PSIT .

%stochastic integrals
Next, we define three kinds of stochastic integrals on a PSIT, and then investigate their fundamental properties and relations to classic stochastic integrals.
All these stochastic integrals on the PSIT are defined only on restricted stochastic subsets, and allow for more general stochastic processes as integrands and integrators, which generally extends classic stochastic integrals in \cite{He}. On the other hand, they can also degenerate into classic stochastic integrals in \cite{He} under some conditions, which guarantees that their fundamental properties are analogous to those of classic stochastic integrals.
Lebesgue-Stieltjes (in short, L-S) integrals by paths of measurable processes w.r.t. processes with finite variation are first extended into L-S integrals on the PSIT. Such an L-S integral on the PSIT can be regarded as the restriction of an L-S integral, but as a stochastic process on the PSIT, it can also be characterized by a coupled sequence of L-S integrals.
Then, stochastic integrals on the PSIT of predictable processes w.r.t. local martingales are defined through quadratic covariations and L-S integrals on the PSIT. The process produced by such integration is still a local martingale on the PSIT, and can be expressed as a summation of a coupled sequence of stochastic integrals of predictable processes w.r.t. local martingales.
Last, stochastic integrals on the PSIT of predictable processes w.r.t. semimartingales are discussed through above two kinds of stochastic integrals on the PSIT. The integrability of such stochastic integration can be equivalently verified by the decompositions of semimartingales on the PSIT, and such a stochastic integral on the PSIT can be characterized by a coupled sequence of stochastic integrals of predictable processes w.r.t. semimartingales. Furthermore, utilizing stochastic integrals on the PSIT of predictable processes w.r.t. semimartingales, we also obtain two well-known formulas in the stochastic calculus on the PSIT (i.e., It\^{o}'s formula for semimartingales on the PSIT, and integration by parts on the PSIT) which provide powerful tools for theoretical and practical applications.

%application
Finally, we apply stochastic integrals on a PSIT to mathematical finance, and develop a theory on financial markets on the PSIT. A new financial market is established by assuming that the time-horizon of the investor is uncertain but can be characterized by a PSIT, and consequently, the dynamic of the risky asset could be more generally chosen as a semimartingale on the PSIT, instead of a semimartingale. Then analogous with classic financial markets (see, e.g., \cite{Aksamit,Shiryaev}), we define self-financing and admissible strategies, no arbitrage, and portfolio problems in the financial market on the PSIT. To explain such a financial market, we present a simple example where a default may occur in the risky asset such that the time-horizon of the investor is a PSIT, and show its close relation to a coupled sequence of classic financial markets.

The rest of the paper is organized as follows. In the next section we define stochastic processes on PSITs, and then present their fundamental properties. In Section \ref{section3}, we discuss L-S integrals on PSITs and their fundamental properties. In Section \ref{section4}, based on local martingales and their quadratic covariations on PSITs, we investigate stochastic integrals on PSITs of predictable processes w.r.t. local martingales. In Section \ref{section5}, we study stochastic integrals on PSITs of predictable processes w.r.t. semimartingales, and present the It\^{o}'s formula for semimartingales on the PSIT. In Section \ref{section6}, stochastic integrals on PSITs are applied to mathematical finance, and some essentials of mathematical finance are constructed in financial markets on PSITs. Finally, a few concluding remarks are presented in Section \ref{section7}.

\section{Stochastic processes on PSITs}\label{section2}\noindent
\setcounter{equation}{0}
In this section, following the definitions of local martingales (see Definition II.2.46 in \cite{Jacod}) and semimartingales (see Definition 8.19 in \cite{He}) on PSITs, we define general processes on PSITs, and then study their fundamental properties.

\subsection{Basic notations and preliminaries}
Let $(\Omega,\mathcal{F},\mathbb{P})$ be a probability space and $\mathbb{F}:=(\mathcal{F}_t,t\geq 0)$ be a given filtration on that space satisfying the usual conditions. Unless otherwise stated, our starting point is always the filtered probability space $(\Omega,\mathcal{F},\mathbb{F},\mathbb{P})$. The following basic notations used in our paper are based on \cite{Jacod,He}.

Denote by $[a,b]$ the interval $\{x:a\leq x\leq b\}$ where $-\infty\leq a<b\leq {+\infty}$, and similarly for $[a,b[$, $]a,b]$ and $]a,b[$. Let $\mathbb{R}$ be the set of all real numbers, $\mathbb{R}^+:=[0,{+\infty}[$ be the set of all non-negative real numbers, and $\mathbb{N}^+:=\{1,2,\cdots\}$ be set of all positive integers. The union and intersection of sets $A$ and $B$ are denoted by $A\cup B$ and $A\cap B$ (or simply $AB$) respectively, and the complement of $A$ is denoted by $A^c$. The indicator function of the set $A$ is defined by
\begin{equation*}
I_A(\omega):=\left\{
\begin{aligned}
1,&\quad \omega\in A,\\
0,&\quad \omega\in A^c.
\end{aligned}
\right.
\end{equation*}
For the sake of simplicity, the set $\{\omega\in\Omega: P(\omega)\}$ (i.e., the set of all elements of $\Omega$ having the property $P$) is denoted by $[P]$, if there is no ambiguity.

For two stopping times $S$ and $T$, we write $T\wedge S:=\min\{T,S\}$, and define four kinds of stochastic intervals as follows:
\begin{align*}
\llbracket{S,T}\rrbracket&:=\left\{(\omega,t)\in \Omega\times \mathbb{R}^+: S(\omega)\leq t\leq T(\omega)\right\},\\
\llbracket{S,T}\llbracket&:=\left\{(\omega,t)\in \Omega\times \mathbb{R}^+: S(\omega)\leq t< T(\omega)\right\},\\
\rrbracket{S,T}\rrbracket&:=\left\{(\omega,t)\in \Omega\times \mathbb{R}^+: S(\omega)< t\leq T(\omega)\right\},\\
\rrbracket{S,T}\llbracket&:=\left\{(\omega,t)\in \Omega\times \mathbb{R}^+: S(\omega)< t< T(\omega)\right\}.
\end{align*}
specially, we write $\llbracket{T}\rrbracket:=\llbracket{T,T}\rrbracket$ (i.e., the graph of $T$). For a stopping time $T$ and a sequence $(T_n)_{n\in\mathbb{N}^+}$ (in short: $(T_n)$) of stopping times, the notation $T_n\uparrow T$ means that $(T_n)$ is an increasing sequence of stopping times satisfying $\lim\limits_{n\rightarrow+\infty}T_n=T$.

A stochastic process $(X_t)_{t\in \mathbb{R}^+}$ (or simply a process, i.e., a family of real random variables indexed by $\mathbb{R}^+$) is also denoted by $X$.
By convention, we set $X_{0-}=X_0$ for any c\`{a}dl\`{a}g process $X$. Two indistinguishable processes are regarded as the same: for two processes $X$ and $Y$, the relation $X=Y$ means that $X$ and $Y$ are indistinguishable. For an integer $n\in\mathbb{N}^+$, we stress that the notation $X^{(n)}$ denotes a process.
For two subsets $C$ and $\widetilde{C}$ of $\Omega\times \mathbb{R}^+$ and a map $X: \widetilde{C}\rightarrow \mathbb{R}$, the relation $C=\widetilde{C}$ means $I_C=I_{\widetilde{C}}$, and we would use the notation $XI_C$ for convenience where $XI_C$ is defined by
\begin{equation*}
(XI_C)(\omega,t):=\left\{
\begin{aligned}
&X(\omega,t),\quad&&(\omega,t)\in C\cap \widetilde{C},\\
&0,\quad&&\text{otherwise}.
\end{aligned}
\right.
\end{equation*}

If $X$ is a process and if $T$ is a stopping time, then we define the ``process stopped at time $T$" (see, e.g., (I.1.9) in \cite{Jacod}), denoted by $X^T=(X^T_t)_{t\in \mathbb{R}^+}$, by $X^T_t:=X_{T\wedge t}$, or equivalently,
\begin{equation}\label{def-XT}
X^{T}:=XI_{\llbracket{0,T}\rrbracket}+X_{T}I_{\rrbracket{T,{+\infty}}\llbracket}.
\end{equation}
For two stopping times $T$ and $S$ and two processes $X$ and $Y$, the following relation holds:
\begin{equation}\label{XYT1}
X^T=Y^T\Leftrightarrow X^TI_{\llbracket{0,T}\rrbracket}=Y^TI_{\llbracket{0,T}\rrbracket}\Leftrightarrow XI_{\llbracket{0,T}\rrbracket}=YI_{\llbracket{0,T}\rrbracket}.
\end{equation}
If $X$ is a c\`{a}dl\`{a}g process and if $T$ and $S$ are two stopping times, we define the stopped process $X^{T-}$ by
\begin{equation*}
X^{T-}:=XI_{\llbracket{0,T}\llbracket}+X_{T-}I_{\llbracket{T,{+\infty}}\llbracket},
\end{equation*}
and write $X^{S\wedge(T-)}:=(X^S)^{T-}$.
For two stopping times $T$ and $S$ and two c\`{a}dl\`{a}g processes $X$ and $Y$, it is not hard to obtain the relations
\begin{equation*}
X^{T-}=Y^{T-}\Leftrightarrow X^{T-}I_{\llbracket{0,T}\llbracket}=Y^{T-}I_{\llbracket{0,T}\llbracket}\Leftrightarrow XI_{\llbracket{0,T}\llbracket}=YI_{\llbracket{0,T}\llbracket}
\end{equation*}
and
\begin{equation}\label{XYT}
X^{S\wedge(T-)}=Y^{S\wedge(T-)}\Leftrightarrow X^{S\wedge(T-)}I_{\llbracket{0,T}\llbracket\llbracket{0,S}\rrbracket}
=Y^{S\wedge(T-)}I_{\llbracket{0,T}\llbracket\llbracket{0,S}\rrbracket}\Leftrightarrow XI_{\llbracket{0,T}\llbracket\llbracket{0,S}\rrbracket}=YI_{\llbracket{0,T}\llbracket\llbracket{0,S}\rrbracket}.
\end{equation}

Let $\mathcal{D}$ be a class of processes. We denote by $\mathcal{D}_0$ the sub-class of $\mathcal{D}$ consisting all processes of $\mathcal{D}$ with null initial values. From Definition 7.1 in \cite{He}, the localized class of $\mathcal{D}$, denoted by $\mathcal{D}_{\mathrm{loc}}$, is the collection of all processes $X$ satisfying the follows: $X_0$ is $\mathcal{F}_0$-measurable and there exists a sequence $(T_n)$ of stopping times with $T_n\uparrow {+\infty}$ such that for each $n\in \mathbb{N}^+$ the stopped process $X^{T_n}-X_0\in\mathcal{D}$. The sequence $(T_n)$ is called a localizing sequence for $X$ (w.r.t. $\mathcal{D}$).
$\mathcal{D}$ is said to be stable under stopping if $X\in \mathcal{D}$ implies $X^T\in\mathcal{D}$ for any stopping time $T$, and $\mathcal{D}$ is said to be stable under localization if $\mathcal{D}=\mathcal{D}_{\mathrm{loc}}$.

Throughout this paper, we use the following notations:
\begin{itemize}
  \item  $\mathfrak{M}$ (resp. $\mathcal{P}$, resp. $\mathcal{R}$) --- the class of all measurable (resp. predictable, resp. c\`{a}dl\`{a}g) processes;
  \item  $\mathfrak{V}$ (resp. $\mathcal{V}$, resp. $\mathcal{A}$) --- the class of all (resp. adapted, resp. adapted integrable) processes with finite variation;
  \item  $\mathcal{V}^+$ (resp. $\mathcal{A}^+$) --- the class of all adapted (resp. adapted integrable) increasing processes;
  \item  $\mathcal{M}_{\mathrm{loc}}$ (resp. $\mathcal{M}^c_{\mathrm{loc}}$, resp. $\mathcal{M}^d_{\mathrm{loc}}$) --- the class of all (resp. continuous, resp. purely discontinuous) local martingales;
  \item  $\mathcal{M}^2_{\mathrm{loc}}$ (resp. $\mathcal{M}^{2,c}_{\mathrm{loc}}$) --- the class of all (resp. continuous) locally square integrable martingales;
  \item  $\mathcal{S}$ --- the class of all semimartingales.
\end{itemize}
Note that $\mathcal{M}^d_{\mathrm{loc}}=\mathcal{M}^d_{\mathrm{loc},0}$, $\mathcal{M}^{c}_{\mathrm{loc}}=\mathcal{M}^{2,c}_{\mathrm{loc}}$ and $\mathfrak{V}\subseteq \mathcal{R}$. We stress that all the elements of $\mathcal{S}$ are supposed to be c\`{a}dl\`{a}g.
The following lemma presents above classes' stability under stopping and localization, which plays an important role in stochastic processes on PSITs.

\begin{lemma}\label{stable}
The following classes are stable under stopping and localization: $\mathfrak{M}_0$, $\mathcal{P}$, $\mathcal{R}_0$, $\mathfrak{V}_0$, $\mathcal{V}$, $\mathcal{A}_{\mathrm{loc}}$, $\mathcal{V}^+$, $\mathcal{A}^+_{\mathrm{loc}}$, $\mathcal{M}_{\mathrm{loc}}$, $\mathcal{M}^c_{\mathrm{loc}}$, $\mathcal{M}^d_{\mathrm{loc}}$, $\mathcal{M}^2_{\mathrm{loc}}$, and $\mathcal{S}$, where $\mathcal{A}^+_{\mathrm{loc}}:=(\mathcal{A}^+)_{\mathrm{loc}}$.
\end{lemma}
\begin{proof}
We first prove the stability under stopping. By studying all paths of the right-hand process of \eqref{def-XT}, it is easy to see that the classes $\mathcal{R}_0$ and $\mathfrak{V}_0$ are stable under stopping. The stability under stopping of the classes $\mathcal{P}$, $\mathcal{V}$, $\mathcal{A}_{\mathrm{loc}}$, $\mathcal{V}^+$, $\mathcal{A}^+_{\mathrm{loc}}$, and $\mathcal{S}$ has been shown by Lemma I.1.35, Proposition I.2.4, the remark after I.3.8, and the remark after Definition I.4.22 in \cite{Jacod}.
The statement that $\mathcal{M}_{\mathrm{loc}}$, $\mathcal{M}^c_{\mathrm{loc}}$, and $\mathcal{M}^d_{\mathrm{loc}}$ are stable under stopping is indicated by Theorem 7.25 in \cite{He}, and the statement that $\mathcal{M}^2_{\mathrm{loc}}$ is stable under stopping is shown by Definition 7.11 in \cite{He}.
It remains to prove that $\mathfrak{M}_0$ is stable under stopping. Let $T$ be a stopping time, and $X\in\mathfrak{M}_0$. By the assumption of stopping time, the mapping $(\omega,t)\mapsto(\omega,T(\omega)\wedge t)$ of $\Omega\times \mathbb{R}^+$ into itself is $\mathcal{F}\otimes \mathcal{B}(\mathbb{R}^+)$-measurable, and by the assumption of measurability, the mapping
\[
(\omega,t)\mapsto X(\omega,t): (\Omega\times \mathbb{R}^+, \mathcal{F}\otimes \mathcal{B}(\mathbb{R}^+))
\rightarrow (\mathbb{R},\mathcal{B}(\mathbb{R}))
\]
is measurable, where $\mathcal{B}(\mathbb{R}^+)$ and $\mathcal{B}(\mathbb{R})$ are the Borel $\sigma$-fields in $\mathbb{R}^+$ and $\mathbb{R}$ respectively, and $\mathcal{F}\otimes \mathcal{B}(\mathbb{R}^+)$ is the product $\sigma$-field formed from the $\sigma$-fields $\mathcal{F}$ and $\mathcal{B}(\mathbb{R}^+)$. Then the composite mapping
\[
(\omega,t)\mapsto X(\omega,T(\omega)\wedge t): (\Omega\times \mathbb{R}^+, \mathcal{F}\otimes \mathcal{B}(\mathbb{R}^+))
\rightarrow (\mathbb{R},\mathcal{B}(\mathbb{R}))
\]
is also measurable, which proves $X^T$ is a measurable process, i.e., $X^T\in\mathfrak{M}_0$.

Next, we prove the stability under localization. The stability under localization of the classes $\mathcal{V}$, $\mathcal{V}^+$, and $\mathcal{S}$ has been shown by the remark after I.3.8 and Proposition I.4.25 in \cite{Jacod}. In reality, $\mathcal{A}_{\mathrm{loc}}$, $\mathcal{A}^+_{\mathrm{loc}}$, $\mathcal{M}_{\mathrm{loc}}$, $\mathcal{M}^c_{\mathrm{loc}}$, $\mathcal{M}^d_{\mathrm{loc}}$, and $\mathcal{M}^2_{\mathrm{loc}}$ are, respectively, the localized classes of the following classes which are stable under stopping: $\mathcal{A}$, $\mathcal{A}^+$, $\mathcal{M}$ (the class of all uniformly integrable martingales), $\mathcal{M}^c$ (the class of all continuous uniformly integrable martingales), $\mathcal{M}^d$ (the class of all uniformly integrable martingales with null initial values which are orthogonal to all continuous local martingales), and $\mathcal{M}^2$ (the class of all square integrable martingales). Then from Lemma I.1.35 in \cite{Jacod}, these classes $\mathcal{A}_{\mathrm{loc}}$, $\mathcal{A}^+_{\mathrm{loc}}$, $\mathcal{M}_{\mathrm{loc}}$, $\mathcal{M}^c_{\mathrm{loc}}$, $\mathcal{M}^d_{\mathrm{loc}}$, and $\mathcal{M}^2_{\mathrm{loc}}$ are stable under localization. Let $\mathcal{D}\in\{\mathfrak{M}_0, \mathcal{P}, \mathcal{R}_0, \mathfrak{V}_0\}$ and $X\in\mathcal{D}_{\mathrm{loc}}$. Since $\mathcal{D}\subseteq\mathcal{D}_{\mathrm{loc}}$ and $X_0$ is $\mathcal{F}_0$-measurable, we just need to prove $X\in\mathcal{D}$ to obtain $\mathcal{D}=\mathcal{D}_{\mathrm{loc}}$. Assume that $(T_n)$ is a localizing sequence for $X$ satisfying $X^{T_n}-X_0\in\mathcal{D}$ for each $n\in \mathbb{N}^+$. Then it is easy to see
\begin{equation}\label{X-loc}
      X=X_0+\sum\limits_{n=1}^{{+\infty}}(X^{T_n}-X_0)I_{\rrbracket{T_{n-1},T_n}
      \rrbracket},\quad T_0=0.
\end{equation}
Since for each $n\in \mathbb{N}^+$, $I_{\rrbracket{T_{n-1},T_n}\rrbracket}$ is a predictable process as well as a measurable process, the relation $X\in\mathcal{D}$ holds in the cases of $\mathcal{D}\in\{\mathfrak{M}_0, \mathcal{P}\}$. As for $\mathcal{D}=\mathcal{R}_0$ (resp.  $\mathcal{D}=\mathfrak{V}_0$), it is easy to verify that all paths of the right-hand process of \eqref{X-loc} are right-continuous with finite left-limits (resp. are right-continuous with finite left-limits and has a finite variation over any finite interval), because $XI_{\llbracket{0,T_n}\rrbracket}=(X^{T_n}-X_0)I_{\llbracket{0,T_n}\rrbracket}$ holds for each $n\in \mathbb{N}^+$. Then  the relation $X\in\mathcal{D}$ holds in the cases of $\mathcal{D}\in\{\mathcal{R}_0, \mathfrak{V}_0\}$.
\end{proof}

Finally, we present fundamental properties of stochastic integrals which are used in our paper.
\begin{lemma}\label{property}
Let $H,K\in \mathcal{P}$, and $X,Y\in \mathcal{S}$, and $\tau$ be a stopping time, and $a,b\in\mathbb{R}$. Suppose that both $H$ and $K$ are $X$-integrable (see Definition 9.13 in \cite{He}, or \eqref{0-HX}), and that $H$ is $Y$-integrable. At this time, we can define stochastic integrals $H.X$, $H.Y$ and $K.X$, where
\[
(H.X)_t=\int_{[0,t]}H_sdX_s=\int_{0}^tH_sdX_s+H_0X_0,\quad t\in\mathbb{R}^+,
\]
and similar for $H.Y$ and $K.X$. Then we have the following statements:
\begin{itemize}
  \item [$(1)$] $aH+bK$ is $X$-integrable satisfying $(aH+bK).X=a(H.X)+b(K.X)$.
  \item [$(2)$] $H$ is $aX+bY$-integrable satisfying $H.(aX+bY)=a(H.X)+b(H.Y)$.
  \item [$(3)$] Let $\widetilde{H}\in \mathcal{P}$. Then $\widetilde{H}$ is $H.X$-integrable if and only if $\widetilde{H}H$ is $X$-integrable, and in either case, $\widetilde{H}.(H.X)=(\widetilde{H}H).X$.
  \item [$(4)$] $\Delta(H.X)=H\Delta X$ and $(H.X)_0=H_0X_0$.
  \item [$(5)$] $(H.X)^\tau=H.X^\tau=(HI_{\llbracket{0,\tau}\rrbracket}).X.$
  \item [$(6)$] $H^\tau$ is $X^\tau$-integrable satisfying
  \begin{equation}\label{HT}
  (H.X)^\tau=H^\tau.X^\tau.
  \end{equation}
\end{itemize}
Furthermore, analogous properties are valid for stochastic integrals of measurable processes w.r.t. processes with finite variation (see Definition 3.45 in \cite{He}, or \eqref{def-c-HA}), and stochastic integrals of predictable processes w.r.t. local martingales (see Definition 9.1 in \cite{He}, or \eqref{hm}).
\end{lemma}
\begin{proof}
The proofs of statements $(1)-(5)$ can be found in Theorems 9.15 and 9.18 in \cite{He}, and we just prove $(6)$.

We first show that $H_{\tau}I_{\rrbracket{\tau,+\infty}\llbracket}$ is a locally bounded predictable process. Theorem 3.16 and Corollary 3.23 in \cite{He} show $H_{\tau}I_{\rrbracket{\tau,+\infty}\llbracket}
=(H_{\tau}I_{[\tau<+\infty]})I_{\rrbracket{\tau,+\infty}\llbracket}\in \mathcal{P}$. Put $T_n=nI_{[H_{\tau}I_{[\tau<+\infty]}\leq n]}$ for each $n\in \mathbb{N}^+$,
and then $(T_n)$ is a sequence of stopping times satisfying $T_n\uparrow +\infty$. From the relations
\[
(H_{\tau}I_{\rrbracket{\tau,+\infty}\llbracket})^{T_n}\leq n, \quad n\in \mathbb{N}^+,
\]
we deduce that $H_{\tau}I_{\rrbracket{\tau,+\infty}\llbracket}$ is a locally bounded predictable process.

Then we prove the statement $(6)$. From the statement $(5)$, it is easy to see
\[
(H.X)^\tau=((HI_{\llbracket{0,\tau}\rrbracket}).X)^{\tau}=(HI_{\llbracket{0,\tau}\rrbracket}).X^{\tau},
\]
which implies that $HI_{\llbracket{0,\tau}\rrbracket}$ is $X^\tau$-integrable. As a locally bounded predictable process, $H_{\tau}I_{\rrbracket{\tau,+\infty}\llbracket}$ is $X^\tau$-integrable (see Theorem I.4.31 in \cite{Jacod}). Consequently, the statement $(1)$ shows $H^\tau=HI_{\llbracket{0,\tau}\rrbracket}+H_{\tau}I_{\rrbracket{\tau,+\infty}\llbracket}$ is $X^\tau$-integrable. And from the statements $(1)$ and $(5)$, the relation
\begin{align*}
H^\tau.X^\tau
&=(HI_{\llbracket{0,\tau}\rrbracket}).X^\tau+(H_{\tau}I_{\rrbracket{\tau,+\infty}\llbracket}).X^\tau\\
&=(H.X)^\tau+(H_{\tau}I_{\rrbracket{\tau,+\infty}\llbracket}I_{\llbracket{0,\tau}\rrbracket}).X\\
&=(H.X)^\tau
\end{align*}
yields \eqref{HT}.
\end{proof}

\subsection{General stochastic processes on PSITs}
We first recall the definition of a PSIT (see Definition 8.16 in \cite{He}) which is the foundation of stochastic processes and stochastic integrals on PSITs.

\begin{definition}\label{setofinterval}
A set $B\subseteq \Omega\times \mathbb{R}^+$ is called a set of interval type if there is a non-negative random variable $T$ such that for each $\omega\in \Omega$ the section $B_\omega=\{t: (\omega,t)\in B\}$ is $[0,T(\omega)[$ or $[0,T(\omega)]$ and $B_\omega\neq\emptyset$. If $B$ is also a predictable set, then it is called a predictable set of interval type.
\end{definition}

If a set $B$ of interval type is also an optional set, then $B$ is called an optional set of interval type. However, optional sets of interval type are not used in our main theory. Furthermore, the following lemma characterizes a PSIT in the form of stochastic intervals.

\begin{lemma}\label{th8.18}
The following statements are equivalent:
\begin{itemize}
  \item [$(1)$] $B$ is a PSIT.
  \item [$(2)$] $I_B=I_FI_{\llbracket{0,T}\llbracket}+I_{F^c}I_{\llbracket{0,T}\rrbracket}$, or equivalently,
      \begin{equation}\label{B}
         B=\llbracket{0,T_F}\llbracket\;\cap\;\llbracket{0,T_{F^c}}\rrbracket,
      \end{equation}
      where $T$ is a stopping time and the debut of $B^c$,
      and $F\in \mathcal{F}_{T-}$, and $T_F=TI_{F}+(+\infty)I_{F^c}>0$ is a predictable stopping time.
  \item [$(3)$] $B=\bigcup\limits_{n=1}^{{+\infty}} \llbracket{0,\tau_n}\rrbracket$, where $(\tau_n)_{n\in\mathbb{N}^+}$ is an increasing sequence of stopping times.
\end{itemize}
\end{lemma}
\begin{proof}
The proof can be found in Theorem 8.18 in \cite{He}.
\end{proof}

The sequence $(\tau_n)$ in Lemma \ref{th8.18} is called a fundamental sequence (in short: FS) for $B$.
Based on Definition \ref{setofinterval}, we can define stochastic processes on PSITs.
\begin{definition}\label{de-processB}
Let $B$ and $\widetilde{B}$ be two PSITs with $B\subseteq \widetilde{B}$.
\begin{itemize}
  \item [$(1)$] Suppose $X$ is a real function defined on $B$. $X$ is called a stochastic process on $B$ (or simply, a process on $B$) if $XI_{B}$ is a process.
  \item [$(2)$] Suppose $X$ is a process on $\widetilde{B}$. Then the restriction of $X$ on $B$, denoted by $X\mathfrak{I}_B$, is defined as follows: $X\mathfrak{I}_B$ is a process on $B$ satisfying $(X\mathfrak{I}_B)I_B=XI_B$.
\end{itemize}
\end{definition}

In this paper, we focus on PSITs to study stochastic processes and stochastic integrals. Thus, in the rest of the paper, we always use the set $B$ to denote a PSIT. Definition \ref{de-processB} provides a practical method for obtaining a process on $B$ from any process: for a process $X$, $X\mathfrak{I}_B$ is always a process on $B$. Using such a method, we can define two usual relations between two processes on $B$.

\begin{definition}\label{XIB}
Let $X$ and $Y$ be two processes on $B$.
\begin{itemize}
  \item [$(1)$] $Y$ is said to be a modification on $B$ of $X$ if $XI_B$ is a modification of $YI_B$.
  \item [$(2)$] $X$ and $Y$ are said to be indistinguishable on $B$ if $XI_B$ and $YI_B$ are indistinguishable.
\end{itemize}
\end{definition}
As usual, two indistinguishable processes on $B$ are regarded as the same, and then we have the relation $X=Y\Leftrightarrow XI_B=YI_B$ for two processes $X$ and $Y$ on $B$.

Following Definition 8.19 in \cite{He}, we define more general processes on $B$ having particular properties, which shows the essential difference between processes on $B$ and processes.

\begin{definition}\label{processB}
Let $X$ be a process on $B$, $T$ be the debut of $B^c$, and the class $\mathcal{D}$ be a class of processes having the property $\mathcal{P}$.
\begin{itemize}
  \item [$(1)$] If there exists an increasing sequence $(T_n)$ of stopping times and a sequence $(X^{(n)})_{n\in\mathbb{N}^+}$ (in short: $(X^{(n)})$) of processes such that $T_n\uparrow T$,
$\bigcup\limits_{n=1}^{{+\infty}}\llbracket{0,T_n}\rrbracket\supseteq B$ and for each $n\in \mathbb{N}^+$,
\[
 (XI_B)^{T_n}=(X^{(n)}I_B)^{T_n}\quad (\text{or equivalently}, \; XI_{B\llbracket{0,T_n}\rrbracket}=X^{(n)}I_{B\llbracket{0,T_n}\rrbracket}),
\]
then $(T_n,X^{(n)})_{n\in\mathbb{N}^+}$ (in short: $(T_n,X^{(n)})$) is called a coupled sequence (in short: CS) for $X$.
  \item [$(2)$] If there exists a CS $(T_n,X^{(n)})$ for $X$ satisfying $X^{(n)}\in\mathcal{D}$ for each $n\in \mathbb{N}^+$,
then $X$ is called a process on $B$ having the property $\mathcal{P}$, and at this time, $(T_n,X^{(n)})$ is called a fundamental coupled sequence (in short: FCS) in $\mathcal{D}$ for $X$. The collection of all processes on $B$ having the property $\mathcal{P}$ is denoted by $\mathcal{D}^B$.
\end{itemize}
\end{definition}

With different choices of $\mathcal{D}$, we can obtain, in the manner of Definition \ref{processB}, the following classes of processes on $B$ which are used in our paper: $\mathfrak{M}^B$, $\mathcal{P}^B$, $\mathcal{R}^B$, $\mathfrak{V}^B$, $\mathcal{V}^B$, $(\mathcal{A}_{\mathrm{loc}})^B$, $(\mathcal{V}^+)^B$, $(\mathcal{A}^+_{\mathrm{loc}})^B$, $(\mathcal{M}_{\mathrm{loc}})^B$, $(\mathcal{M}^c_{\mathrm{loc}})^B$, $(\mathcal{M}^d_{\mathrm{loc}})^B$, $(\mathcal{M}^2_{\mathrm{loc}})^B$, and $\mathcal{S}^B$.

\begin{remark}\label{c-process}
In general, Definition \ref{processB} extends the classic definition of processes.  Let ${B}=\llbracket{0,+\infty}\llbracket=\Omega\times\mathbb{R}^+$, and $\mathcal{D}$ be a class of processes. It is obvious that $X\in \mathcal{D}$ implies $X \in \mathcal{D}^{\llbracket{0,+\infty}\llbracket}$ (because $(T_n=+\infty,X^{(n)}=X)$ is always an FCS for $X \in \mathcal{D}^{\llbracket{0,+\infty}\llbracket}$), but $X\in \mathcal{D}^{\llbracket{0,+\infty}\llbracket}$ may not imply $X \in \mathcal{D}$. We give the following two examples:
\begin{itemize}
  \item [$(1)$] Let $\mathcal{D}$ be the class of all bounded process, and put $X(\omega,t)=t,\;(\omega,t)\in \llbracket{0,+\infty}\llbracket$. Then $X\in\mathcal{D}^{\llbracket{0,+\infty}\llbracket}$ because $(T_n=n,X^{T_n})$ is a CS for $X$ and $X^{T_n}$ is a bounded process for each $n\in \mathbb{N}^+$, but $X\notin \mathcal{D}$.
  \item [$(2)$] Let $\mathcal{D}=\mathcal{M}$ be the class of all uniformly integrable martingales, and $W$ be a standard Brownian motion (see, e.g., \cite{Karatzas-Shreve}). Then it is well-known that $W$ is a martingale but not a uniformly integrable martingale, i.e., $X\notin \mathcal{M}$ (The latter statement can be easily proved by the optional stopping theorem, e.g., Theorem II.3.2 in \cite{Revuz}). On the other hand, $W\in\mathcal{M}^{\llbracket{0,+\infty}\llbracket}$ holds true because $(T_n=n,W^{T_n})$ is a CS for $X$ and $W^{T_n}$ is a uniformly integrable martingale for each $n\in \mathbb{N}^+$ (see, e.g., \cite{Revuz}).
\end{itemize}
Therefore, processes on $\Omega\times\mathbb{R}^+$ in Definition \ref{processB} could be different from those in the classic definition, but fortunately, in most cases, the relation $\mathcal{D}=\mathcal{D}^{\Omega\times\mathbb{R}^+}$ holds (see Corollary \ref{cD=DB}), especially in our study of stochastic integrals on PSITs.
\end{remark}

A stopping time $T$ is called a stopping time on $B$ if $\llbracket{0,T}\rrbracket\subseteq B$. Obviously, if $(\tau_n)$ is an FS for $B$, then for each $n\in \mathbb{N}^+$, $\tau_n$ is a stopping time on $B$.
Analogous to the stopped process defined by \eqref{def-XT},  we would use the following stopped process $X^T$ defined by
\[
X^T:=XI_{\llbracket{0,T}\rrbracket}+X_TI_{\rrbracket{T,{+\infty}}\llbracket},
\]
and it is easy to check that $(X^T)^S=X^{T\wedge S}=(X^S)^T$, where $X$ is a process on $B$, and $T$ and $S$ are two stopping times on $B$. The following two theorems present the importance of such a stopped process.

\begin{theorem}\label{fcs}
Let $S$ be a stopping time on $B$, $\mathcal{D}$ be a class of processes, and $X\in \mathcal{D}^B$ with an FCS $(T_n,X^{(n)})$. If the class $\mathcal{D}$ is stable under stopping and localization, then $X^S\in \mathcal{D}$, and $(T_n,(X^{(n)})^S)$ is an FCS for $X^S\mathfrak{I}_B\in \mathcal{D}^B$.
\end{theorem}
\begin{proof}
The proof of $X^S\in \mathcal{D}$ can be found in Theorem 8.20 of \cite{He}, and it suffices to prove that $(T_n,(X^{(n)})^S)$ is an FCS for $X^S\mathfrak{I}_B\in \mathcal{D}^B$. For each $n\in \mathbb{N}^+$, noticing that $S\wedge T_n$ is a stopping time on $B$, we have
\[
X^{S\wedge T_n}I_{\llbracket{0,S\wedge T_n}\rrbracket}
=XI_{\llbracket{0,S\wedge T_n}\rrbracket}
=(XI_{B\llbracket{0,T_n}\rrbracket})I_{\llbracket{0,S\wedge T_n}\rrbracket}
=(X^{(n)}I_{B\llbracket{0,T_n}\rrbracket})I_{\llbracket{0,S\wedge T_n}\rrbracket}
=X^{(n)}I_{\llbracket{0,S\wedge T_n}\rrbracket},
\]
which, by \eqref{XYT1}, implies that $X^{S\wedge T_n}=(X^{(n)})^{S\wedge T_n}$. Then the relations
\[
(X^S\mathfrak{I}_B)I_{B\llbracket{0,T_n}\rrbracket}
=X^{S\wedge T_n}I_{B\llbracket{0,T_n}\rrbracket}
=(X^{(n)})^{S\wedge T_n}I_{B\llbracket{0,T_n}\rrbracket}
=(X^{(n)})^{S}I_{B\llbracket{0,T_n}\rrbracket}, \quad n\in \mathbb{N}^+
\]
show that $(T_n,(X^{(n)})^S)$ is a CS for $X^S\mathfrak{I}_B$. Since $\mathcal{D}$ is stable under stopping such that $(X^{(n)})^S\in \mathcal{D}$ for each $n\in \mathbb{N}^+$, the sequence $(T_n,(X^{(n)})^S)$ is indeed an FCS for $X^S\mathfrak{I}_B\in \mathcal{D}^B$.
\end{proof}

\begin{theorem}\label{fcs-p}
Let $(\tau_n)$ be an FS for $B$, $\mathcal{D}$ be a class of processes, and $X\in \mathcal{D}^B$.
\begin{enumerate}
  \item [$(1)$] $(\tau_n,X^{\tau_n})$ is a CS for $X$.
  \item [$(2)$] If the class $\mathcal{D}$ is stable under localization and stopping, then $(\tau_n,X^{\tau_n})$ is an FCS for $X\in\mathcal{D}^B$.
\end{enumerate}
\end{theorem}
\begin{proof}
The proof of $(1)$ is trivial, and we just prove $(2)$.
From Theorem \ref{th8.18}, $B=\bigcup\limits_{n=1}^{{+\infty}} \llbracket{0,\tau_n}\rrbracket$. Then for each $n\in \mathbb{N}^+$, from the definition of $X^{\tau_n}$, the relation
\[
XI_{B\llbracket{0,\tau_n}\rrbracket}=XI_{\llbracket{0,\tau_n}\rrbracket}
=X^{\tau_n}I_{\llbracket{0,\tau_n}\rrbracket}
=X^{\tau_n}I_{B\llbracket{0,\tau_n}\rrbracket}
\]
shows that $(\tau_n,X^{\tau_n})$ is a CS for $X$. Since Theorem \ref{fcs} shows $X^{\tau_n}\in\mathcal{D}$ for each $n\in \mathbb{N}^+$, the sequence $(\tau_n,X^{\tau_n})$ is an FCS for $X\in\mathcal{D}^B$.
\end{proof}

\begin{corollary}\label{cD=DB}
Let $\mathcal{D}$ be a class of processes, and ${B}=\llbracket{0,+\infty}\llbracket=\Omega\times\mathbb{R}^+$. If the class $\mathcal{D}$ is stable under stopping and localization, then $\mathcal{D}^{\llbracket{0,+\infty}\llbracket}=\mathcal{D}$.
\end{corollary}
\begin{proof}
The inclusion $\mathcal{D}\subseteq\mathcal{D}^{\llbracket{0,+\infty}\llbracket}$ has been shown in Remark \ref{c-process}, and it suffices to prove $\mathcal{D}^{\llbracket{0,+\infty}\llbracket}\subseteq\mathcal{D}$. Let $X\in \mathcal{D}^{\llbracket{0,+\infty}\llbracket}$. In fact, put $\tau=+\infty$, and then $\tau$ is a stopping time on $\llbracket{0,+\infty}\llbracket$. Theorem \ref{fcs} shows $X=X^\tau\in\mathcal{D}$ which finishes the proof.
\end{proof}

Theorem \ref{fcs-p} provides a practical method for obtaining FCSs for processes on $B$, and such a method can be generally applied to classes studied in Lemma \ref{stable}. For example, $(\tau_n,X^{\tau_n})$ is an FCS for $X\in(\mathfrak{M}_0)^B$, where $(\tau_n)$ is an FS for $B$. Furthermore, FCSs in Theorem \ref{fcs-p} can be used to characterize processes on $B$, which is presented in the following theorem.

\begin{theorem}\label{pro}
Let $X$ be a process on $B$, and $\mathcal{D}$ be a class of processes.
Suppose $\mathcal{D}$ is stable under stopping and localization. Then $X\in \mathcal{D}^B$ if and only if there exists an FS $(\tau_n)$ for $B$ satisfying $X^{\tau_n}\in \mathcal{D}$ for each $n\in \mathbb{N}^+$.
\end{theorem}
\begin{proof}
The necessity has been shown in Theorem \ref{fcs-p}. Suppose $(\tau_n)$ is an FS for $B$ satisfying $X^{\tau_n}\in \mathcal{D}$ for each $n\in \mathbb{N}^+$. From the relations
\[
XI_{B\llbracket{0,\tau_n}\rrbracket}=X^{\tau_n}I_{B\llbracket{0,\tau_n}\rrbracket},\quad n\in \mathbb{N}^+,
\]
$(\tau_n,X^{\tau_n})$ is a CS for $X$, and from Theorem \ref{fcs}, $X^{\tau_n}\in\mathcal{D}$ for each $n\in \mathbb{N}^+$. Therefore, we obtain $X\in \mathcal{D}^B$, which proves the sufficiency.
\end{proof}

The fundamental properties of processes on $B$ are summarized in the following two theorems: the former is based on general FCSs for processes on $B$, and the later focuses on FSs for $B$.
\begin{theorem}\label{process}
Let $\mathcal{D}$ be a class of processes, $T$ be the debut of $B^c$, and $X,Y\in \mathcal{D}^B$. Suppose that $(T_n,X^{(n)})$ is an FCS for $X\in \mathcal{D}^B$  (resp. a CS for $X$), and that $(S_n)$ is an increasing sequence of stopping times with $S_n\uparrow T$ and $\bigcup\limits_{n=1}^{{+\infty}}\llbracket{0,S_n}\rrbracket\supseteq B$.
\begin{itemize}
  \item[$(1)$] $X=Y$ if and only if $XI_{B\llbracket{0,S_n}\rrbracket}=YI_{B\llbracket{0,S_n}\rrbracket}$ for each $n\in \mathbb{N}^+$.
  \item[$(2)$] $X=X^{(k)}=X^{(l)}$ on $B\llbracket{0,T_k}\rrbracket$ for any $k,\;l\in \mathbb{N}^+$ with $k\leq l$, i.e.,
      \begin{equation}\label{xkl}
      XI_{B\llbracket{0,T_k}\rrbracket}=X^{(k)}I_{B\llbracket{0,T_k}\rrbracket}
      =X^{(l)}I_{B\llbracket{0,T_k}\rrbracket}.
      \end{equation}
      Specially, $X^{(k)}I_{\llbracket{0}\rrbracket}=XI_{\llbracket{0}\rrbracket}$.
  \item[$(3)$] $(\tau_n,X^{(n)})$ is an FCS for $X\in \mathcal{D}^B$ (resp. a CS for $X$), where $\tau_n=T_n\wedge S_n$ for each $n\in \mathbb{N}^+$.
  \item[$(4)$] Suppose that $\mathcal{D}$ satisfies the following linearity: $aU+bV\in \mathcal{D}$ holds for all $U,V\in \mathcal{D}$ and all $a,b\in \mathbb{R}$.
      Then $aX+bY\in \mathcal{D}^B$ holds for all $a,b\in \mathbb{R}$.
  \item[$(5)$] $X$ can be expressed as
      \begin{equation}\label{x-expression}
      X=\left(X_0I_{\llbracket{0}\rrbracket}+\sum\limits_{n=1}^{{+\infty}}X^{(n)}I_{\rrbracket{T_{n-1},T_n}
      \rrbracket}\right)\mathfrak{I}_B,\quad T_0=0.
      \end{equation}
      Furthermore, if $(S_n,\widetilde{X}^{(n)})$ is also an FCS for $X\in \mathcal{D}^B$ (resp. a CS for $X$), then $X=\widetilde{X}$ where the process $\widetilde{X}$ given by
      \begin{equation*}
      \widetilde{X}=\left(X_0I_{\llbracket{0}\rrbracket}+\sum\limits_{n=1}^{{+\infty}}\widetilde{X}^{(n)}I_{\rrbracket{S_{n-1},S_n}
      \rrbracket}\right)\mathfrak{I}_B,\quad S_0=0.
      \end{equation*}
      In this case, we say the expression of \eqref{x-expression} is independent of the choice of the FCS $(T_n,X^{(n)})$ for $X\in \mathcal{D}^B$ (resp. the CS $(T_n,X^{(n)})$ for $X$).
      \end{itemize}
\end{theorem}
\begin{proof}
We just prove the case of FCS, and the case of CS can be proved similarly.

(1) The necessity is trivial, and we need to prove the sufficiency. Suppose $XI_{B\llbracket{0,S_n}\rrbracket}=YI_{B\llbracket{0,S_n}\rrbracket}$ for each $n\in \mathbb{N}^+$. It is easy to obtain $X_0I_{\llbracket{0}\rrbracket}=Y_0I_{\llbracket{0}\rrbracket}$ and
\[
 XI_{B\rrbracket{S_{n-1},S_n}\rrbracket}=YI_{B\rrbracket{S_{n-1},S_n}\rrbracket},
\quad S_0=0,\quad n\in \mathbb{N}^+.
\]
Then by noticing
\[
B\subseteq\bigcup\limits_{n=1}^{{+\infty}}\llbracket{0,S_n}\rrbracket = \llbracket{0}\rrbracket\cup \left(\bigcup\limits_{n=1}^{{+\infty}}\rrbracket{S_{n-1},S_n}\rrbracket\right), \]
we deduce
\[
XI_B=X_0I_{\llbracket{0}\rrbracket}+\sum\limits_{n=1}^{{+\infty}}XI_{B\rrbracket{S_{n-1},S_n}
      \rrbracket}
 =Y_0I_{\llbracket{0}\rrbracket}+\sum\limits_{n=1}^{{+\infty}}YI_{B\rrbracket{S_{n-1},S_n}
      \rrbracket}
 =YI_B,
\]
which, by Definition \ref{XIB}, implies that $X=Y$.

(2) $XI_{B\llbracket{0,T_k}\rrbracket}=X^{(k)}I_{B\llbracket{0,T_k}\rrbracket}$ is a direct result of Definition \ref{processB}. And using $T_k\leq T_l$, (\ref{xkl}) is finally obtained by
      \[
      XI_{B\llbracket{0,T_k}\rrbracket}=
      (XI_{B\llbracket{0,T_l}\rrbracket})I_{\llbracket{0,T_k}\rrbracket}
      =(X^{(l)}I_{B\llbracket{0,T_l}\rrbracket})I_{\llbracket{0,T_k}\rrbracket}
      =X^{(l)}I_{B\llbracket{0,T_k}\rrbracket}.
      \]

(3) It is easy to see
      \[
      (XI_B)^{\tau_n}=((XI_B)^{T_n})^{S_n}=((X^{(n)}I_B)^{T_n})^{S_n}=(X^{(n)}I_B)^{\tau_n},\quad n\in \mathbb{N}^+.
      \]
 From $\tau_n\uparrow T$ and $\bigcup\limits_{n=1}^{{+\infty}}\llbracket{0,\tau_n}\rrbracket\supseteq B$, the sequence $(\tau_n,X^{(n)})$ is a CS for $X$. Since $X^{(n)}\in \mathcal{D}$ for each $n\in \mathbb{N}^+$,  the sequence $(\tau_n,X^{(n)})$ is indeed an FCS for $X\in \mathcal{D}^B$.

(4) Suppose that $(\widetilde{T}_n,Y^{(n)})$ is an FCS for $Y\in \mathcal{D}^B$. Put $\tau_n=\widetilde{T}_n\wedge T_n$ for each $n\in \mathbb{N}^+$. From the statement of (3), $(\tau_n,X^{(n)})$ is an FCS for $X\in \mathcal{D}^B$, and $(\tau_n,Y^{(n)})$ is an FCS for $Y\in \mathcal{D}^B$. Then we obtain the relations
\begin{align*}
(aX+bY)I_{B\llbracket{0,\tau_n}\rrbracket}
&=a(XI_{B\llbracket{0,\tau_n}\rrbracket})+b(YI_{B\llbracket{0,\tau_n}\rrbracket})\\
&=a(X^{(n)}I_{B\llbracket{0,\tau_n}\rrbracket})+b(Y^{(n)}I_{B\llbracket{0,\tau_n}\rrbracket})\\
&=(aX^{(n)}+bY^{(n)})I_{B\llbracket{0,\tau_n}\rrbracket},\quad n\in \mathbb{N}^+,
\end{align*}
which shows $(\tau_n,aX^{(n)}+bY^{(n)})$ is a CS for $aX+bY$. By the assumption of the linearity, $aX^{(n)}+bY^{(n)}\in\mathcal{D}$ holds for each $n\in \mathbb{N}^+$. Therefore, $aX+bY\in \mathcal{D}^B$.

(5) For each $l\in \mathbb{N}^+$, by the statement of (2), we have
\begin{align*}
XI_{B\llbracket{0,T_l}\rrbracket}
=& X^{(l)}I_{B\llbracket{0,T_l}\rrbracket}\\
=& \left(X^{(l)}_0I_{\llbracket{0}\rrbracket}+\sum\limits_{n=1}^{l}X^{(l)}I_{\rrbracket{T_{n-1},T_n}
      \rrbracket}\right)I_{B\llbracket{0,T_l}\rrbracket} \\
=& \left(X_0I_{\llbracket{0}\rrbracket}+\sum\limits_{n=1}^{l}X^{(n)}I_{\rrbracket{T_{n-1},T_n}
      \rrbracket}\right)I_{B\llbracket{0,T_l}\rrbracket}\\
=&\left(X_0I_{\llbracket{0}\rrbracket}+\sum\limits_{n=1}^{{+\infty}}X^{(n)}I_{\rrbracket{T_{n-1},T_n}
      \rrbracket}\right)\mathfrak{I}_BI_{B\llbracket{0,T_l}\rrbracket},
\end{align*}
and this implies (\ref{x-expression}) by the statement of (1).

As for the independence of the choice of FCS, it suffices to prove that $XI_{B\llbracket{0,\tau_l}\rrbracket}=\widetilde{X}I_{B\llbracket{0,\tau_l}\rrbracket}$ holds for each $l\in \mathbb{N}^+$ and $\tau_l=T_l\wedge S_l$. From the statement of (3), $(\tau_n,X^{(n)})$ and $(\tau_n,\widetilde{X}^{(n)})$ are both FCSs for $X\in \mathcal{D}^B$. Then using the statement of (2) again, we obtain
\begin{align*}
\widetilde{X}I_{B\llbracket{0,\tau_l}\rrbracket}
&=\left(X_0I_{\llbracket{0}\rrbracket}+\sum\limits_{n=1}^{+\infty}\widetilde{X}^{(n)}I_{\rrbracket{S_{n-1},S_n}
      \rrbracket}\right)I_{B\llbracket{0,S_l}\rrbracket}I_{\llbracket{0,T_l}\rrbracket}\\
&=\left(X_0I_{\llbracket{0}\rrbracket}+\sum\limits_{n=1}^{l}\widetilde{X}^{(n)}I_{\rrbracket{S_{n-1},S_n}
      \rrbracket}\right)I_{B\llbracket{0,S_l}\rrbracket}I_{\llbracket{0,T_l}\rrbracket}\\
&=\left(X_0I_{\llbracket{0}\rrbracket}+\sum\limits_{n=1}^{l}\widetilde{X}^{(l)}I_{\rrbracket{S_{n-1},S_n}
      \rrbracket}\right)I_{B\llbracket{0,S_l}\rrbracket}I_{\llbracket{0,T_l}\rrbracket}\\
&=\widetilde{X}^{(l)}I_{B\llbracket{0,S_l}\rrbracket}I_{\llbracket{0,T_l}\rrbracket}\\
&=XI_{B\llbracket{0,\tau_l}\rrbracket}, \quad l\in \mathbb{N}^+,
\end{align*}
which completes the proof.
\end{proof}

\begin{theorem}\label{process-FS}
Let $(\tau_n)$ be an FS for $B$, $\mathcal{D}$ be a class of processes, and $X,Y\in \mathcal{D}^B$. Suppose that $(T_n,X^{(n)})$ is an FCS for $X\in \mathcal{D}^B$  (resp. a CS for $X$). Then we have the following statements:
\begin{itemize}
  \item [$(1)$] $X=Y$ if and only if for each $n\in \mathbb{N}^+$, $XI_{\llbracket{0,\tau_n}\rrbracket}=YI_{\llbracket{0,\tau_n}\rrbracket}$, or equivalently, $X^{\tau_n}=Y^{\tau_n}$.
  \item [$(2)$] $(S_n,X^{(n)})$ is also an FCS for $X\in \mathcal{D}^B$ (resp. a CS for $X$), and $(S_n)$ is also an FS for $B$, where $S_n=T_n\wedge \tau_n$ for each $n\in \mathbb{N}^+$.
  \item [$(3)$] If $\mathcal{D}\in \{\mathfrak{M}, \mathfrak{V}\}$, then $(\tau_n,X^{\tau_n})$ is an FCS for $X\in \mathcal{D}^B$ (resp. a CS for $X$).
  \item [$(4)$] Suppose that $\mathcal{D}$ is stable under stopping and localization, or that $\mathcal{D}\in \{\mathfrak{M}, \mathfrak{V}\}$. Then $X$ can be expressed as
      \begin{equation}\label{x-expression-FS}
      X=\left(X_0I_{\llbracket{0}\rrbracket}+\sum\limits_{n=1}^{{+\infty}}X^{\tau_n}I_{\rrbracket{\tau_{n-1},\tau_n}
      \rrbracket}\right)\mathfrak{I}_B,\quad \tau_0=0.
      \end{equation}
      Furthermore, if $(\widetilde{\tau}_n)$ is also an FS for $B$, then $X=\widetilde{X}$ where the process $\widetilde{X}$ given by
      \begin{equation*}
      \widetilde{X}=\left(X_0I_{\llbracket{0}\rrbracket}+\sum\limits_{n=1}^{{+\infty}}X^{\widetilde{\tau}_n}
      I_{\rrbracket{\widetilde{\tau}_{n-1},\widetilde{\tau}_n}
      \rrbracket}\right)\mathfrak{I}_B,\quad \widetilde{\tau}_0=0.
      \end{equation*}
      In this case, we say the expression of \eqref{x-expression-FS} is independent of the choice of FS $(\tau_n)$.
\end{itemize}
\end{theorem}
\begin{proof}
We just prove the case of FCS, and the case of CS can be proved similarly. Let $T$ be the debut of $B^c$.

$(1)$ From the definition of FS for $B$, $(\tau_n)$ is an increasing sequence of stopping times with $\tau_n\uparrow T$ and $\bigcup\limits_{n=1}^{{+\infty}}\llbracket{0,\tau_n}\rrbracket= B$. Then the statement is obtained by the statement $(1)$ of Theorem \ref{process}.

$(2)$ The former statement is easily obtained by using the statement $(3)$ of Theorem \ref{process} and the definition of FS for $B$. From the facts $S_n\uparrow T$ and
\[
B=\left(\bigcup\limits_{n=1}^{{+\infty}}\llbracket{0,\tau_n}\rrbracket\right)\bigcap
\left(\bigcup\limits_{m=1}^{{+\infty}}\llbracket{0,T_m}\rrbracket\right)
=\bigcup\limits_{n=1}^{{+\infty}}\llbracket{0,S_n}\rrbracket,
\]
the latter statement is obvious.

$(3)$ We just prove the case of $\mathcal{D}=\mathfrak{M}$, and the proof of the case of $\mathcal{D}=\mathfrak{V}$ is analogous. Put $Y:=X_0I_{\llbracket{0,+\infty}\llbracket}$. The statement (2) of Theorem \ref{process} shows $X^{(1)}I_{\llbracket{0}\rrbracket}=XI_{\llbracket{0}\rrbracket}$, i.e., $X^{(1)}_0=X_0, \; a.s,$ and this implies $Y=X^{(1)}_0I_{\llbracket{0,+\infty}\llbracket}$. From $X^{(1)}_0I_{\llbracket{0,+\infty}\llbracket}\in \mathfrak{M}$, we deduce $Y\in \mathfrak{M}$.
It is easy to see that $X-X_0\mathfrak{I}_B\in(\mathfrak{M}_0)^B$ with the FCS $(T_n,X^{(n)}-Y)$.
Theorem \ref{fcs-p} has shown that $(\tau_n,(X-X_0\mathfrak{I}_B)^{\tau_n}=X^{\tau_n}-Y)$ is an FCS for $X-X_0\mathfrak{I}_B\in(\mathfrak{M}_0)^B$. Then using $X^{\tau_n}-Y\in\mathfrak{M}_0$ and $Y\in \mathfrak{M}$, we have $X^{\tau_n}=(X^{\tau_n}-Y)+Y\in\mathfrak{M}$ for each $n\in \mathbb{N}^+$. Since the relations
\[
X^{\tau_n}I_{B\llbracket{0,\tau_n}\rrbracket}=((X^{\tau_n}-Y)+Y)I_{B\llbracket{0,\tau_n}\rrbracket}
=((X-X_0\mathfrak{I}_B)+X_0\mathfrak{I}_B)I_{B\llbracket{0,\tau_n}\rrbracket}
=XI_{B\llbracket{0,\tau_n}\rrbracket},\quad n\in \mathbb{N}^+
\]
show that $(\tau_n,X^{\tau_n})$ is  a CS for $X$, we deduce that $(\tau_n,X^{\tau_n})$ is an FCS for $X\in \mathfrak{M}^B$.

$(4)$ From Theorem \ref{fcs-p} or the statement (3), $(\tau_n,X^{\tau_n})$ is an FCS for $X\in \mathcal{D}^B$. Then the result is obtained by the statement $(5)$ of Theorem \ref{process}.
\end{proof}

Let $X$ be a process on $B$, and $\mathcal{D}\in \{\mathfrak{M}, \mathcal{P}\}$.
Then from the statement (5) of Theorem \ref{process}, the following relations hold true:
\begin{equation}\label{XI_B}
X\in\mathcal{D}^B \Leftrightarrow XI_B\in \mathcal{D} \Leftrightarrow \text{there exists a process $Y\in \mathcal{D}$ such that } X=Y\mathfrak{I}_B.
\end{equation}
However, such relations do not always hold for any class $\mathcal{D}$, and we give the following example.

\begin{example}
Put
\begin{align*}
B&=\llbracket{0,1}\llbracket,\\
X_t(\omega)&=\frac{1}{1-t},\quad (\omega,t)\in B,\\
T_n&=1-\frac{1}{2n},\quad n\in \mathbb{N}^+,\\
X^{(n)}&=XI_{\llbracket{0,T_n}\rrbracket}+2nI_{\rrbracket{T_n,{+\infty}}\llbracket}, \quad n\in \mathbb{N}^+.
\end{align*}
Then for each $n\in \mathbb{N}^+$, $X^{(n)}$ is a process with finite variation satisfying $XI_{B\llbracket{0,T_n}\rrbracket}=X^{(n)}I_{B\llbracket{0,T_n}\rrbracket}$. Hence, $X\in \mathfrak{V}^B$ holds.
However, $XI_B\notin \mathfrak{V}$ because it is not a c\`{a}dl\`{a}g process.
\end{example}

\begin{remark}\label{remark-cs}
Let $(X^{(n)})$ be a sequence of processes, and $(T_n)$ be an increasing sequence of stopping times with $T_n\uparrow T$ ($T$ is the debut of $B^c$) and $\bigcup\limits_{n=1}^{{+\infty}}\llbracket{0,T_n}\rrbracket\supseteq B$. From the proof of Theorem \ref{process}, if for any $k, l\in \mathbb{N}^+$ with $k\leq l$, the relation
      \begin{equation*}
      X^{(k)}I_{B\llbracket{0,T_k}\rrbracket}=X^{(l)}I_{B\llbracket{0,T_k}\rrbracket}
      \end{equation*}
holds true, then $(T_n,X^{(n)})$ is a CS for the process $X$ defined by
\begin{equation*}
      X=\left(X^{(1)}_0I_{\llbracket{0}\rrbracket}+\sum\limits_{n=1}^{{+\infty}}X^{(n)}I_{\rrbracket{T_{n-1},T_n}
      \rrbracket}\right)\mathfrak{I}_B,\quad T_0=0.
      \end{equation*}
\end{remark}

\begin{remark}\label{reprocess}
Let $\mathcal{D}$, $\mathcal{D}_1$ and $\mathcal{D}_2$ be classes of processes. The following statements are easy to obtain:
\begin{itemize}
  \item [$(1)$] If $\mathcal{D}_1\subseteq\mathcal{D}_2$, then $(\mathcal{D}_1)^B\subseteq(\mathcal{D}_2)^B$.
  \item [$(2)$]In general, we have
  \[
   (\mathcal{D}_1\cap\mathcal{D}_2)^B\subseteq(\mathcal{D}_1)^B\cap(\mathcal{D}_2)^B.
  \]
  Furthermore, if both $\mathcal{D}_1$ and $\mathcal{D}_2$ are stable under stopping and localization, then
  \[
    (\mathcal{D}_1\cap\mathcal{D}_2)^B=(\mathcal{D}_1)^B\cap(\mathcal{D}_2)^B.
  \]
  \item [$(3)$] Suppose $X\in\mathcal{D}^B$ with an FCS $(T_n,X^{(n)})$. If $\widetilde{B}$ is another PSIT with $\widetilde{B}\subseteq B$, then $X\mathfrak{I}_{\widetilde{B}}\in\mathcal{D}^{\widetilde{B}}$ with the FCS $(S_n,X^{(n)})$, where for each $n\in \mathbb{N}^+$, $S_n=T_n\wedge S$ and $S$ is the debut of $\widetilde{B}^c$. Specially, if $Y\in\mathcal{D}$, then $Y\mathfrak{I}_{B}\in\mathcal{D}^{B}$.
  \item [$(4)$] Let $(T_n,X^{(n)})$ be an FCS for $X\in\mathcal{D}^B$. The sequence $(T_n)$ is not necessarily an FS for $B$. For instance, providing $B=\llbracket{0,T}\llbracket$, $Y\in \mathcal{S}$, and $X=Y\mathfrak{I}_{B}$ with a predictable stopping time $T$, then $(T_n=T,X^{(n)}=Y)$ is an FCS for $X\in\mathcal{S}^B$, but $(T_n)$ is not an FS for $B$.
\end{itemize}
\end{remark}

\subsection{Jump processes of c\`{a}dl\`{a}g processes on PSITs}

Jump processes of c\`{a}dl\`{a}g processes play an important role in classic stochastic calculus, and then it is natural to introduce jump processes of c\`{a}dl\`{a}g processes on PSITs.
Recall that, for a c\`{a}dl\`{a}g process $X$, its left-hand limit process (or simply, its left-limit process) $X_{-}=(X_{t-})_{t\in \mathbb{R}^+}$, and its jump process $\Delta X=(\Delta X_t)_{t\in \mathbb{R}^+}$ are respectively defined by
\begin{equation}\label{limit}
\left\{
\begin{aligned}
X_{t-}&=\lim\limits_{s<t,s\uparrow t}X_s\quad \text{for $t>0$},\quad X_{0-}=X_0;\\
\Delta X_t&=X_t-X_{t-}.
\end{aligned}
\right.
\end{equation}
In this subsection, we set $X_{0-}=X_0$ for any process $X$ that we consider. Obviously, the left-limit processes $X_{-}$ in \eqref{limit} can be also defined for a more general process $X$, for example, the left-limit process $X_{-}$ of a left-continuous process. We say the left-limit process $X_{-}$ of a process $X$ exists if all paths of $X$ admit finite left-hand limits.

Following \eqref{limit}, we first define left-limit processes for general processes on PSITs (if they exist), and then study their fundamental properties.

\begin{definition}\label{X+-}
Let $X$ be a process on $B$.
If for all $(\omega,t)\in B$ with $t>0$, the left-hand limits $X(\omega,t-)$ exist, then the left-limit process on $B$ of $X$, denoted by $X_{-}$, is defined by
  \[
   X_{-}(\omega,t)=\left\{
   \begin{aligned}
    &X(\omega,0),&&\quad \omega\in \Omega,\;t=0,\\
    &X(\omega,t-)=\lim\limits_{s<t,s\uparrow t}X(\omega,s),&&\quad (\omega,t)\in B,\; t>0.
    \end{aligned}
    \right.
  \]
\end{definition}

Similarly, we say the left-limit process $X_{-}$ of a process $X$ on $B$ exists, if $X_{-}$ is well-defined in Definition \ref{X+-}. It is obvious that $X_{-}$ is a process on $B$. More importantly, the definition of $X_{-}$ is analogous to that of the left-limit process in \eqref{limit}.
In brief, the left-limit process $X_{-}$ on $B$ exists if, for each $\omega\in \Omega$, the path $X_{.}(\omega)$ admits finite left-hand limits on the section $B_\omega=\{t: (\omega,t)\in B\}$.

For a process on $B$, the following theorem presents sufficient conditions of the existence of its left-limit process, and reveals the relationship between its left-limit process on $B$ and the classic left-limit process.
\begin{theorem}\label{X-left}
Let $X$ be a process on $B$.
\begin{itemize}
  \item [$(1)$] If there exists a CS $(T_n,X^{(n)})$ for $X$ such that $(X^{(n)})_{-}$ exists for each $n\in \mathbb{N}^+$, then $X_{-}$ exists, and $(T_n,(X^{(n)})_{-})$ is a CS for $X_{-}$.
  \item [$(2)$] If there exists an FS $(\tau_n)$ for $B$ such that $(X^{\tau_n})_{-}$ exists for each $n\in \mathbb{N}^+$, then $X_{-}$ exists, and $(\tau_n,(X^{\tau_n})_{-})$ is a CS for $X_{-}$.
\end{itemize}
\end{theorem}
\begin{proof}
$(1)$ Let $(\omega,t)\in B$ with $t>0$. From $\bigcup\limits_{n=1}^{{+\infty}}\llbracket{0,T_n}\rrbracket\supseteq B$, there exists an integer $m\in \mathbb{N}^+$ such that $(\omega,t)\in B\llbracket{0,T_m}\rrbracket$. By the CS $(T_n,X^{(n)})$ for $X$, we deduce $X(\omega,s)=X^{(m)}(\omega,s)$ for all $s\in[0,t]$. Since $X^{(m)}(\omega,t-)$ exists, we also obtain the existence of $X(\omega,t-)$ and the relation $X(\omega,t-)=X^{(m)}(\omega,t-)$. Using the arbitrariness of $(\omega,t)\in B$ with $t>0$, we deduce that for all $(\omega,t)\in B$ with $t>0$, the left-hand limits $X(\omega,t-)$ exist, thus proving the existence of $X_{-}$. Similarly, we can prove that for each $n\in \mathbb{N}^+$,
\[
X(\omega,t-)=X^{(n)}(\omega,t-), \quad (\omega,t)\in B\llbracket{0,T_n}\rrbracket,
\]
which, by Definition \ref{X+-}, implies that
\[
X_{-}I_{B\llbracket{0,T_n}\rrbracket}=(X^{(n)})_{-}I_{B\llbracket{0,T_n}\rrbracket}, \quad n\in \mathbb{N}^+.
\]
Therefore, $(T_n,(X^{(n)})_{-})$ is a CS for $X_{-}$.

$(2)$ From Theorem \ref{fcs-p}, $(\tau_n,X^{\tau_n})$ is always a CS for $X$. Then the statement is a direct result of $(1)$.
\end{proof}

\begin{corollary}\label{deltaX}
Let $X\in \mathcal{R}^B$ with the FCS $(T_n,X^{(n)})$. Then $X_{-}$ exists, and $(T_n,(X^{(n)})_{-})$ is a CS for $X_{-}$.
\end{corollary}
\begin{proof}
From the FCS $(T_n,X^{(n)})$ for $X\in \mathcal{R}^B$, we have $X^{(n)}\in \mathcal{R}$ for each $n\in \mathbb{N}^+$. Then $(T_n,X^{(n)})$ is a CS for $X$ such that $(X^{(n)})_{-}$ exists for each $n\in \mathbb{N}^+$. Hence, by Theorem \ref{X-left}, we obtain the statement.
\end{proof}

\begin{corollary}\label{SP-}
If $X$ is an adapted c\`{a}dl\`{a}g process on $B$ with the FCS $(T_n,X^{(n)})$, then $X_{-}$ is a locally bounded predictable process on $B$, and $(T_n,(X^{(n)})_{-})$ is an FCS for $X_{-}$ (a locally bounded predictable process on $B$).
\end{corollary}
\begin{proof}
For each $n\in \mathbb{N}^+$, $X^{(n)}$ is an adapted c\`{a}dl\`{a}g process, and then from Theorem 7.7 in \cite{He}, $(X^{(n)})_{-}$ is  a locally bounded predictable process. Since Theorem \ref{X-left} also shows that $(T_n,(X^{(n)})_{-})$ is a CS for $X_{-}$, we complete the proof.
\end{proof}

Now we can introduce and then study the jump process of a c\`{a}dl\`{a}g process on a PSIT.
Let $X\in \mathcal{R}^B$. As usual, we denote by
 \[
  \Delta X:=X-X_{-}
 \]
the jump process of $X$. From Corollary \ref{deltaX}, the jump process $\Delta X$ is well defined. Fundamental properties of jump processes on PSITs are presented in the following theorem.

\begin{theorem}\label{delta}
Let $X,Y\in \mathcal{R}^B$, and $Z\in\mathcal{R}$.
\begin{itemize}
  \item [$(1)$] If $(T_n,X^{(n)})$ is an FCS for $X\in \mathcal{R}^B$, then $(T_n,\Delta X^{(n)})$ is a CS for $\Delta X$.

  \item [$(2)$] For all $a\in \mathbb{R}$,
\begin{align}\label{cadlag-linear}
\Delta(aX)=a \Delta X, \quad \Delta(X+Y)= \Delta X+\Delta Y.
\end{align}

  \item [$(3)$] If $\widetilde{B}$ is another PSIT satisfying $\widetilde{B}\subseteq B$, then
  \begin{equation}\label{eqXT0}
  \Delta (X\mathfrak{I}_{\widetilde{B}})=(\Delta X)\mathfrak{I}_{\widetilde{B}}.
  \end{equation}
  Specially, $\Delta (Z\mathfrak{I}_{B})=(\Delta Z)\mathfrak{I}_{B}$.

  \item [$(4)$] If $T$ is a stopping time on $B$, then
  \begin{align}\label{eqXT}
\Delta X^T=\Delta XI_{\llbracket{0,T}\rrbracket},\quad \Delta X^{T-}=\Delta XI_{\llbracket{0,T}\llbracket},
\end{align}
  where $X^{T-}$ is defined by
 \begin{equation*}
     X^{T-}:=XI_{\llbracket{0,T}\llbracket}+X_{T-}I_{\llbracket{T,{+\infty}}\llbracket}.
  \end{equation*}
  Specially, $\Delta Z^S=\Delta ZI_{\llbracket{0,S}\rrbracket}$ and $\Delta Z^{S-}=\Delta ZI_{\llbracket{0,S}\llbracket}$, where $S$ is a stopping time.

  \item [$(5)$] Let $\mathcal{C}$ be the class of all continuous processes. Then $\Delta X=0$ if and only if $X\in \mathcal{C}^B$.
\end{itemize}
\end{theorem}
\begin{proof}
$(1)$ The statement is a direct result of Theorem \ref{X-left} and Corollary \ref{deltaX}.

$(2)$ From the statement (3) of Theorem \ref{process}, we can assume that $(T_n,X^{(n)})$ and $(T_n,Y^{(n)})$ are FCSs for $X\in \mathcal{R}^B$ and $Y\in \mathcal{R}^B$, respectively. Then $(T_n,aX^{(n)})$ is an FCS for $aX\in \mathcal{R}^B$, and $(T_n,X^{(n)}+Y^{(n)})$ is an FCS for $X+Y\in \mathcal{R}^B$. From $(1)$, $(T_n,\Delta X^{(n)})$ is a CS for $\Delta X$, and $(T_n,\Delta Y^{(n)})$ is a CS for $\Delta Y$. For each $n\in \mathbb{N}^+$, we have the relations
\begin{align*}
\Delta(aX)I_{B\llbracket{0,T_n}\rrbracket}
=\Delta(aX^{(n)})I_{B\llbracket{0,T_n}\rrbracket}
=a\Delta(X^{(n)})I_{B\llbracket{0,T_n}\rrbracket}
=(a\Delta X)I_{B\llbracket{0,T_n}\rrbracket}
\end{align*}
and
\begin{align*}
\Delta(X+Y)I_{B\llbracket{0,T_n}\rrbracket}
=\Delta(X^{(n)}+Y^{(n)})I_{B\llbracket{0,T_n}\rrbracket}
=(\Delta(X^{(n)})+\Delta(Y^{(n)}))I_{B\llbracket{0,T_n}\rrbracket}
=(\Delta X+\Delta Y)I_{B\llbracket{0,T_n}\rrbracket},
\end{align*}
which, by $(1)$ of Theorem \ref{process}, implies \eqref{cadlag-linear}.

$(3)$ Suppose that $(T_n,X^{(n)})$ is an FCS for $X\in \mathcal{R}^B$. From Remark \ref{reprocess}, we deduce $X\mathfrak{I}_{\widetilde{B}}\in\mathcal{R}^{\widetilde{B}}$ with the FCS $(S_n,X^{(n)})$, where $S_n=S\wedge T_n$ for each $n\in \mathbb{N}^+$, and $S$ is the debut of ${\widetilde{B}}^c$. The statement of $(1)$ shows that $(T_n,\Delta X^{(n)})$ and $(S_n,\Delta X^{(n)})$ are CSs for $\Delta X$ and $\Delta (X\mathfrak{I}_{\widetilde{B}})$ respectively. Then for each $n\in \mathbb{N}^+$,
\[
\Delta(X\mathfrak{I}_{\widetilde{B}})I_{\widetilde{B}\llbracket{0,S_n}\rrbracket}
=\Delta X^{(n)}I_{\widetilde{B}\llbracket{0,S_n}\rrbracket}
=(\Delta X^{(n)}I_{B\llbracket{0,T_n}\rrbracket})I_{\widetilde{B}\llbracket{0,S_n}\rrbracket}
=(\Delta XI_{B\llbracket{0,T_n}\rrbracket})I_{\widetilde{B}\llbracket{0,S_n}\rrbracket}
=(\Delta X)\mathfrak{I}_{\widetilde{B}}I_{\widetilde{B}\llbracket{0,S_n}\rrbracket},
\]
which, by $(1)$ of Theorem \ref{process}, implies \eqref{eqXT0}.

$(4)$ Let $(\omega,t)\in\llbracket{0,T}\rrbracket\subseteq B$. From Definition \ref{X+-}, $X(\omega,t-)$ exists. It is easy to see that $X(\omega,s)=X(\omega,s\wedge T(\omega))=X^T(\omega,s)$ for $s\in[0,t]$ such that $X(\omega,t-)=X^T(\omega,t-)$. Then by Definition \ref{X+-}, we deduce that
\[
(\Delta X^T)(\omega,t)=X^T(\omega,t)-X^T(\omega,t-)=X(\omega,t)-X(\omega,t-)=(\Delta X)(\omega,t),
\]
which implies $\Delta X^T=\Delta XI_{\llbracket{0,T}\rrbracket}$ on $\llbracket{0,T}\rrbracket$. On the other hand, let $(\omega,t)\in\rrbracket{T,+\infty}\llbracket$. From the definition of $X^T$, it is obvious that $X^T(\omega,s)=X(\omega,T(\omega))$ for $s\in]T(\omega),t]$ such that $X^T(\omega,t-)=X^T(\omega,t)$. Then we deduce that
\[
(\Delta X^T)(\omega,t)=X^T(\omega,t)-X^T(\omega,t-)=0,
\]
which implies $\Delta X^T=\Delta XI_{\llbracket{0,T}\rrbracket}$ on $\rrbracket{T,+\infty}\llbracket$. Thus, we obtain the former equation of \eqref{eqXT}. Using the facts $X^{T-}=X^T-\Delta X_TI_{\llbracket{T,+\infty}\llbracket}$ and $\Delta (I_{\llbracket{T,+\infty}\llbracket})=I_{\llbracket{T}\rrbracket}$, the latter equation of \eqref{eqXT} can be obtained by
\[
\Delta X^{T-}=\Delta X^T-\Delta (\Delta X_TI_{\llbracket{T,+\infty}\llbracket})
=\Delta XI_{\llbracket{0,T}\rrbracket}-\Delta X_TI_{\llbracket{T}\rrbracket}
=\Delta XI_{\llbracket{0,T}\llbracket},
\]
and we complete the proof of $(4)$.

$(5)$ {\it Sufficiency}. Suppose that $(T_n,X^{(n)})$ is an FCS for $X\in \mathcal{C}^B$. Then $\Delta X^{(n)}=0$ for each $n\in \mathbb{N}^+$. From $\mathcal{C}^B\subseteq \mathcal{R}^B$, the statement $(1)$ shows
\[
\Delta X I_{B\llbracket{0,T_n}\rrbracket}=\Delta X^{(n)}I_{B\llbracket{0,T_n}\rrbracket}=0,\quad n\in \mathbb{N}^+,
\]
which yields $\Delta X=0$.

{\it Necessity}. Suppose $\Delta X=0$. Let $B$ be given by \eqref{B},
and $(T_n,X^{(n)})$ be an FCS for $X\in \mathcal{R}^B$. For each $n\in \mathbb{N}^+$, $X^{(n)}$ is a c\`{a}dl\`{a}g process, and then $Y^{(n)}$ is a c\`{a}dl\`{a}g process, where $Y^{(n)}:=(X^{(n)})^{{T_n}\wedge (T_F-)}$. From the statement $(4)$, for each $n\in \mathbb{N}^+$,
\[
\Delta(Y^{(n)})=\Delta\left((X^{(n)})^{{T_n}\wedge (T_F-)}\right)=\Delta(X^{(n)})I_{\llbracket{0,T_n}\rrbracket\llbracket{0,T_F}\llbracket}
=\Delta XI_{B\llbracket{0,T_n}\rrbracket}=0,
\]
which implies $Y^{(n)}\in \mathcal{C}$. By \eqref{XYT}, we deduce
\[
XI_{B\llbracket{0,T_n}\rrbracket}=X^{(n)}I_{B\llbracket{0,T_n}\rrbracket}=Y^{(n)}I_{B\llbracket{0,T_n}\rrbracket},
\quad n\in \mathbb{N}^+.
\]
Hence, $X\in \mathcal{C}^B$ with the FCS $(T_n,Y^{(n)})$.
\end{proof}

Finally, we present an example of processes on PSITs in the study of stochastic analysis.
\begin{example}\label{ex-absolute}
Suppose that $\mathbb{Q}$ is another probability measure on the filtered space $(\Omega,\mathcal{F},\mathbb{F})$, and that $\mathbb{Q}$ is locally absolutely continuous w.r.t. $\mathbb{P}$, i.e.,
$\mathbb{Q}\overset{\mathrm{loc}}\ll\mathbb{P}$ (see, e.g., \cite{Jacod,He}). Put
\begin{align}
B&=\bigcup_n\llbracket{0,\tau_n}\rrbracket,\label{absolute}\\
\tau_n&:=\inf\left\{t: Z_t\leq \frac{1}{n}\right\},\quad n\in \mathbb{N}^+,\nonumber
\end{align}
where $Z$ is the density process of $\mathbb{Q}$, relative to $\mathbb{P}$. Let $\mathcal{V}(\mathbb{Q})$ and $\mathcal{S}(\mathbb{Q})$ be the classes of all $\mathbb{Q}$-adapted process with finite variation and  $\mathbb{Q}$-semimartingales, respectively.  Then we have the following statements:
\begin{itemize}
  \item [$(1)$] $B$ is a PIST, and $(\tau_n)$ is an FS for $B$.
  \item [$(2)$] Let $X\in\mathcal{V}^B$ and $Y\in\mathcal{S}^B$. From Theorem \ref{fcs-p}, $(\tau_n,X^{\tau_n})$ and $(\tau_n,Y^{\tau_n})$ are FCSs for $X\in\mathcal{V}^B$ and $Y\in\mathcal{S}^B$, respectively. And from $(1)$ of Theorem \ref{delta}, $(\tau_n,\Delta X^{\tau_n})$ and $(\tau_n,\Delta Y^{\tau_n})$ are CSs for $\Delta X$ and $\Delta Y$, respectively.
  \item [$(3)$] Let $X$ be an adapted c\`{a}dl\`{a}g process. From Theorem 12.18 in \cite{He}, the following relations hold:
  \[
  \left\{
  \begin{aligned}
  X\in\mathcal{V}(\mathbb{Q})\Leftrightarrow \widetilde{X}\in\mathcal{V}^B,\\
  X\in\mathcal{S}(\mathbb{Q})\Leftrightarrow \widetilde{X}\in\mathcal{S}^B,
  \end{aligned}
  \right.
   \]
  where $\widetilde{X}=X\mathfrak{I}_B$.
\end{itemize}
 \end{example}

\section{Lebesgue-Stieltjes integrals on PSITs}\label{section3}\noindent
\setcounter{equation}{0}
In this section, we first investigate L-S integrals on PSITs of measurable processes w.r.t. processes with finite variation, and then use their fundamental properties to study L-S integrals on PSITs of predictable processes w.r.t. adapted processes with finite variation.

Let $H\in \mathfrak{M}$ and $V\in \mathfrak{V}$. Recall that $H$ is integrable w.r.t. $A$ (see, e.g., Definition 3.45 in \cite{He}) if for all $(\omega,t)\in \Omega\times \mathbb{R}^{+}$,
$\int_{[0,t]}|H_s(\omega)||dA_s(\omega)|<{+\infty}$, where $\int_{[0,t]}|dA_s|$ is the variation process of $A$. And the L-S integral by paths of $H$ w.r.t. $A$, denoted by $H.A$, is define by
\begin{equation}\label{def-c-HA}
(H.A)_t(\omega):=\int_{[0,t]}H_s(\omega)dA_s(\omega),\quad (\omega,t)\in \Omega\times \mathbb{R}^{+}.
\end{equation}
We also say the integral $H.A$ exists if $H$ is integrable w.r.t. $A$.

Based on the stochastic integral \eqref{def-c-HA}, we define the L-S integrals on PSITs of measurable processes w.r.t. processes with finite variation.

\begin{definition}\label{HA}
Let $H\in \mathfrak{M}^B$ and $V\in \mathfrak{V}^B$. We say that $H$ is integrable on $B$ w.r.t. $A$, if for all $(\omega,t)\in B$,
  \[
   \int_{[0,t]}|H_s(\omega)||dA_s(\omega)|<{+\infty}.
   \]
At this time, the process $L$ defined by
  \begin{equation}\label{HA-de}
   L(\omega,t):=\int_{[0,t]}H_s(\omega)dA_s(\omega),\quad (\omega,t)\in B
  \end{equation}
is called the L-S integral on $B$ of $H$ w.r.t $A$, and is denoted by $H_{\bullet}A$.
\end{definition}

Let $H\in \mathfrak{M}^B$ and $V\in \mathfrak{V}^B$. We also say the integral $H_{\bullet}A$ exists if $H$ is integrable on $B$ w.r.t. $A$.  In Definition \ref{HA}, suppose that $(T_n, H^{(n)})$ is an FCS for $H\in\mathfrak{M}^{B}$ and that $(T_n,A^{(n)})$ is an FCS for $A\in\mathfrak{V}^{B}$. For each $(\omega,t)\in B$, there exists an integer $n$ such that $(\omega,t)\in {B\llbracket{0,T_n}\rrbracket}$, and hence, $H(\omega,s)=H^{(n)}(\omega,s)$ and $A(\omega,s)=A^{(n)}(\omega,s)$ for $s\in [0,t]$. Therefore, the integral in \eqref{HA-de} are considered as the L-S integral by paths.
Furthermore, if $H$ is integrable on $B$ w.r.t. $A$, then it is easy to see that $H_{\bullet}A$ is a process on $B$.

\begin{remark}\label{HAB=HA}
The L-S integral $H_{\bullet}A$ defined by \eqref{HA-de} degenerates to the L-S integral $H.A$ defined by \eqref{def-c-HA} if $B=\llbracket{0,+\infty}\llbracket=\Omega\times\mathbb{R}^+$. More precisely, the following relation holds:
\begin{itemize}
  \item [] If $H\in \mathfrak{M}^{\llbracket{0,+\infty}\llbracket}$ and $V\in \mathfrak{V}^{\llbracket{0,+\infty}\llbracket}$, then $H_{\bullet}A=H.A$.
\end{itemize}
Indeed, from Definition \ref{HA}, it suffices to prove $\mathfrak{M}=\mathfrak{M}^{\llbracket{0,+\infty}\llbracket}$ and $\mathfrak{V}=\mathfrak{V}^{\llbracket{0,+\infty}\llbracket}$.
We just prove $\mathfrak{M}=\mathfrak{M}^{\llbracket{0,+\infty}\llbracket}$, and the proof of $\mathfrak{V}=\mathfrak{V}^{\llbracket{0,+\infty}\llbracket}$ is analogous. The inclusion $\mathfrak{M}\subseteq\mathfrak{M}^{\llbracket{0,+\infty}\llbracket}$  has been shown in Remark \ref{c-process}, and it remains to prove $\mathfrak{M}^{\llbracket{0,+\infty}\llbracket}\subseteq\mathfrak{M}$.
Let $X\in\mathfrak{M}^{\llbracket{0,+\infty}\llbracket}$. Put $\tau_n=+\infty$ for each $n\in \mathbb{N}^+$, and then $(\tau_n)$ is an FS for $\llbracket{0,+\infty}\llbracket$. From the statement $(3)$ of Theorem \ref{process-FS}, $(\tau_n,X^{\tau_n})$ is an FCS for $X\in\mathfrak{M}^{\llbracket{0,+\infty}\llbracket}$, which implies $X=X^{\tau_1}\in\mathfrak{M}$.
Thus, we deduce $\mathfrak{M}^{\llbracket{0,+\infty}\llbracket}\subseteq\mathfrak{M}$, and finish the proof.
\end{remark}

The following two theorems reveal the relation between L-S integrals on $B$ and classic L-S integrals.
Theorem \ref{HA-equivalent} shows that the integrability of L-S integrals on $B$ essentially depends on the integrability of classic L-S integrals, and also presents the sufficient and necessary conditions of the existence of L-S integrals on $B$. More precisely, Theorem \ref{HA-FCS} characterizes an L-S integral on $B$ as a summation of a sequence of classic L-S integrals, which is in accord with Definition \ref{processB} of processes on $B$.

\begin{theorem}\label{HA-equivalent}
Let $H\in \mathfrak{M}^B$ and $A\in \mathfrak{V}^B$. Then the following statements are equivalent:
\begin{description}
  \item [$(i)$] $H$ is integrable on $B$ w.r.t. $A$.
  \item [$(ii)$] There exists an FS $(\tau_n)$ for $B$ such that for each $n\in \mathbb{N}^+$, $H^{\tau_n}$ is integrable w.r.t. $A^{\tau_n}$.
  \item [$(iii)$] There exist FCSs $(T_n, H^{(n)})$ for $H\in\mathfrak{M}^{B}$ and $(T_n,A^{(n)})$ for $A\in\mathfrak{V}^{B}$ such that for each $n\in \mathbb{N}^+$, $H^{(n)}$ is integrable w.r.t. $A^{(n)}$.
\end{description}
\end{theorem}
\begin{proof}
$(i)\Rightarrow (ii)$. Suppose $H$ is integrable on $B$ w.r.t. $A$. Let $(\tau_n)$ be an FS for $B$. From the statement $(3)$ of Theorem \ref{process-FS}, for each $n\in \mathbb{N}^+$, $H^{\tau_n}\in\mathfrak{M}$ and $A^{\tau_n}\in\mathfrak{V}$, and $H=H^{\tau_n}$ and $A=A^{\tau_n}$ on ${\llbracket{0,\tau_n}\rrbracket}$. For each $(\omega,t)\in \Omega\times\mathbb{R}^+$ and $n\in \mathbb{N}^+$, by using $(\omega,t\wedge \tau_n(\omega))\in B$, we have
\begin{align*}
\int_{[0,t]}|H^{\tau_n}_s(\omega)||dA^{\tau_n}_s(\omega)|
&=\int_{[0,t\wedge \tau_n(\omega)]}|H^{\tau_n}_s(\omega)||dA^{\tau_n}_s(\omega)|\\
&=\int_{[0,t\wedge \tau_n(\omega)]}|H_s(\omega)||dA_s(\omega)|\\
&<{+\infty},
\end{align*}
which implies $H^{\tau_n}$ is integrable w.r.t. $A^{\tau_n}$.

$(ii)\Rightarrow (iii)$. Suppose the statement $(ii)$ holds. For each $n\in \mathbb{N}^+$, put $T_n=\tau_n$, $H^{(n)}=H^{\tau_n}$ and $A^{(n)}=A^{\tau_n}$. Then from the statement (3) of Theorem \ref{process-FS}, $(T_n,H^{(n)})$ and $(T_n,A^{(n)})$ are FCSs for $H\in \mathfrak{M}$ and $A\in \mathfrak{V}$ respectively such that for each $n\in \mathbb{N}^+$, $H^{(n)}.A^{(n)}=H^{\tau_n}.A^{\tau_n}$ exists.

$(iii)\Rightarrow (i)$. Suppose the statement $(iii)$ holds. Let $(\omega,t)\in B$. There exists an integer $n\in \mathbb{N}^+$ such that $(\omega,t)\in {B\llbracket{0,T_n}\rrbracket}$. Noticing the facts $H=H^{(n)}$ and $A=A^{(n)}$ on ${B\llbracket{0,T_n}\rrbracket}$ and using the existence of $H^{(n)}.A^{(n)}$, we have
\[
\int_{[0,t]}|H_s(\omega)||dA_s(\omega)|
=\int_{[0,t]}|H^{(n)}_s(\omega)||dA^{(n)}_s(\omega)|<{+\infty},
 \]
which implies $H$ is integrable on $B$ w.r.t. $A$.
\end{proof}

\begin{theorem}\label{HA-FCS}
Let $H\in \mathfrak{M}^B$ and $A\in \mathfrak{V}^B$. Suppose $H$ is integrable on $B$ w.r.t. $A$. Then $H_{\bullet}A\in \mathfrak{V}^B$, and the following statements holds:
\begin{itemize}
  \item [$(1)$] If $(T_n, H^{(n)})$ for $H\in\mathfrak{M}^{B}$ and $(T_n,A^{(n)})$ for $A\in\mathfrak{V}^{B}$ are FCSs such that for each $n\in \mathbb{N}^+$, $H^{(n)}$ is integrable w.r.t. $A^{(n)}$, then $(T_n,H^{(n)}.A^{(n)})$ is an FCS for $H_{\bullet}A\in\mathfrak{V}^{B}$, and $H_{\bullet}A$ can be expressed as
      \begin{equation}\label{HA-expression}
      H_{\bullet}A=\left((H_0A_0)I_{\llbracket{0}\rrbracket}+\sum\limits_{n=1}^{{+\infty}}(H^{(n)}.A^{(n)})
      I_{\rrbracket{T_{n-1},T_n}\rrbracket}\right)\mathfrak{I}_B,\quad T_0=0.
      \end{equation}
      Furthermore, if $(S_n, \widetilde{H}^{(n)})$ for $H\in\mathfrak{M}^{B}$ and $(\widetilde{S}_n,\widetilde{A}^{(n)})$ for $A\in\mathfrak{V}^{B}$ are FCSs such that for each $n\in \mathbb{N}^+$, $\widetilde{H}^{(n)}$ is integrable w.r.t. $\widetilde{A}^{(n)}$, then $H_{\bullet}A=\widetilde{X}$ where the process $\widetilde{X}$ is given by
      \begin{equation*}
      \widetilde{X}=\left((H_0A_0)I_{\llbracket{0}\rrbracket}+\sum\limits_{n=1}^{{+\infty}}(\widetilde{H}^{(n)}.\widetilde{A}^{(n)})
      I_{\rrbracket{\widetilde{T}_{n-1},\widetilde{T}_n}\rrbracket}\right)\mathfrak{I}_B,\quad \widetilde{T}_0=0,
      \end{equation*}
      and $\widetilde{T}_n=S_n\wedge \widetilde{S}_n,\;n\in \mathbb{N}^+$. In this case, we say that the expression of \eqref{HA-expression} is independent of the choice of FCSs $(T_n, H^{(n)})$ for $H\in\mathfrak{M}^{B}$ and $(T_n,A^{(n)})$ for $A\in\mathfrak{V}^{B}$.
  \item [$(2)$] If $(\tau_n)$ is an FS for $B$, then $(T_n,H^{\tau_n}.A^{\tau_n})$ is an FCS for $H_{\bullet}A\in\mathfrak{V}^{B}$, and $H_{\bullet}A$ can be expressed as
      \begin{equation}\label{HA-expression0}
      H_{\bullet}A=\left((H_0A_0)I_{\llbracket{0}\rrbracket}+\sum\limits_{n=1}^{{+\infty}}
      (H^{\tau_n}.A^{\tau_n})I_{\rrbracket{\tau_{n-1},\tau_n}
      \rrbracket}\right)\mathfrak{I}_B,\quad \tau_0=0.
      \end{equation}
      Furthermore, if $(\widetilde{\tau}_n)$ is also an FS for $B$, then $H_{\bullet}A=\widetilde{X}$ where the process $\widetilde{X}$ is given by
      \begin{equation*}
      \widetilde{X}=\left((H_0A_0)I_{\llbracket{0}\rrbracket}+\sum\limits_{n=1}^{{+\infty}}
      (H^{\widetilde{\tau}_n}.A^{\widetilde{\tau}_n})I_{\rrbracket{\widetilde{\tau}_{n-1},\widetilde{\tau}_n}
      \rrbracket}\right)\mathfrak{I}_B,\quad \widetilde{\tau}_0=0.
      \end{equation*}
      In this case, we say that the expression of \eqref{HA-expression0} is independent of the choice of FS $(\tau_n)$ for $B$.
\end{itemize}
\end{theorem}
\begin{proof}
$(1)$ From the definitions of FCSs, for each $n\in \mathbb{N}^+$,
\[
H^{(n)}I_{B\llbracket{0,T_n}\rrbracket}=HI_{B\llbracket{0,T_n}\rrbracket},\quad
A^{(n)}I_{B\llbracket{0,T_n}\rrbracket}=AI_{B\llbracket{0,T_n}\rrbracket}.
\]
Then for all $n\in \mathbb{N}^+$ and $(\omega,t)\in B\llbracket{0,T_n}\rrbracket$, it is easy to obtain
\[
H^{(n)}(\omega,s)=H(\omega,s),\quad A^{(n)}(\omega,s)=A(\omega,s),\quad 0\leq s\leq t,
\]
which indicates
\[
\int_{[0,t]}H^{(n)}_s(\omega)dA^{(n)}_s(\omega)=\int_{[0,t]}H_s(\omega)dA_s(\omega).
\]
Hence, we deduce that for each $n\in \mathbb{N}^+$, $(H^{(n)}.A^{(n)})I_{B\llbracket{0,T_n}\rrbracket}=(H_{\bullet}A)I_{B\llbracket{0,T_n}\rrbracket}$. By noticing $H^{(n)}.A^{(n)}\in\mathfrak{V}$ for each $n\in \mathbb{N}^+$, we have $H_{\bullet}A\in \mathfrak{V}^B$ with the FCS $(T_n,H^{(n)}.A^{(n)})$.
The expression \eqref{HA-expression} can be obtained easily from \eqref{x-expression}.

From Theorem \ref{process}, $(\widetilde{T}_n, \widetilde{H}^{(n)})$ is an FCS for $H\in\mathfrak{M}^{B}$, and $(\widetilde{T}_n,\widetilde{A}^{(n)})$ is an FCS for $A\in\mathfrak{V}^{B}$. Similarly, we can prove that $(\widetilde{T}_n,\widetilde{H}^{(n)}.\widetilde{A}^{(n)})$ is an FCS for $H_{\bullet}A\in\mathfrak{V}^{B}$. Then using the independence property of \eqref{x-expression}, we deduce $H_{\bullet}A=\widetilde{X}$.

$(2)$ Suppose $(\tau_n)$ is an FS for $B$. Using the statement (3) of Theorem \ref{process-FS} and the definitions of FCSs, it is easy to see that for each $n\in \mathbb{N}^+$,
\[
H^{\tau_n}I_{\llbracket{0,\tau_n}\rrbracket}=HI_{\llbracket{0,\tau_n}\rrbracket},\quad
A^{\tau_n}I_{\llbracket{0,\tau_n}\rrbracket}=AI_{\llbracket{0,\tau_n}\rrbracket}.
\]
Then for all $n\in \mathbb{N}^+$ and $(\omega,t)\in \llbracket{0,\tau_n}\rrbracket$, we have
\[
H^{\tau_n}(\omega,s)=H(\omega,s),\quad A^{\tau_n}(\omega,s)=A(\omega,s),\quad 0 \leq s\leq t,
\]
which implies
\[
\int_{[0,t]}H^{\tau_n}_s(\omega)dA^{\tau_n}_s(\omega)=\int_{[0,t]}H_s(\omega)dA_s(\omega).
\]
Hence, for each $n\in \mathbb{N}^+$, $(H^{\tau_n}.A^{\tau_n})I_{\llbracket{0,\tau_n}\rrbracket}=(H_{\bullet}A)I_{\llbracket{0,\tau_n}\rrbracket}$ is deduced. By noticing $H^{\tau_n}.A^{\tau_n}\in\mathfrak{V}$ for each $n\in \mathbb{N}^+$, we have $H_{\bullet}A\in \mathfrak{V}^B$ with the FCS $(\tau_n,H^{\tau_n}.A^{\tau_n})$.
The expression \eqref{HA-expression0} can be obtained easily from \eqref{x-expression}.

Similarly, we can prove that $(\widetilde{\tau}_n,H^{\widetilde{\tau}_n}.A^{\widetilde{\tau}_n})$ is also an FCS for $H_{\bullet}A\in\mathfrak{V}^{B}$. Then the independence property of \eqref{x-expression} deduces $H_{\bullet}A=\widetilde{X}$.
\end{proof}

Let $H\in \mathfrak{M}^B$ and $A\in \mathfrak{V}^B$, and $H$ be integrable on $B$ w.r.t. $A$. From Theorems \ref{HA-equivalent} and \ref{HA-FCS}, the stochastic integral $H_{\bullet}A$ is essentially characterized by a sequence of stochastic integrals relative to FCSs for $H\in \mathfrak{M}^B$ and $A\in \mathfrak{V}^B$. On the other hand, if $(T_n, H^{(n)})$ and $(T_n,A^{(n)})$ are FCSs for $H\in\mathfrak{M}^{B}$ and $A\in\mathfrak{V}^{B}$ respectively, then $H^{(n)}$ is not necessarily integrable w.r.t $A^{(n)}$ for each $n\in \mathbb{N}^+$. The reason is that the processes $H^{(n)}$ and $A^{(n)}$ include information outside $B$ such that $H^{(n)}$ may not be integrable w.r.t. $A^{(n)}$. We give a simple example.
\begin{example}\label{example_A}
Let $B=\llbracket{0,1}\rrbracket$, $H=1\mathfrak{I}_B$ and $A(\omega,t)=t$ for $(\omega,t)\in B$.
For each $n\in \mathbb{N}^+$ and $(\omega,t)\in \Omega\times \mathbb{R}^+$, put $T_n=1$, $A^{(n)}(\omega,t)=t$ and
\[
H^{(n)}(\omega,t)=I_{\llbracket{0,1}\rrbracket}(\omega,t)+\frac{1}{2-t}I_{\llbracket{1,2}\llbracket}(\omega,t)
+I_{\llbracket{2,+\infty}\llbracket}(\omega,t).
\]
Then $(T_n,H^{(n)})$ is an FCS for $H\in \mathfrak{M}^B$ and $(T_n,A^{(n)})$ is an FCS for $A\in \mathfrak{V}^B$, and for each $n\in \mathbb{N}^+$, $H^{(n)}.A^{(n)}$ does not exist. However, $H$ is integrable on $B$ w.r.t. $A$. To see this, we prove $(ii)$ in Theorem \ref{HA-equivalent}. Putting $\tau_n=1$ for each $n\in \mathbb{N}^+$, $(\tau_n)$ is an FS for $B$. For each $n\in \mathbb{N}^+$, $H^{\tau_n}$ is integrable w.r.t. $A^{\tau_n}$, which is just what we need.
\end{example}

From Theorem \ref{HA-FCS}, the L-S integral $H_{\bullet}A$ on $B$ remains a process on $B$ with finite variation, which is analogous with L-S stochastic integrals by paths. It is also of much significance to study conditions under which the integral $H_{\bullet}A$ becomes an adapted process on $B$, and such study can be applied to stochastic integrals on $B$ of predictable processes w.r.t. semimartingales in Section \ref{section5}.

\begin{theorem}\label{HA-equivalent-p}
Let $H\in\mathcal{P}^{B}$ and $A\in\mathcal{V}^{B}$. Then the following statements are equivalent:
\begin{itemize}
  \item [$(i)$] $H$ is integrable on $B$ w.r.t. $A$.
  \item [$(ii)$] There exists an FS $(\tau_n)$ for $B$ such that for each $n\in \mathbb{N}^+$, $H^{\tau_n}$ is integrable w.r.t. $A^{\tau_n}$.
  \item [$(iii)$] There exist FCSs $(T_n, H^{(n)})$ for $H\in\mathcal{P}^{B}$ and $(T_n,A^{(n)})$ for $A\in\mathcal{V}^{B}$ such that for each $n\in \mathbb{N}^+$, $H^{(n)}$ is integrable w.r.t. $A^{(n)}$.
\end{itemize}
\end{theorem}
\begin{proof}

Using Theorem 3.46 in \cite{He},
the proof is analogous to that of Theorem \ref{HA-equivalent}.
\end{proof}

\begin{remark}\label{HA==}
The condition $(iii)$ in Theorem \ref{HA-equivalent-p} can be changed equivalently to the following condition:
\begin{description}
  \item [$(iii')$] There exist FCSs $(T_n, H^{(n)})$ for $H\in\mathcal{P}^{B}$ and $(S_n,A^{(n)})$ for $A\in\mathcal{V}^{B}$ such that for each $n\in \mathbb{N}^+$, $H^{(n)}$ is integrable w.r.t. $A^{(n)}$.
\end{description}
Suppose the statement $(iii)$ holds. Putting $S_n=T_n$ for each $n\in \mathbb{N}^+$, the statement $(iii')$ is obtained obviously.
On the other hand, suppose the statement $(iii')$ holds. Put $\tau_n=T_n\wedge S_n$ for each $n\in \mathbb{N}^+$. Then from the statement (4) of Theorem \ref{process}, $(\tau_n, H^{(n)})$ is an FCS for $H\in\mathcal{P}^{B}$ and $(\tau_n,A^{(n)})$ is an FCS for $A\in\mathcal{V}^{B}$, which proves $(iii)$.
\end{remark}

\begin{theorem}\label{HA-FCS-p}
Let $H\in\mathcal{P}^{B}$ and $A\in\mathcal{V}^{B}$. Suppose $H$ is integrable on $B$ w.r.t. $A$. Then $H_{\bullet}A\in \mathcal{V}^B$, and the following statements hold:
\begin{itemize}
  \item [$(1)$] If $(T_n, H^{(n)})$ for $H\in\mathcal{P}^{B}$ and $(T_n,A^{(n)})$ for $A\in\mathcal{V}^{B}$ are FCSs such that for each $n\in \mathbb{N}^+$, $H^{(n)}$ is integrable w.r.t. $A^{(n)}$, then $(T_n,H^{(n)}.A^{(n)})$ is an FCS for $H_{\bullet}A\in\mathcal{V}^B$, and $H_{\bullet}A$ can be expressed in the form of \eqref{HA-expression}, where the expression is independent of the choice of FCSs $(T_n, H^{(n)})$ for $H\in\mathcal{P}^{B}$ and $(T_n,A^{(n)})$ for $A\in\mathcal{V}^{B}$.
  \item [$(2)$] If $(\tau_n)$ is an FS for $B$, then $(\tau_n,H^{\tau_n}.A^{\tau_n})$ is an FCS for $H_{\bullet}A\in\mathcal{V}^{B}$, and $H_{\bullet}A$ can be expressed as \eqref{HA-expression0},
      where the expression is independent of the choice of FS $(\tau_n)$ for $B$.
\end{itemize}
\end{theorem}
\begin{proof}
Using Theorem 3.46 in \cite{He}, the proof is analogous to that of Theorem \ref{HA-FCS}.
\end{proof}

Let $A\in\mathcal{V}^{B}$, and $\mathcal{D}_b$ be the class of all bounded process. According to Definition 7.5 in \cite{He}, a process of $\mathcal{D}_{b,loc}$ is said to be locally bounded, and the class of locally bounded predictable processes is $\mathcal{D}_{b,loc}\cap \mathcal{P}$. It can be easily shown that a locally bounded predictable process on $B$ is always integrable w.r.t. $A$, which is the following corollary.

\begin{corollary}\label{bound-HA}
Let $H$ be a locally bounded predictable process on $B$, and $A\in\mathcal{V}^{B}$. Then
$H$ is integrable on $B$ w.r.t. $A$, and both $(T_n,H^{(n)}.A^{(n)})$ and $(\tau_n,H^{\tau_n}.A^{\tau_n})$ are FCSs for $H_{\bullet}A\in\mathcal{V}^B$, where $(T_n,H^{(n)})$ is an FCS for $H$ (a locally bounded predictable process on $B$), and $(T_n,A^{(n)})$ is an FCS for $A\in\mathcal{V}^{B}$, and $(\tau_n)$ is an FS for $B$.
\end{corollary}
\begin{proof}
Suppose that $(T_n, H^{(n)})$ is an FCS for $H$ (a locally bounded predictable process on $B$), and that $(T_n,A^{(n)})$ is an FCS for $A\in\mathcal{V}^{B}$. For each $n\in \mathbb{N}^+$, $H^{(n)}$ is integrable w.r.t. $A^{(n)}$ (see, e.g., Theorem I.4.31 in \cite{Jacod}), and $H^{(n)}.A^{(n)}\in \mathcal{V}$. Then, by Theorems \ref{HA-equivalent-p} and \ref{HA-FCS-p}, $H$ is integrable on $B$ w.r.t. $A$, and $H_{\bullet}A\in\mathcal{V}^B$ with the FCS $(T_n,H^{(n)}.A^{(n)})$. It is not hard to see that $\mathcal{D}_{b,loc}\cap \mathcal{P}$ is stable under stopping and localization. Then from Theorem \ref{fcs-p}, $(\tau_n,H^{\tau_n})$ is an FCS for $H$ (a locally bounded predictable process on $B$), and $(\tau_n,A^{\tau_n})$ is an FCS for $A\in\mathcal{V}^{B}$. Since $H^{\tau_n}$ is integrable w.r.t. $A^{\tau_n}$ for each $n\in \mathbb{N}^+$, we deduce that $(\tau_n,H^{\tau_n}.A^{\tau_n})$ is an FCS for $H_{\bullet}A\in\mathcal{V}^B$ from Theorem \ref{HA-FCS-p}.
\end{proof}

As is shown in Theorems \ref{HA-equivalent} and \ref{HA-FCS}, the L-S integral on a PSIT is essentially characterized by a sequence of L-S integrals. Consequently, there is no doubt that L-S integrals on PSITs have similar properties with L-S integrals.
And we present fundamental properties of L-S integrals on PSITs in the following two theorems.

\begin{theorem}\label{HAproperty}
Let $H,K,\widetilde{H}\in(\mathcal{D}_i)^B$ and $A,V\in(\mathcal{E}_i)^B$ for $i=1,2$, and $a,b\in \mathbb{R}$ be two constants, and $(\tau_n)$ be an FS for $B$, where $(\mathcal{D}_1,\mathcal{E}_1)=(\mathfrak{M},\mathfrak{V})$ and  $(\mathcal{D}_2,\mathcal{E}_2)=(\mathcal{P},\mathcal{V})$. Suppose that both $H$ and $K$ are integrable on $B$ w.r.t. $A$, and that $H$ are integrable on $B$ w.r.t. $V$. Then we have the following statements:
\begin{itemize}
\item[$(1)$] $aH+bK$ is integrable on $B$ w.r.t. $A$, and in this case, we have
  \begin{equation}\label{ab}
   (aH+bK)_{\bullet}A=a(H_{\bullet}A)+b(K_{\bullet}A).
  \end{equation}
  Furthermore, $(\tau_n,(aH^{\tau_n}+bK^{\tau_n}).A^{\tau_n}=a(H ^{\tau_n}.A^{\tau_n})+b(K^{\tau_n}.A^{\tau_n}))$ is an FCS for $(aH+bK)_{\bullet}A\in(\mathcal{E}_i)^{B}$.
\item[$(2)$] $H$ is integrable on $B$ w.r.t. $aA+bV$, and in this case, we have
  \begin{equation}\label{ab2}
   H_{\bullet}(aA+bV)=a(H_{\bullet}A)+b(H_{\bullet}V).
  \end{equation}
  Furthermore, $(\tau_n,H^{\tau_n}.(aA^{\tau_n}+bV^{\tau_n})=a(H ^{\tau_n}.A^{\tau_n})+b(H^{\tau_n}.V^{\tau_n}))$ is an FCS for $H_{\bullet}(aA+bV)\in(\mathcal{E}_i)^{B}$.
\item[$(3)$] $\widetilde{H}$ is integrable on $B$ w.r.t. $H_{\bullet}A$ if and only if $\widetilde{H}H$ is integrable on $B$ w.r.t. $A$.
  Furthermore, if $\widetilde{H}$ is integrable on $B$ w.r.t. $H_{\bullet}A$ (or equivalently, $\widetilde{H}H$ is integrable on $B$ w.r.t. $A$), then
  \begin{equation}\label{ab3}
  (\widetilde{H}H)_{\bullet}A=\widetilde{H}_{\bullet}(H_{\bullet}A),
  \end{equation}
  and $(\tau_n,\widetilde{H}^{\tau_n}.(H^{\tau_n}.A^{\tau_n})=(\widetilde{H}^{\tau_n}H^{\tau_n}).A^{\tau_n})$ is an FCS for $(\widetilde{H}H)_{\bullet}A=\widetilde{H}_{\bullet}(H_{\bullet}A)\in(\mathcal{E}_i)^{B}$.
  \end{itemize}
\end{theorem}
\begin{proof}
$(1)$  For all $(\omega,t)\in B$, we have
\[
\int_{[0,t]}|aH_s(\omega)+bK_s(\omega)||dA_s(\omega)|
\leq|a|\int_{[0,t]}|H_s(\omega)||dA_s(\omega)|+|b|\int_{[0,t]}|K_s(\omega)||dA_s(\omega)|<+\infty,
\]
which, by existence of $H_{\bullet}A$ and $K_{\bullet}A$, implies that $aH+bK$ is integrable on $B$ w.r.t. $A$. From the definitions of $H_{\bullet}A$ and $K_{\bullet}A$, \eqref{ab} can be easily obtained from for all $(\omega,t)\in B$,
\[
\int_{[0,t]}(aH_s(\omega)+bK_s(\omega))dA_s(\omega)
=a\int_{[0,t]}H_s(\omega)dA_s(\omega)+b\int_{[0,t]}K_s(\omega)dA_s(\omega).
\]

From Theorems \ref{HA-FCS} and \ref{HA-FCS-p}, for each $n\in \mathbb{N}^+$, $H^{\tau_n}\in\mathcal{D}_i$ is integrable w.r.t. $A^{\tau_n}\in\mathcal{E}_i$, and $K^{\tau_n}\in\mathcal{D}_i$ is integrable w.r.t. $A^{\tau_n}\in\mathcal{E}_i$. Then for each $n\in \mathbb{N}^+$, $(aH^{\tau_n}+bK^{\tau_n}).A^{\tau_n}=a(H ^{\tau_n}.A^{\tau_n})+b(K^{\tau_n}.A^{\tau_n})\in \mathcal{E}_i$, and from \eqref{ab},
\begin{align*}
((aH+bK)_{\bullet}A)I_{\llbracket{0,\tau_n}\rrbracket}
&=a(H_{\bullet}A)I_{\llbracket{0,\tau_n}\rrbracket}+b(K_{\bullet}A)I_{\llbracket{0,\tau_n}\rrbracket}\\
&=a(H^{\tau_n}.A^{\tau_n})I_{\llbracket{0,\tau_n}\rrbracket}+b(K^{\tau_n}.A^{\tau_n})I_{\llbracket{0,\tau_n}\rrbracket}\\
&=((aH^{\tau_n}+bK^{\tau_n}).A^{\tau_n})I_{\llbracket{0,\tau_n}\rrbracket}.
\end{align*}
Hence, $(\tau_n,(aH^{\tau_n}+bK^{\tau_n}).A^{\tau_n})$ is an FCS for $(aH+bK)_{\bullet}A\in(\mathcal{E}_i)^{B}$.

$(2).$ For all $(\omega,t)\in B$, we have
\[
\int_{[0,t]}|H_s(\omega)||d(aA_s(\omega)+bV_s(\omega))|
\leq|a|\int_{[0,t]}|H_s(\omega)||dA_s(\omega)|+|b|\int_{[0,t]}|H_s(\omega)||dV_s(\omega)|<+\infty,
\]
which, by existence of $H_{\bullet}A$ and $H_{\bullet}V$, implies that $H$ is integrable on $B$ w.r.t. $aA+bV$. By the definitions of $H_{\bullet}A$ and $H_{\bullet}V$, \eqref{ab2} can be easily obtained from for all $(\omega,t)\in B$,
\[
\int_{[0,t]}H_s(\omega)d(aA_s(\omega)+bV_s(\omega))
= a\int_{[0,t]}H_s(\omega)dA_s(\omega)+b\int_{[0,t]}H_s(\omega)dV_s(\omega).
\]

From Theorems \ref{HA-FCS} and \ref{HA-FCS-p}, for each $n\in \mathbb{N}^+$, $H^{\tau_n}\in\mathcal{D}_i$ is integrable w.r.t. $A^{\tau_n}\in\mathcal{E}_i$, and $H^{\tau_n}\in\mathcal{D}_i$ is integrable w.r.t. $V^{\tau_n}\in\mathcal{E}_i$. Then for each $n\in \mathbb{N}^+$, $H^{\tau_n}.(aA^{\tau_n}+bV^{\tau_n})=a(H ^{\tau_n}.A^{\tau_n})+b(H^{\tau_n}.V^{\tau_n}))\in \mathcal{E}_i$, and from \eqref{ab2},
\begin{align*}
(H_{\bullet}(aA+bV))I_{\llbracket{0,\tau_n}\rrbracket}
&=a(H_{\bullet}A)I_{\llbracket{0,\tau_n}\rrbracket}+b(H_{\bullet}V)I_{\llbracket{0,\tau_n}\rrbracket}\\
&=a(H^{\tau_n}.A^{\tau_n})I_{\llbracket{0,\tau_n}\rrbracket}+b(H^{\tau_n}.V^{\tau_n})I_{\llbracket{0,\tau_n}\rrbracket}\\
&=(H^{\tau_n}.(aA^{\tau_n}+bV^{\tau_n}))I_{\llbracket{0,\tau_n}\rrbracket}.
\end{align*}
Hence, $(\tau_n,H^{\tau_n}.(aA^{\tau_n}+bV^{\tau_n}))$ is an FCS for $H_{\bullet}(aA+bV)\in(\mathcal{E}_i)^{B}$.

$(3).$ The first statement can be obtained from the relation that for all $(\omega,t)\in B$,
\begin{equation*}
\int_{[0,t]}|\widetilde{H}_s(\omega)H_s(\omega)||dA_s(\omega)|<+\infty
\quad\Leftrightarrow \quad \int_{[0,t]}|\widetilde{H}_s(\omega)||d\widetilde{A}_s(\omega)| +\infty,
\end{equation*}
where $\widetilde{A}_t(\omega)$ is given by
\[
\widetilde{A}_t(\omega)=\int_{[0,t]}H_s(\omega)dA_s(\omega).
\]

Suppose that $\widetilde{H}H$ is integrable on $B$ w.r.t. $A$. By the definitions of $(\widetilde{H}H)_{\bullet}A$ and $\widetilde{H}_{\bullet}\widetilde{A}$, \eqref{ab3} can be easily obtained from for all $(\omega,t)\in B$,
\[
\int_{[0,t]}\widetilde{H}_s(\omega)H_s(\omega)dA_s(\omega)
=\int_{[0,t]}\widetilde{H}_s(\omega)d\widetilde{A}_s(\omega).
\]
From Theorems \ref{HA-FCS} and \ref{HA-FCS-p}, for each $n\in \mathbb{N}^+$, $(\widetilde{H}H)^{\tau_n}\in\mathcal{D}_i$ is integrable w.r.t. $A^{\tau_n}\in\mathcal{E}_i$, and $H^{\tau_n}\in\mathcal{D}_i$ is integrable w.r.t. $A^{\tau_n}\in\mathcal{E}_i$.
Then for each $n\in \mathbb{N}^+$, $(\widetilde{H}H)^{\tau_n}=\widetilde{H}^{\tau_n}H^{\tau_n}$ and $(\widetilde{H}^{\tau_n}H^{\tau_n}).A^{\tau_n}=\widetilde{H}^{\tau_n}.(H^{\tau_n}.A^{\tau_n})$, which, by Theorems \ref{HA-equivalent-p} and \ref{HA-FCS-p}, implies that $(\tau_n,\widetilde{H}^{\tau_n}.(H^{\tau_n}.A^{\tau_n})=(\widetilde{H}^{\tau_n}H^{\tau_n}).A^{\tau_n})$ is an FCS for $(\widetilde{H}H)_{\bullet}A=\widetilde{H}_{\bullet}(H_{\bullet}A)\in(\mathcal{E}_i)^B$.
\end{proof}

\begin{remark}
Let the conditions in Theorem \ref{HAproperty} hold for $i=1,2$.
\begin{itemize}
  \item [$(1)$] From \eqref{ab}, $(T_n,a(H^{(n)}.A^{(n)})+b(K^{(n)}.\widetilde{A}^{(n)}))$ is also an FCS for $(aH+bK)_{\bullet}A\in(\mathcal{E}_i)^{B}$, where $(T_n,H^{(n)})$ and $(T_n,A^{(n)})$ are FCSs for $H\in(\mathcal{D}_i)^B$ and $A\in(\mathcal{E}_i)^B$ respectively such that for each $n\in \mathbb{N}^+$, $H^{(n)}$ is integrable w.r.t. $A^{(n)}$, and where $(T_n,K^{(n)})$ and $(T_n,\widetilde{A}^{(n)})$ are FCSs for $K\in(\mathcal{D}_i)^B$ and $A\in(\mathcal{E}_i)^B$ respectively such that for each $n\in \mathbb{N}^+$, $K^{(n)}$ is integrable w.r.t. $\widetilde{A}^{(n)}$.
  \item [$(2)$] From \eqref{ab2}, $(T_n,a(H^{(n)}.A^{(n)})+b(\widetilde{H}^{(n)}.V^{(n)}))$ is also an FCS for $H_{\bullet}(aA+bV)\in(\mathcal{E}_i)^{B}$, where $(T_n,H^{(n)})$ and $(T_n,A^{(n)})$ are FCSs for $H\in(\mathcal{D}_i)^B$ and $A\in(\mathcal{E}_i)^B$ respectively such that for each $n\in \mathbb{N}^+$, $H^{(n)}$ is integrable w.r.t. $A^{(n)}$, and where $(T_n,\widetilde{H}^{(n)})$ and $(T_n,V^{(n)})$ are FCSs for $H\in(\mathcal{D}_i)^B$ and $V\in(\mathcal{E}_i)^B$ respectively such that for each $n\in \mathbb{N}^+$, $\widetilde{H}^{(n)}$ is integrable w.r.t. $V^{(n)}$.
  \item [$(3)$] Suppose $\widetilde{H}H$ is integrable on $B$ w.r.t. $A$. Then from \eqref{ab3}, $(T_n,(\widetilde{H}^{(n)}H^{(n)}).A^{(n)}=\widetilde{H}^{(n)}.(H^{(n)}.A^{(n)}))$ is also an FCS for $(\widetilde{H}H)_{\bullet}A=\widetilde{H}_{\bullet}(H_{\bullet}A)\in(\mathcal{E}_i)^{B}$, where $(T_n,\widetilde{H}^{(n)})$, $(T_n,H^{(n)})$ and $(T_n,A^{(n)})$ are FCSs for $\widetilde{H}\in(\mathcal{D}_i)^B$, $H\in(\mathcal{D}_i)^B$ and $A\in(\mathcal{E}_i)^B$ respectively such that for each $n\in \mathbb{N}^+$, both $\widetilde{H}^{(n)}H^{(n)}$ and $H^{(n)}$ are integrable w.r.t. $A^{(n)}$.
\end{itemize}
\end{remark}

\begin{theorem}\label{HAproperty1}
Let $H\in \mathfrak{M}^B$ and $A\in \mathfrak{V}^B$. If $H$ is integrable on $B$ w.r.t. $A$, then we have the following statements:
\begin{itemize}
  \item[$(1)$] $(H_{\bullet}A)I_{\llbracket{0}\rrbracket}=HAI_{\llbracket{0}\rrbracket}$ and $\Delta (H_{\bullet}A)=H\Delta A$.
  \item[$(2)$] For any stopping time $\tau$ on $B$, we have
  \begin{equation}\label{HAtau}
   (H_{\bullet}A)^\tau\mathfrak{I}_B=H_{\bullet}(A^\tau \mathfrak{I}_B)=(H^\tau \mathfrak{I}_B)_{\bullet}(A^\tau \mathfrak{I}_B)=(HI_{\llbracket{0,\tau}\rrbracket}\mathfrak{I}_B)_{\bullet}A.
  \end{equation}
  \end{itemize}
\end{theorem}
\begin{proof}
From Theorem \ref{HA-equivalent}, there exist FCSs $(T_n,H^{(n)})$ for $H\in \mathfrak{M}^B$ and $(T_n,A^{(n)})$ for $A\in \mathfrak{V}^B$ such that for each $n\in \mathbb{N}^+$, $H^{(n)}$ is integrable w.r.t. $A^{(n)}$.

$(1)$ From \eqref{HA-expression}, the first equation is trivial. For each $n\in \mathbb{N}^+$, by Theorems \ref{delta} and \ref{HA-FCS}, we deduce
\[
\Delta (H_{\bullet}A)I_{B\llbracket{0,T_n}\rrbracket}
=\Delta (H^{(n)}.A^{(n)})I_{B\llbracket{0,T_n}\rrbracket}
=(H^{(n)}\Delta A^{(n)})I_{B\llbracket{0,T_n}\rrbracket}
=(H\Delta A)I_{B\llbracket{0,T_n}\rrbracket},
\]
which yields the second equation.

$(2)$ For each $n\in \mathbb{N}^+$, the relation
\begin{equation}\label{HFCS}
(HI_{\llbracket{0,\tau}\rrbracket}\mathfrak{I}_B)I_{B\llbracket{0,T_n}\rrbracket}
=HI_{B\llbracket{0,T_n}\rrbracket}I_{\llbracket{0,\tau}\rrbracket}
=(H^{(n)}I_{\llbracket{0,\tau}\rrbracket})I_{B\llbracket{0,T_n}\rrbracket}
\end{equation}
shows $(T_n,H^{(n)}I_{\llbracket{0,\tau}\rrbracket})$ is a CS for $HI_{\llbracket{0,\tau}\rrbracket}\mathfrak{I}_B$. From $H^{(n)}I_{\llbracket{0,\tau}\rrbracket}\in\mathfrak{M}$ for each $n\in \mathbb{N}^+$, the sequence $(T_n,H^{(n)}I_{\llbracket{0,\tau}\rrbracket})$ is an FCS for $HI_{\llbracket{0,\tau}\rrbracket}\mathfrak{I}_B\in\mathfrak{M}^B$. Then we have the following statements:
\begin{itemize}
  \item [$(a)$] From the existence of $H^{(n)}.A^{(n)}$ and the relation $(H^{(n)}.A^{(n)})^\tau
=H^{(n)}.(A^{(n)})^\tau$ for each $n\in \mathbb{N}^+$, sequences $(T_n,H^{(n)})$ and $(T_n,(A^{(n)})^\tau)$ are FCSs for $H\in\mathfrak{M}^B$ and $A^\tau \mathfrak{I}_B\in\mathfrak{V}^B$ (see Theorem \ref{fcs}) respectively such that $H^{(n)}$ is integrable w.r.t. $(A^{(n)})^\tau$ for each $n\in \mathbb{N}^+$. Then Theorem \ref{HA-equivalent} shows that $H$ is integrable on $B$ w.r.t. $A^\tau \mathfrak{I}_B$. From Theorems \ref{fcs}, \ref{process} and \ref{HA-FCS}, the relations
\begin{align*}
(H_{\bullet}A)^\tau I_{B\llbracket{0,T_n}\rrbracket}
=(H^{(n)}.A^{(n)})^\tau I_{B\llbracket{0,T_n}\rrbracket}
=(H^{(n)}.(A^{(n)})^\tau)I_{B\llbracket{0,T_n}\rrbracket}
=(H_{\bullet}(A^\tau \mathfrak{I}_B))I_{B\llbracket{0,T_n}\rrbracket},\;n\in\mathbb{N}^+
\end{align*}
give the first equality of (\ref{HAtau}).

  \item [$(b)$] From the existence of $H^{(n)}.A^{(n)}$ and the relation $(H^{(n)}.A^{(n)})^\tau
=(H^{(n)})^\tau.(A^{(n)})^\tau$ for each $n\in \mathbb{N}^+$, sequences $(T_n,(H^{(n)})^\tau)$ and $(T_n,(A^{(n)})^\tau)$ are FCSs for $H^\tau \mathfrak{I}_B\in\mathfrak{M}^B$ and $A^\tau \mathfrak{I}_B\in\mathfrak{V}^B$ (see Theorem \ref{fcs}) respectively such that $(H^{(n)})^\tau$ is integrable w.r.t. $(A^{(n)})^\tau$ for each $n\in \mathbb{N}^+$. Then Theorem \ref{HA-equivalent} shows that $H^\tau \mathfrak{I}_B$ is integrable on $B$ w.r.t. $A^\tau \mathfrak{I}_B$. From Theorems \ref{fcs}, \ref{process} and \ref{HA-FCS}, the relations
\begin{align*}
(H_{\bullet}A)^\tau I_{B\llbracket{0,T_n}\rrbracket}
=(H^{(n)}.A^{(n)})^\tau I_{B\llbracket{0,T_n}\rrbracket}
=((H^{(n)})^\tau.(A^{(n)})^\tau)I_{B\llbracket{0,T_n}\rrbracket}
=((H^\tau \mathfrak{I}_B)_{\bullet}(A^\tau \mathfrak{I}_B))I_{B\llbracket{0,T_n}\rrbracket},\;n\in\mathbb{N}^+
\end{align*}
give the second equality of (\ref{HAtau}).

  \item [$(c)$] From the existence of $H^{(n)}.A^{(n)}$ and the relation $(H^{(n)}.A^{(n)})^\tau
=(H^{(n)}I_{\llbracket{0,\tau}\rrbracket}).A^{(n)}$ for each $n\in \mathbb{N}^+$, sequences $(T_n,H^{(n)}I_{\llbracket{0,\tau}\rrbracket})$ and $(T_n,A^{(n)})$ are FCSs for $HI_{\llbracket{0,\tau}\rrbracket}\mathfrak{I}_B\in\mathfrak{M}^B$ (see \eqref{HFCS}) and $A\in\mathfrak{V}^B$ respectively such that $H^{(n)}I_{\llbracket{0,\tau}\rrbracket}$ is integrable w.r.t. $A^{(n)}$ for each $n\in \mathbb{N}^+$. Then Theorem \ref{HA-equivalent} shows that $HI_{\llbracket{0,\tau}\rrbracket}\mathfrak{I}_B$ is integrable on $B$ w.r.t. $A$. From Theorems \ref{fcs}, \ref{process} and \ref{HA-FCS}, the relations
\begin{align*}
(H_{\bullet}A)^\tau I_{B\llbracket{0,T_n}\rrbracket}
=(H^{(n)}.A^{(n)})^\tau I_{B\llbracket{0,T_n}\rrbracket}
=((H^{(n)}I_{\llbracket{0,\tau}\rrbracket}).A^{(n)})I_{B\llbracket{0,T_n}\rrbracket}
=((HI_{\llbracket{0,\tau}\rrbracket}\mathfrak{I}_B)_{\bullet}A)I_{B\llbracket{0,T_n}\rrbracket},\;n\in\mathbb{N}^+
\end{align*}
give the last equality of (\ref{HAtau}).
\end{itemize}
Summarizing, we deduce (\ref{HAtau}).
\end{proof}

\section{Local martingales on PSITs and stochastic integrals on PSITs of predictable processes with respect to local martingales}\label{section4}\noindent
\setcounter{equation}{0}
In this section,  we first extend classic quadratic covariations of two local martingales to those on PSITs, and then develop stochastic integrals on PSITs of predictable processes w.r.t. local martingales.

Let $H\in \mathcal{P}$ and $M\in\mathcal{M}_{\mathrm{loc}}$. Recall that $H$ is integrable w.r.t.  $M$ (see, e.g., Definition 9.1 in \cite{He}) if there exists a (unique) local martingale $L$ such that
\begin{equation}\label{hm}
[L,N]=H.[M,N]
\end{equation}
holds for every $N\in\mathcal{M}_{\mathrm{loc}}$, where $[L,N]$ and $[M,N]$ are the quadratic covariations of local martingales. And the unique $L$, denoted by $H.M$, is called the stochastic integral of $H$ w.r.t. $M$. The collection of all predictable processes which are integrable w.r.t. $M$ is denoted by $\mathcal{L}_m(M)$.

\subsection{Local martingales on PSITs}
We present main properties of local martingales on PSITs in the following two theorems: the former presents a unique decomposition of a local martingale on $B$, and the later considers stopped processes and FCSs relative to local martingales on $B$.

\begin{theorem}\label{th8.23}
Let $M\in (\mathcal{M}_{\mathrm{loc}})^B$. Then $M$ admits a unique decomposition
\begin{equation}\label{con-M}
M=M_0\mathfrak{I}_B+M^c+M^d,
\end{equation}
where $M^c\in (\mathcal{M}^c_{\mathrm{loc},0})^B$ and $M^d\in (\mathcal{M}^d_{\mathrm{loc}})^B$. $M^c$ is called the continuous part of $M$, and $M^d$ is called the purely discontinuous part of $M$.
\end{theorem}
\begin{proof}
The proof can be found in Theorem 8.23 in \cite{He}.
\end{proof}

\begin{theorem}\label{Mcn}
Let $M\in (\mathcal{M}_{\mathrm{loc}})^B$ and $M=M_0\mathfrak{I}_B+M^c+M^d$ where $M^c\in (\mathcal{M}^c_{\mathrm{loc},0})^B$ and $M^d\in (\mathcal{M}^d_{\mathrm{loc}})^B$.
\begin{itemize}
  \item [$(1)$] If $(T_n,M^{(n)})$ is an FCS for $M\in (\mathcal{M}_{\mathrm{loc}})^B$, then $(T_n,(M^{(n)})^c)$ and $(T_n,(M^{(n)})^d)$ are FCSs for $M^c\in(\mathcal{M}^c_{\mathrm{loc},0})^B$ and $M^d\in(\mathcal{M}^d_{\mathrm{loc}})^B$, respectively.
  \item [$(2)$] If $\tau$ is a stopping time on $B$, then we have
  \begin{equation}\label{McS}
  (M^c)^\tau=(M^\tau)^c,\quad (M^d)^\tau=(M^\tau)^d
  \end{equation}
  and
  \begin{equation}\label{McSB}
  (M^c)^\tau\mathfrak{I}_B=(M^\tau)^c\mathfrak{I}_B=(M^\tau\mathfrak{I}_B)^c,\quad (M^d)^\tau\mathfrak{I}_B=(M^\tau)^d\mathfrak{I}_B=(M^\tau\mathfrak{I}_B)^d.
  \end{equation}
  \item [$(3)$] If $(\tau_n)$ is an FS for $B$, then for each $n\in \mathbb{N}^+$,
  \[
  (M^c)^{\tau_n}=(M^{\tau_n})^c,\quad (M^d)^{\tau_n}=(M^{\tau_n})^d,
  \]
  and $(\tau_n,(M^{\tau_n})^c)$ and $(\tau_n,(M^{\tau_n})^d)$ are FCSs for $M^c\in(\mathcal{M}^c_{\mathrm{loc},0})^B$ and $M^d\in(\mathcal{M}^d_{\mathrm{loc}})^B$, respectively.
\end{itemize}
\end{theorem}
\begin{proof}
$(1)$ Assume that $(T_n,M^{(n)})$ is an FCS for $M\in (\mathcal{M}_{\mathrm{loc}})^B$, and that $B$ is given by \eqref{B}.

Fix $n\in \mathbb{N}^+$. Let $B_n=B\llbracket{0,T_n}\rrbracket$, and $(S_m)$ be an announcing sequence for the predictable stopping time $T_F>0$. Then we have
\[
B_n=\bigcup_{m=1}^{+\infty}\llbracket{0,S_m\wedge T_n}\rrbracket,
\]
which, by Theorem \ref{th8.18}, implies that $B_n$ is a PSIT. From the statement $(3)$ of Remark \ref{reprocess}, $M\mathfrak{I}_{B_n}$ and $M^{(n)}\mathfrak{I}_{B_n}$ are both local martingales on $B_n$. Then \eqref{con-M} gives the following decompositions:
\[
\left\{
\begin{aligned}
M\mathfrak{I}_{B_n}&=(M_0\mathfrak{I}_B+M^c+M^d)\mathfrak{I}_{B_n}=M_0\mathfrak{I}_{B_n}+M^c\mathfrak{I}_{B_n}+M^d\mathfrak{I}_{B_n},\\
M\mathfrak{I}_{B_n}&=M_0\mathfrak{I}_{B_n}+(M\mathfrak{I}_{B_n})^c+(M\mathfrak{I}_{B_n})^d,\\
M^{(n)}\mathfrak{I}_{B_n}&=(M_0+(M^{(n)})^c+(M^{(n)})^d)\mathfrak{I}_{B_n}=M_0\mathfrak{I}_{B_n}+(M^{(n)})^c\mathfrak{I}_{B_n}+(M^{(n)})^d\mathfrak{I}_{B_n},\\
M^{(n)}\mathfrak{I}_{B_n}&=M_0\mathfrak{I}_{B_n}+(M^{(n)}\mathfrak{I}_{B_n})^c+(M^{(n)}\mathfrak{I}_{B_n})^d.
\end{aligned}
\right.
\]
Noticing $M\mathfrak{I}_{B_n}=M^{(n)}\mathfrak{I}_{B_n}$ and using the uniqueness of above decompositions (Theorem \ref{th8.23}), we deduce
\begin{equation}\label{tm11}
M^c\mathfrak{I}_{B_n}=(M\mathfrak{I}_{B_n})^c=(M^{(n)}\mathfrak{I}_{B_n})^c=(M^{(n)})^c\mathfrak{I}_{B_n}
\end{equation}
and
\begin{equation}\label{tm12}
M^d\mathfrak{I}_{B_n}=(M\mathfrak{I}_{B_n})^d=(M^{(n)}\mathfrak{I}_{B_n})^d=(M^{(n)})^d\mathfrak{I}_{B_n}.
\end{equation}

Since (\ref{tm11}) and (\ref{tm12}) hold for each $n\in \mathbb{N}^+$, we obtain the relations  $M^cI_{B\llbracket{0,T_n}\rrbracket}=(M^{(n)})^cI_{B\llbracket{0,T_n}\rrbracket}$ and $M^dI_{B\llbracket{0,T_n}\rrbracket}=(M^{(n)})^dI_{B\llbracket{0,T_n}\rrbracket}$. Then $(T_n,(M^{(n)})^c)$ and $(T_n,(M^{(n)})^d)$ are CSs for $M^c$ and $M^d$, respectively. Finally, from $(M^{(n)})^c\in\mathcal{M}^c_{\mathrm{loc},0}$ and $(M^{(n)})^d\in\mathcal{M}^d_{\mathrm{loc}}$ for each $n\in \mathbb{N}^+$, the statement is proved.

$(2)$ From Theorem \ref{fcs}, $M^\tau\in\mathcal{M}_{\mathrm{loc}}$ and it admits a unique decomposition
$M^\tau=M_0+(M^\tau)^c+(M^\tau)^d$. Using $M=M_0+M^c+M^d$, we also have another decomposition $M^\tau=M_0+(M^c)^\tau+(M^d)^\tau$. Hence, (\ref{McS}) is obtained by the unique decomposition of $M^\tau$.

From Theorem \ref{th8.23}, $M^\tau\mathfrak{I}_B\in (\mathcal{M}_{\mathrm{loc}})^B$ admits a unique decomposition $M^\tau\mathfrak{I}_B=M_0\mathfrak{I}_B+(M^\tau\mathfrak{I}_B)^c+(M^\tau\mathfrak{I}_B)^d$. And using (\ref{McS}) and the fact $M^\tau=M_0+(M^\tau)^c+(M^\tau)^d$, we have $M^\tau\mathfrak{I}_B=M_0\mathfrak{I}_B+(M^\tau)^c\mathfrak{I}_B+(M^\tau)^d\mathfrak{I}_B$ and
$M^\tau\mathfrak{I}_B=M_0\mathfrak{I}_B+(M^c)^\tau\mathfrak{I}_B+(M^d)^\tau\mathfrak{I}_B$. Hence, (\ref{McSB}) is obtained by the unique decomposition of $M^\tau\mathfrak{I}_B\in (\mathcal{M}_{\mathrm{loc}})^B$.

$(3)$ From Theorems \ref{fcs} and \ref{fcs-p}, the statement is a direct result of $(1)$ and $(2)$.
\end{proof}

Now, we turn to quadratic covariations of local martingales on PSITs. Recall that, for $M,\;N\in \mathcal{M}_{\mathrm{loc}}$, the quadratic covariation process $[M,N]$ is the unique process $V\in \mathcal{V}^B$ such that $MN-V\in \mathcal{M}_{\mathrm{loc},0}$ and $\Delta V=\Delta M\Delta N$. And such characterization can be extended to quadratic covariations of local martingales on $B$.

\begin{lemma}\label{ex-quad}
Let $M,N\in(\mathcal{M}_{\mathrm{loc}})^B$. Then there exists a unique process $V\in \mathcal{V}^B$ such that $MN-V\in (\mathcal{M}_{\mathrm{loc},0})^B$ and $\Delta V=\Delta M\Delta N$.
\end{lemma}
\begin{proof}
Assume that $B$ is given by \eqref{B}. Without loss of generalization, let $(T_n,M^{(n)})$ and $(T_n,N^{(n)})$ be FCSs for $M\in (\mathcal{M}_{\mathrm{loc}})^B$ and $N\in (\mathcal{M}_{\mathrm{loc}})^B$, respectively.

Firstly, we show that $V\in \mathcal{V}^B$ with the FCS $(T_n,[M^{(n)},N^{(n)}])$, where the process $V$ is defined by
\begin{equation}\label{L=[M,N]}
V=\left(M_0N_0I_{\llbracket{0}\rrbracket}+\sum\limits_{k=1}^{{+\infty}}[M^{(k)},N^{(k)}]I_{\rrbracket{T_{k-1},T_k}
      \rrbracket}\right)\mathfrak{I}_B, \quad T_0=0.
\end{equation}
For any $l,k\in\mathbb{N}^+$ with $k\leq l$, Theorem \ref{process} shows
$M^{(k)}I_{B\llbracket{0,T_k}\rrbracket}=M^{(l)}I_{B\llbracket{0,T_k}\rrbracket}$ and $N^{(k)}I_{B\llbracket{0,T_k}\rrbracket}=N^{(l)}I_{B\llbracket{0,T_k}\rrbracket}$, or equivalently,
\begin{equation}\label{qr-1}
(M^{(k)})^{T_k\wedge(T_F-)}=(M^{(l)})^{T_k\wedge(T_F-)},\quad (N^{(k)})^{T_k\wedge (T_F-)}=(N^{(l)})^{T_k\wedge (T_F-)},
\end{equation}
which, by \eqref{XYT} and Definition 8.2 in \cite{He}, implies that
\begin{align}\label{eq14}
 [M^{(k)},N^{(k)}]I_{B\llbracket{0,T_k}\rrbracket}
=&[M^{(k)},N^{(k)}]^{T_k\wedge (T_F-)}I_{B\llbracket{0,T_k}\rrbracket}\nonumber\\
=&[(M^{(k)})^{T_k\wedge (T_F-)},(N^{(k)})^{T_k\wedge (T_F-)}]I_{B\llbracket{0,T_k}\rrbracket}\nonumber\\
=&[(M^{(l)})^{T_k\wedge (T_F-)},(N^{(l)})^{T_k\wedge (T_F-)}]I_{B\llbracket{0,T_k}\rrbracket}\nonumber\\
=&[M^{(l)},N^{(l)}]^{T_k\wedge (T_F-)}I_{B\llbracket{0,T_k}\rrbracket}\nonumber\\
=&[M^{(l)},N^{(l)}]I_{B\llbracket{0,T_k}\rrbracket}.
\end{align}
Then Remark \ref{remark-cs} shows that $(T_n,[M^{(n)},N^{(n)}])$ is a CS for $V$. Since $[M^{(n)},N^{(n)}]\in \mathcal{V}$ for each $n\in \mathbb{N}^+$, we deduce that $(T_n,[M^{(n)},N^{(n)}])$ is an FCS for $V\in \mathcal{V}^B$.

Secondly, we show that $V$ satisfies $\Delta V=\Delta M\Delta N$.
Theorem \ref{delta} shows that $(T_n,\Delta[M^{(n)},N^{(n)}])$ is a CS for $\Delta V$. Then the relations
\[
\Delta VI_{B\llbracket{0,T_n}\rrbracket}=\Delta[M^{(n)},N^{(n)}]I_{B\llbracket{0,T_n}\rrbracket}
=(\Delta M^{(n)}\Delta N^{(n)})I_{B\llbracket{0,T_n}\rrbracket}
=(\Delta M\Delta N)I_{B\llbracket{0,T_n}\rrbracket},\quad n\in \mathbb{N}^+
\]
show $\Delta V=\Delta M\Delta N$.

Thirdly, we show $MN-V\in (\mathcal{M}_{\mathrm{loc},0})^B$ with the FCS $(T_n,M^{(n)}N^{(n)}-[M^{(n)},N^{(n)}])$.
For each $n\in \mathbb{N}^+$, by (\ref{eq14}), we have
\begin{align*}
(MN-V)I_{B\llbracket{0,T_n}\rrbracket}&=\sum\limits_{k=1}^{n}(M^{(k)}N^{(k)}-[M^{(k)},N^{(k)}])I_{B\rrbracket{T_{k-1},T_k}
      \rrbracket}\\
      &=\sum\limits_{k=1}^{n}(M^{(n)}N^{(n)}-[M^{(n)},N^{(n)}])I_{B\rrbracket{T_{k-1},T_k}
      \rrbracket}\\
      &=(M^{(n)}N^{(n)}-[M^{(n)},N^{(n)}])I_{B\llbracket{0,T_n}\rrbracket},
\end{align*}
which implies $(T_n,M^{(n)}N^{(n)}-[M^{(n)},N^{(n)}])$ is a CS for $MN-V$.
Since $M^{(n)}N^{(n)}-[M^{(n)},N^{(n)}]\in \mathcal{M}_{\mathrm{loc},0}$ (see, e.g., Theorem 7.31 in \cite{He}) for each $n\in \mathbb{N}^+$, the relation $MN-V\in(\mathcal{M}_{\mathrm{loc},0})^B$ holds true, and $(T_n,M^{(n)}N^{(n)}-[M^{(n)},N^{(n)}])$ is an FCS for $MN-V\in (\mathcal{M}_{\mathrm{loc},0})^B$. Therefore, we obtain the existence of $V$ in the statement.

Finally, we show the uniqueness of $V$. Suppose that there exists another process $\widetilde{V}\in \mathcal{V}^B$ such that $MN-\widetilde{V}\in (\mathcal{M}_{\mathrm{loc},0})^B$ and $\Delta \widetilde{V}=\Delta M\Delta N$. Put $L=V-\widetilde{V}$. Then $L\in (\mathcal{M}_{\mathrm{loc},0})^B\cap\mathcal{V}^B$ and $\Delta L=0$. From Remark \ref{reprocess} and the statement (5) of Theorem \ref{delta}, it follows that
\[
L\in (\mathcal{M}_{\mathrm{loc},0})^B \cap\mathcal{V}^B \cap (\mathcal{C}_0)^B=(\mathcal{M}_{\mathrm{loc},0} \cap\mathcal{V} \cap \mathcal{C}_0)^B,
\]
where we use the fact that the classes $\mathcal{M}_{\mathrm{loc},0}$, $\mathcal{V}$, and $\mathcal{C}_0$ are stable under stopping and localization. Providing $(S_n,L^{(n)})$ is an FCS for $L\in(\mathcal{M}_{\mathrm{loc},0} \cap\mathcal{V} \cap \mathcal{C})^B$, Lemmas I.4.13 and I.4.14 in \cite{Jacod} show $L^{(n)}=0$ for each $n\in \mathbb{N}^+$. Then it is easy to obtain $L=0$, i.e. the uniqueness of $V$.

Summarizing, the process $V$ defined by \eqref{L=[M,N]} is what we need, and we finish the proof.
\end{proof}

Note that $(M^{(k)})^{T_k\wedge (T_F-)}$ and $(N^{(k)})^{T_k\wedge (T_F-)}$ in \eqref{qr-1} remain to be local martingales (see, e.g. Example 9.4 in \cite{He}). Hence, both $[(M^{(k)})^{T_k\wedge S\wedge (T-)},(N^{(k)})^{T_k\wedge (T_F-)}]$ and $[(M^{(l)})^{T_k\wedge(T_F-)},(N^{(l)})^{T_k\wedge (T_F-)}]$ are the quadratic covariations of local martingales.

\begin{definition}\label{de-quad}
Let $M,N\in(\mathcal{M}_{\mathrm{loc}})^B$. The unique process $V\in \mathcal{V}^B$ in Lemma \ref{ex-quad}, denoted by $[M,N]$, is called the quadratic covariation on $B$ of $M$ and $N$. Furthermore, if $M=N$, then $[M,M]$ (or simply, $[M]$) is called the quadratic variation on $B$ of $M$.
\end{definition}

The following theorem shows that $[M,N]$ in Definition \ref{de-quad} is symmetric and bilinear in $M$ and $N$, which is the same as quadratic covariations of local martingales.

\begin{theorem}\label{bilinear}
Let $M,\widetilde{M},N\in(\mathcal{M}_{\mathrm{loc}})^B$ and $a,b\in \mathbb{R}$. Then we have
\[
[M,N]=[N,M],\quad [aM+b\widetilde{M},N]=a[M,N]+b[\widetilde{M},N].
\]
\end{theorem}
\begin{proof}
From Theorem \ref{ex-quad}, the first equation is trivial, and it suffices to prove the second equation. Indeed, by the statement (2) of Theorem \ref{delta},
\[
(aM+b\widetilde{M})N-(a[M,N]+b[\widetilde{M},N])
=a(MN-[M,N])+b(\widetilde{M}N-[\widetilde{M},N])
\in (\mathcal{M}_{\mathrm{loc},0})^B
\]
and
\[
\Delta(a[M,N]+b[\widetilde{M},N])
=a\Delta[M,N]+b\Delta[\widetilde{M},N]
=a\Delta M \Delta N+b\Delta \widetilde{M} \Delta N
=\Delta(aM+b\widetilde{M})\Delta N.
\]
Hence, from Definition \ref{de-quad}, the second equation holds.
\end{proof}

We present fundamental properties of quadratic covariations on PSITs in the following theorem and corollary which play a crucial role in studying stochastic integrals on PSITs of predictable process w.r.t. local martingales.

\begin{theorem}\label{condition-qr}
For $M,\; N\in (\mathcal{M}_{\mathrm{loc}})^B$, we have the following statements:
\begin{itemize}
  \item [$(1)$] If $(T_n,M^{(n)})$ and $(T_n,N^{(n)})$ are FCSs for $M\in (\mathcal{M}_{\mathrm{loc}})^B$ and $N\in (\mathcal{M}_{\mathrm{loc}})^B$ respectively, then the sequence $(T_n,[M^{(n)},N^{(n)}])$ is an FCS for $[M,N]\in \mathcal{V}^B$.
  \item [$(2)$] If $(\tau_n)$ is an FS for $B$, then $(\tau_n,[M^{\tau_n},N^{\tau_n}])$ is an FCS for $[M,N]\in \mathcal{V}^B$.
  \item [$(3)$]  If $\tau$ is a stopping time on $B$, then
\begin{equation}\label{MN}
[M^{\tau},N^{\tau}]
=[M,N]^{\tau}
\end{equation}
and
\begin{equation}\label{MNB}
[M^{\tau}\mathfrak{I}_B,N^{\tau}\mathfrak{I}_B]
=[M^{\tau},N^{\tau}]\mathfrak{I}_B
=[M,N]^{\tau}\mathfrak{I}_B=[M^{\tau}\mathfrak{I}_B,N].
\end{equation}
\end{itemize}
\end{theorem}
\begin{proof}
$(1)$ Let the process $L$ be defined by \eqref{L=[M,N]}. Noticing $[M,N]=L$,
we have shown the statement in the proof of Lemma \ref{ex-quad}.

$(2)$ From Theorem \ref{fcs-p}, $(\tau_n,M^{\tau_n})$ is an FCS for $M\in (\mathcal{M}_{\mathrm{loc}})^B$, and $(\tau_n,N^{\tau_n})$ is an FCS for $N\in (\mathcal{M}_{\mathrm{loc}})^B$. Then using the statement $(1)$, we deduce that $(\tau_n,[M^{\tau_n},N^{\tau_n}])$ is an FCS for $[M,N]\in \mathcal{V}^B$.

$(3)$ From Theorem \ref{fcs}, it is obvious that $M^{\tau},\;N^{\tau}\in \mathcal{M}_{\mathrm{loc}}$. Using Theorem \ref{fcs}, the statement (4) of Theorem \ref{delta} and Definition \ref{de-quad}, we have
\[M^\tau N^\tau-[M,N]^\tau=(MN-[M,N])^\tau\in \mathcal{M}_{\mathrm{loc},0}\]
and
\[\Delta([M,N]^{\tau})=(\Delta[M,N])I_{\llbracket{0,\tau}\rrbracket}
=(\Delta M I_{\llbracket{0,\tau}\rrbracket})(\Delta N I_{\llbracket{0,\tau}\rrbracket})=\Delta M^{\tau}\Delta N^{\tau},\]
and then \eqref{MN} is obtained by the characterization of $[M^{\tau},N^{\tau}]$.

By Definition \ref{de-quad} and the statement (3) of Theorem \ref{delta}, the first equality  of (\ref{MNB}) can be obtained by
 \[
\Delta([M^{\tau},N^{\tau}]\mathfrak{I}_B)=(\Delta[M^{\tau},N^{\tau}])\mathfrak{I}_B
=(\Delta M^{\tau}\Delta N^{\tau})\mathfrak{I}_B
=\Delta(M^{\tau}\mathfrak{I}_B)\Delta(M^{\tau}\mathfrak{I}_B)
\]
and
\[
(M^{\tau}\mathfrak{I}_B)(N^{\tau}\mathfrak{I}_B)-[M^{\tau},N^{\tau}]\mathfrak{I}_B
=(M^{\tau}N^{\tau}-[M^{\tau},N^{\tau}])\mathfrak{I}_B\in (\mathcal{M}_{\mathrm{loc},0})^B.
\]
The second equality of (\ref{MNB}) is the direct result of \eqref{MN}.
As for the last equality of (\ref{MNB}), suppose that $(T_n,M^{(n)})$ and $(T_n,N^{(n)})$ are FCSs for $M\in (\mathcal{M}_{\mathrm{loc}})^B$ and $N\in (\mathcal{M}_{\mathrm{loc}})^B$, respectively. Using Theorem \ref{fcs} and the relation
\begin{align*}
&(M^{(n)})^{\tau}N^{(n)}-(M^{(n)})^{\tau}(N^{(n)})^{\tau}\\
=&\left((M^{(n)})^{\tau}N^{(n)}-[(M^{(n)})^{\tau},N^{(n)}]\right)+\left([(M^{(n)})^{\tau},(N^n)^{\tau}]
-(M^{(n)})^{\tau}(N^{(n)})^{\tau}\right)\in\mathcal{M}_{\mathrm{loc},0},
\end{align*}
we deduce $(M^{\tau}\mathfrak{I}_B)N-(M^{\tau}\mathfrak{I}_B)(N^{\tau}\mathfrak{I}_B)\in (\mathcal{M}_{\mathrm{loc},0})^B$ with the FCS $(T_n,(M^{(n)})^{\tau}N^{(n)}-(M^{(n)})^{\tau}(N^{(n)})^{\tau})$. Since $(M^{\tau}\mathfrak{I}_B)(N^{\tau}\mathfrak{I}_B)-[M^{\tau}\mathfrak{I}_B,N^{\tau}\mathfrak{I}_B]\in (\mathcal{M}_{\mathrm{loc},0})^B$, it is easy to see
\begin{equation}\label{tem4.10-1}
(M^{\tau}\mathfrak{I}_B)N-[M^{\tau}\mathfrak{I}_B,N^{\tau}\mathfrak{I}_B]\in (\mathcal{M}_{\mathrm{loc},0})^B.
\end{equation}
Furthermore, by the statements (3) and (4) of Theorem \ref{delta}, we obtain
\begin{align}
\Delta[M^{\tau}\mathfrak{I}_B,N^{\tau}\mathfrak{I}_B]
&=\Delta (M^{\tau}\mathfrak{I}_B)\Delta (N^{\tau}\mathfrak{I}_B)
=(\Delta M^{\tau}\Delta N^{\tau})\mathfrak{I}_B\nonumber\\
&=(\Delta M^{\tau}\Delta NI_{\llbracket{0,\tau}\rrbracket})\mathfrak{I}_B
=\Delta (M^{\tau}\mathfrak{I}_B)\Delta N,\label{tem4.10-2}
\end{align}
Hence, from Definition \ref{de-quad}, the relations \eqref{tem4.10-1} and \eqref{tem4.10-2} show $[M^{\tau}\mathfrak{I}_B,N^{\tau}\mathfrak{I}_B]=[M^{\tau}\mathfrak{I}_B,N]$, i.e., the last equality of (\ref{MNB}), and we complete the proof.
\end{proof}

\begin{corollary}\label{Aloc}
Let $M\in (\mathcal{M}_{\mathrm{loc}})^B$. Suppose that $(T_n,M^{(n)})$ is an FCS for $M\in (\mathcal{M}_{\mathrm{loc}})^B$, and that $(\tau_n)$ be an FS for $B$. The we have the following statements:
\begin{itemize}
  \item [$(1)$] $[M]\in (\mathcal{V}^+)^B$, and both $(T_n,[M^{(n)}])$ and $(\tau_n,[M^{\tau_n}])$ are FCSs for $[M]\in (\mathcal{V}^+)^B$.
  \item [$(2)$] $\sqrt{[M]}\in (\mathcal{A}^+_{\mathrm{loc}})^B$, and both $(T_n,\sqrt{[M^{(n)}]})$ and $(\tau_n,\sqrt{[M^{\tau_n}]})$ are FCSs for $\sqrt{[M]}\in (\mathcal{A}^+_{\mathrm{loc}})^B$.
  \item [$(3)$] $[M]=0\mathfrak{I}_B$ if and only of $M=0\mathfrak{I}_B$.

  \item [$(4)$] If $M=M_0\mathfrak{I}_B+M^c+M^d$ ($M^c\in (\mathcal{M}^c_{\mathrm{loc},0})^B$ and $M^d\in (\mathcal{M}^d_{\mathrm{loc}})^B$) is the decomposition of $M$, then
  \begin{equation}\label{decom-M-eq}
[M]=M_0^{2}\mathfrak{I}_B+[M^c]+[M^d].
\end{equation}
\end{itemize}
\end{corollary}
\begin{proof}
$(1)$ From Theorem \ref{condition-qr}, $(T_n,[M^{(n)}])$ is an FCS for $[M]\in \mathcal{V}^B$, and then $(T_n,[M^{(n)}])$ is also a CS for $[M]$, i.e.,
\begin{equation}\label{[M]}
[M]I_{B\llbracket{0,T_n}\rrbracket}=[M^{(n)}]I_{B\llbracket{0,T_n}\rrbracket},\; n\in \mathbb{N}^+.
\end{equation}
Using $[M^{(n)}]\in \mathcal{V}^+$ (see, e.g., Definition 7.29 in \cite{He}) for each $n\in \mathbb{N}^+$, we deduce $[M]\in (\mathcal{V}^+)^B$ with the FCS $(T_n,[M^{(n)}])$. From Theorem \ref{fcs-p}, $(\tau_n,M^{\tau_n})$ is an FCS for $M\in (\mathcal{M}_{\mathrm{loc}})^B$, and then $(\tau_n,[M^{\tau_n}])$ is also an FCS for $[M]\in (\mathcal{V}^+)^B$.

$(2)$ The equation \eqref{[M]} obviously shows
\begin{equation}\label{eq-Aloc}
\sqrt{[M]}I_{B\llbracket{0,T_n}\rrbracket}=\sqrt{[M^{(n)}]}I_{B\llbracket{0,T_n}\rrbracket},\quad n\in \mathbb{N}^+.
\end{equation}
From Theorem 7.30 in \cite{He}, the relation $\sqrt{[M^{(n)}]}\in \mathcal{A}^+_{\mathrm{loc}}$ holds for each $n\in \mathbb{N}^+$. Therefore, by \eqref{eq-Aloc}, we deduce $\sqrt{[M]}\in (\mathcal{A}^+_{\mathrm{loc}})^B$ with the FCS $(T_n,\sqrt{[M^{(n)}]})$. From Theorem \ref{fcs-p}, $(\tau_n,M^{\tau_n})$ is an FCS for $M\in (\mathcal{M}_{\mathrm{loc}})^B$, and then $(\tau_n,\sqrt{[M^{\tau_n}]})$ is also an FCS for $\sqrt{[M]}\in (\mathcal{A}^+_{\mathrm{loc}})^B$.

$(3)$ The sufficiency is trivial, and it remains to prove the necessity. Let $[M]=0\mathfrak{I}_B$. From the statement $(1)$, $(\tau_n,[M^{\tau_n}])$ is an FCS for $[M]\in(\mathcal{V}^+)^B$. Then for each $n\in \mathbb{N}^+$, we have
\[
[M^{\tau_n}]I_{\llbracket{0,\tau_n}\rrbracket}
=[M]I_{\llbracket{0,\tau_n}\rrbracket}=0,
\]
which yields $[M^{\tau_n}]=0$. For each $n\in \mathbb{N}^+$, Definition 7.29 in \cite{He} indicates $M^{\tau_n}=0$, and then the relation
\[
MI_{\llbracket{0,\tau_n}\rrbracket}=M^{\tau_n}I_{\llbracket{0,\tau_n}\rrbracket}=0
\]
gives $M=0\mathfrak{I}_B$.

$(4)$ From Theorem \ref{Mcn} and the statement $(1)$, $(T_n,[M^{(n)}])$ is an FCS for $[M]\in (\mathcal{V}^+)^B$, $(T_n,[(M^{(n)})^c])$ is an FCS for $[M^c]\in (\mathcal{V}^+)^B$, and $(T_n,[(M^{(n)})^d])$ is an FCS for $[M^d]\in (\mathcal{V}^+)^B$. Then for each $n\in \mathbb{N}^+$,
\begin{align*}
[M]I_{B\llbracket{0,T_n}\rrbracket}
&=[M^{(n)}]I_{B\llbracket{0,T_n}\rrbracket}\\
&=\left((M^{(n)}_0)^{ 2}+[(M^{(n)})^c]+[(M^{(n)})^d]\right)I_{B\llbracket{0,T_n}\rrbracket}\\
&=\left((M_0)^{2}\mathfrak{I}_B+[M^c]+[M^d]\right)I_{B\llbracket{0,T_n}\rrbracket},
\end{align*}
which clearly yields \eqref{decom-M-eq}.
\end{proof}

Similarly, we can define predictable quadratic covariations of locally square integrable martingales on PSITs, and then study their fundamental properties.

\begin{lemma}\label{ex-quad-p}
Let $M,N\in(\mathcal{M}^2_{\mathrm{loc}})^B$. Then there exists a unique process $V\in (\mathcal{A}_{\mathrm{loc}}\cap \mathcal{P})^B$ such that $MN-V\in (\mathcal{M}^2_{\mathrm{loc},0})^B$.
\end{lemma}
\begin{proof}
By Theorem 7.28 in \cite{He}, the proof is analogous with that of Lemma \ref{ex-quad}.
\end{proof}

\begin{definition}\label{de-quad-p}
Let $M,N\in(\mathcal{M}^2_{\mathrm{loc}})^B$. The unique process $V\in (\mathcal{A}_{\mathrm{loc}}\cap \mathcal{P})^B$ in Lemma \ref{ex-quad-p}, denoted by $\langle M,N\rangle$, is called the predictable quadratic covariation on $B$ of $M$ and $N$. Furthermore, if $M=N$, then $\langle M,M\rangle$ (or simply, $\langle M\rangle$) is called the predictable quadratic variation on $B$ of $M$.
\end{definition}

Let $M,N\in(\mathcal{M}^c_{\mathrm{loc}})^B$. Since $\mathcal{M}^c_{\mathrm{loc}}=\mathcal{M}^{2,c}_{\mathrm{loc}}$, the predictable quadratic covariation $\langle M,N\rangle$ and the predictable quadratic variation $\langle M\rangle$ are well defined. The main properties of predictable quadratic covariations on PSITs are presented in the following theorem and corollary.
\begin{theorem}\label{property-qr-p}
Let $M,\; N,\;\widetilde{M}\in (\mathcal{M}^2_{\mathrm{loc}})^B$. Then we have the following statements:
\begin{itemize}
  \item [$(1)$] For $a,b\in \mathbb{R}$, we have
      \[
       \langle M,N\rangle=\langle N,M\rangle,\quad \langle aM+b\widetilde{M},N\rangle=a\langle M,N\rangle+b\langle\widetilde{M},N\rangle.
      \]
  \item [$(2)$] If $(T_n,M^{(n)})$ and $(T_n,N^{(n)})$ are FCSs for $M\in(\mathcal{M}^2_{\mathrm{loc}})^B$ and $N\in(\mathcal{M}^2_{\mathrm{loc}})^B$ respectively, then the sequence $(T_n,\langle M^{(n)},N^{(n)}\rangle)$ is an FCS for $\langle M,N\rangle\in (\mathcal{A}_{\mathrm{loc}}\cap \mathcal{P})^B$.
  \item [$(3)$] If $(\tau_n)$ is an FS for $B$, then $(\tau_n,\langle M^{\tau_n},N^{\tau_n}\rangle)$ is an FCS for $\langle M,N\rangle\in (\mathcal{A}_{\mathrm{loc}}\cap \mathcal{P})^B$.
  \item [$(4)$] $\langle M\rangle\in (\mathcal{A}^{+}_{\mathrm{loc}}\cap \mathcal{P})^B$, and both $(T_n,\langle M^{(n)}\rangle)$ and $(\tau_n,\langle M^{\tau_n}\rangle)$ are FCSs for $\langle M\rangle\in (\mathcal{A}^{+}_{\mathrm{loc}}\cap \mathcal{P})^B$, where $(T_n,M^{(n)})$ is an FCS for $M\in(\mathcal{M}^2_{\mathrm{loc}})^B$ and $(\tau_n)$ is an FS for $B$.
  \item [$(5)$] If $\tau$ is a stopping time on $B$, then we have
      \begin{equation*}
       \langle M^{\tau},N^{\tau}\rangle
       =\langle M,N\rangle^{\tau}
      \end{equation*}
  and
      \begin{equation*}
       \langle M^{\tau}\mathfrak{I}_B,N^{\tau}\mathfrak{I}_B\rangle
      =\langle M^{\tau},N^{\tau}\rangle\mathfrak{I}_B
      =\langle M,N\rangle^{\tau}\mathfrak{I}_B=\langle M^{\tau}\mathfrak{I}_B,N\rangle.
      \end{equation*}
\end{itemize}
\end{theorem}
\begin{proof}
The proofs are analogous to those of Theorems \ref{bilinear} and \ref{condition-qr}, and Corollary \ref{Aloc}.
\end{proof}

\begin{corollary}\label{qr-pc}
Let $M\in(\mathcal{M}^c_{\mathrm{loc}})^B$. Then $\langle M\rangle\in (\mathcal{A}^{+}_{\mathrm{loc}}\cap \mathcal{C})^B$, and both $(T_n,\langle M^{(n)}\rangle)$ and $(\tau_n,\langle M^{\tau_n}\rangle)$ are FCSs for $\langle M\rangle\in (\mathcal{A}^{+}_{\mathrm{loc}}\cap \mathcal{C})^B$, where $(T_n,M^{(n)})$ is an FCS for $M\in(\mathcal{M}^c_{\mathrm{loc}})^B$, and $(\tau_n)$ is an FS for $B$.
\end{corollary}
\begin{proof}
Since $M^{(n)}\in\mathcal{M}^c_{\mathrm{loc}}\subseteq \mathcal{M}^2_{\mathrm{loc}}$ for each $n\in \mathbb{N}^+$, we have $(T_n,\langle M^{(n)}\rangle)$ is an FCS for $\langle M\rangle\in (\mathcal{A}^{+}_{\mathrm{loc}}\cap \mathcal{P})^B$ by Theorem \ref{property-qr-p}. Then $(T_n,\langle M^{(n)}\rangle)$ is a CS for $\langle M\rangle$. On the other hand, the relation $M^{(n)}\in\mathcal{M}^c_{\mathrm{loc}}$ implies $\langle M^{(n)}\rangle\in\mathcal{A}^{+}_{\mathrm{loc}}\cap \mathcal{C}$ (see the remark after Lemma 7.28 in \cite{He}) for each $n\in \mathbb{N}^+$. Therefore, $\langle M\rangle\in (\mathcal{A}^{+}_{\mathrm{loc}}\cap \mathcal{C})^B$, and $(T_n,\langle M^{(n)}\rangle)$ is an FCS for $\langle M\rangle\in (\mathcal{A}^{+}_{\mathrm{loc}}\cap \mathcal{C})^B$.
Furthermore, noticing $(\tau_n, M^{\tau_n})$ is also an FCS for $M\in(\mathcal{M}^c_{\mathrm{loc}})^B$, we deduce that $(\tau_n,\langle M^{\tau_n}\rangle)$ is an FCS for $\langle M\rangle\in (\mathcal{A}^{+}_{\mathrm{loc}}\cap \mathcal{C})^B$.
\end{proof}

Recall that an optional process $\widetilde{X}$ is said to be thin if $[\widetilde{X}\neq 0]$ is a thin set (i.e., $[\widetilde{X}\neq 0]=\bigcup_{n\in\mathbb{N}^+}\llbracket{T_n}\rrbracket$ where $(T_n)$ is a sequence of stopping times). Then a thin process $X$ on $B$ is well defined in the manner of Definition \ref{processB}, and we have the following result.
\begin{lemma}\label{X-thin}
Let $X$ be an optional process on $B$. Then $X$ is a thin process on $B$ if and only if $XI_B$ is a thin process.
\end{lemma}
\begin{proof}
Note that, by \eqref{x-expression}, $XI_B$ is an optional process.

{\it Necessity.} Suppose that $X$ is a thin process on $B$ with an FCS $(T_n,X^{(n)})$. Then for each $n\in \mathbb{N}^+$, $[X^{(n)}\neq 0]$ is a thin set. From the relations
\[
[XI_B\neq 0]=\bigcup_{n\in\mathbb{N}^+}[X^{(n)}I_{B\llbracket{0,T_n}\rrbracket}\neq 0]\subseteq\bigcup_{n\in\mathbb{N}^+}[X^{(n)}\neq 0],
\]
it is easy to see that $[XI_B\neq 0]$ is an optional set contained in a thin set. Hence, by Theorem 3.19 in \cite{He}, $[XI_B\neq 0]$ is a thin set, which implies that $XI_B$ is a thin process.

{\it Sufficiency.} Suppose that $XI_B$ is a thin process. Let $T$ be the debut of $B^c$. Then $X$ is a thin process on $B$, because $(T_n=T,X^{(n)}=XI_B)$ is an FCS for $X$ (a thin process on $B$).
\end{proof}

From Lemma \ref{X-thin}, a thin process $X$ on $B$ is closely related to the thin process $XI_B$. Then we can define a summation process on $B$ of $X$ which is based on the summation process of $XI_B$ (see Definition 7.39 in \cite{He}).

\begin{definition}\label{sum}
Let $X$ be a thin process on $B$. If for all $t>0$,
\[
\sum_{s\leq t}|\widetilde{X}_s|<+\infty\quad  a.s. \quad (\widetilde{X}=XI_B),
\]
then the summation process on $B$ of $X$, denoted by $\Sigma X$, is defined as
\[
\Sigma X:=\left(\sum_{s\leq \cdot}\widetilde{X}_s\right)\mathfrak{I}_B,
\]
or equivalently,
\[
(\Sigma X)(\omega,t):=\sum_{s\leq t}X(\omega,s),\quad (\omega,t)\in B.
\]
\end{definition}

The following theorem presents main properties of thin processes on $B$, and shows the relationship between summation processes on $B$ and associated  FCSs.
\begin{theorem}\label{thin}
For thin processes on $B$, we have the following statements:
\begin{itemize}
  \item [$(1)$] If $X$ is an adapted c\`{a}dl\`{a}g process on $B$, then $\Delta X$ is a thin process on $B$ satisfying $\Sigma (\Delta X)\in \mathcal{V}^B$. Furthermore, if $(T_n,X^{(n)})$ is an FCS for $X$ (an adapted c\`{a}dl\`{a}g process on $B$), then $(T_n,\Delta X^{(n)})$ is an FCS for $\Delta X$ (a thin process on $B$), and $(T_n,\Sigma(\Delta X^{(n)}))$ is an FCS for $\Sigma (\Delta X)\in \mathcal{V}^B$.
  \item [$(2)$] If $X$ and $Y$ are adapted c\`{a}dl\`{a}g processes on $B$, then $\Delta X\Delta Y$ is a thin process on $B$ satisfying $\Sigma (\Delta X\Delta Y)\in \mathcal{V}^B$. Furthermore, if $(T_n,X^{(n)})$ and $(T_n,Y^{(n)})$ are FCSs for $X$ and $Y$ (adapted c\`{a}dl\`{a}g processes on $B$) respectively, then $(T_n,\Sigma(\Delta X^{(n)}\Delta Y^{(n)}))$ is an FCS for $\Sigma (\Delta X\Delta Y)\in \mathcal{V}^B$.
  \item [$(3)$] Let $X$ be a thin process on $B$, and $\tau$ be a stopping time on $B$. Then $XI_{\llbracket{0,\tau}\rrbracket}$ is a thin process, and when $\Sigma X$ is well defined, the following relation holds:
      \begin{equation}\label{sigmaX}
      \Sigma (XI_{\llbracket{0,\tau}\rrbracket})=(\Sigma X)^\tau.
      \end{equation}
\end{itemize}
\end{theorem}
\begin{proof}
$(1)$ Suppose $(T_n,X^{(n)})$ is an FCS for $X$ (an adapted c\`{a}dl\`{a}g process on $B$), where for each $n\in \mathbb{N}^+$, $X^{(n)}$ is an adapted c\`{a}dl\`{a}g process. Then for each $n\in \mathbb{N}^+$, $\Delta X^{(n)}$ is a thin process (see, e.g., Definition 7.39 in \cite{He}), and by the statement $(1)$ of Theorem \ref{delta},
\[
\Delta XI_{B\llbracket{0,T_n}\rrbracket}=\Delta X^{(n)}I_{B\llbracket{0,T_n}\rrbracket}.
\]
Hence, $\Delta X$ is a thin process on $B$ with the FCS $(T_n,\Delta X^{(n)})$.
For each $n\in \mathbb{N}^+$, the relation $\Sigma (\Delta X^{(n)})\in \mathcal{V}$ holds because $\Sigma (\Delta X^{(n)})$ is a summation process (see Definition 7.39 in \cite{He}), and from the definition of $\Sigma (\Delta X)$,
\[
\Sigma (\Delta X)I_{B\llbracket{0,T_n}\rrbracket}=\Sigma (\Delta X^{(n)})I_{B\llbracket{0,T_n}\rrbracket},
\]
which yields $\Sigma (\Delta X)\in \mathcal{V}^B$ with the FCS $(T_n,\Sigma(\Delta X^{(n)}))$.

$(2)$ The proof is analogous to that of $(1)$.

$(3)$ Obviously, $XI_{\llbracket{0,\tau}\rrbracket}=X^{\tau}I_{\llbracket{0,\tau}\rrbracket}$ is an optional process. From Lemma \ref{X-thin}, $[XI_B\neq 0]$ is a thin set, and then the statement that $XI_{\llbracket{0,\tau}\rrbracket}$ is a thin process can be deduced easily by Theorem 3.19 in \cite{He} and the inclusion
\[
[XI_{\llbracket{0,\tau}\rrbracket}\neq 0]\subseteq [XI_B\neq 0].
\]
Suppose that $\Sigma X$ is well defined. By the definition of the summation process, it is not difficult to see
\begin{align*}
(\Sigma X)^\tau
&=\left(\left(\sum_{s\leq \cdot}(XI_B)_{s}\right)\mathfrak{I}_B\right)^\tau
=\left(\sum_{s\leq \cdot}(XI_B)_{s}\right)^\tau\\
&=\sum_{s\leq \cdot}(XI_{\llbracket{0,\tau}\rrbracket})_s
=\Sigma (XI_{\llbracket{0,\tau}\rrbracket})
\end{align*}
which yields \eqref{sigmaX}.
\end{proof}

Utilizing the summation process on $B$ in Definition \ref{sum}, we can present the following relation between quadratic covariations on $B$ and predictable quadratic covariations on $B$.
\begin{theorem}\label{delatM}
Let $M,\;N\in (\mathcal{M}_{\mathrm{loc}})^B$.
Then $[M,N]$ can be expressed as
  \begin{equation}\label{express[M,N]}
[M,N]=M_0N_0\mathfrak{I}_B+\langle M^c,N^c\rangle+\Sigma (\Delta M\Delta N).
  \end{equation}
Specially, $[M]=M_0^2\mathfrak{I}_B+\langle M^c\rangle+\Sigma (\Delta M)^2$.
\end{theorem}
\begin{proof}
Let $(\tau_n)$ be an FS for $B$. Then using Theorems \ref{Mcn}, \ref{condition-qr}, \ref{property-qr-p} and \ref{thin}, and Definition 7.29 in \cite{He}, we deduce that for each $n\in \mathbb{N}^+$,
\begin{align*}
[M,N]I_{\llbracket{0,\tau_n}\rrbracket}
&=[M^{\tau_n},N^{\tau_n}]I_{\llbracket{0,\tau_n}\rrbracket}\\
&=\left(M^{\tau_n}_0N^{\tau_n}_0+\langle (M^{\tau_n})^c,(N^{\tau_n})^c\rangle+\Sigma (\Delta M^{\tau_n}\Delta N^{\tau_n})\right)I_{\llbracket{0,\tau_n}\rrbracket}\\
&=\left(M_0N_0\mathfrak{I}_B+\langle M^c,N^c\rangle+\Sigma (\Delta M\Delta N)\right)I_{\llbracket{0,\tau_n}\rrbracket}
\end{align*}
which yields \eqref{express[M,N]}.
\end{proof}

%%%%%%%%%%%%%%%%%%%%%%%%%%%%%%%%%%%%%%%%%%%%%%%%%%%%%%%%%%%%%%%%%%%%%%%%%%%%%%%%%%%%%%%%%%
\iffalse
\begin{example}
Suppose $W$ is a standard Brownian motion on $(\Omega,\mathcal{F},\mathbb{P})$. Put
\[
T=\inf\{t\geq 0: W_t=1\},\quad \tau_n=\inf\left\{t\geq 0: W_t=1-\frac{1}{n+1}\right\}
\]
and
\[
W^{(1)}=2W^{\tau_1}-W,\quad W^{(n+1)}=2(W^{(n)})^{\tau_{n+1}}-W^{(n)}
\]
for each $n\in \mathbb{N}^+$. Then for each $n\in \mathbb{N}^+$, $T$ and $\tau_n$ predictable times satisfying $\tau_n\uparrow T$, and by reflection principal, $W^{(n)}$ is a standard Brownian motion. By induction, it is easy to see $(W^{(n)})^{\tau_{k}}=(W^{(k)})^{\tau_{k}}$ for $k\leq n$. Put $B=\llbracket{0,T}\llbracket$. $B$ is a predictable set of interval type, and $(\tau_n)$ is an FS for $B$. Define the process $M$ on $B$ as
\[
M=\left(\sum\limits_{n=1}^{{+\infty}}
      W^{(n)}I_{\rrbracket{\tau_{n-1},\tau_n}
      \rrbracket}\right)\mathfrak{I}_B,\quad \tau_0=0.
\]
Then we can prove the following statements:
\begin{itemize}
  \item [$(1)$] $M\in (\mathcal{M}_{\mathrm{loc}}^c)^B$. $(\tau_n,W^{(n)})$ is an FCS for $M\in (\mathcal{M}_{\mathrm{loc}}^c)^B$, and $(\tau_n,W^{\tau_n})$ is an FCS for $M\in (\mathcal{M}_{\mathrm{loc}}^c)^B$.
  \item [$(2)$] $[M]=A\mathfrak{I}_B$, and $(\tau_n,[W^{(n)}])$ is an FCS for $[M]\in (\mathcal{V}^+)^B$, where $A_t=t$.
\end{itemize}
\end{example}
\fi
%%%%%%%%%%%%%%%%%%%%%%%%%%%%%%%%%%%%%%%%%%%%%%%%%%%%%%%%%%%%%%%%%%%%%%%%%%%%%%%%%%%%%%%%%%

Finally, we give an example to show that a sequence of local martingales can be used to construct a general local martingale on $B$.
\begin{example}\label{gen-M}
Assume that $(\tau_n)$ is an FS for $B$, and that $(N^{(n)})$ is a sequence of local martingales. Put
\[
M^{(1)}=N^{(1)}, \quad
M^{(n+1)}=N^{(n+1)}+(M^{(n)}-N^{(n+1)})^{\tau_{n}}, \quad n\in \mathbb{N}^+
\]
and
\begin{equation}\label{gen-eM}
M=\left(M_0^{(1)}I_{\llbracket{0}\rrbracket}+\sum\limits_{n=1}^{{+\infty}}
      M^{(n)}I_{\rrbracket{\tau_{n-1},\tau_n}
      \rrbracket}\right)\mathfrak{I}_B,\quad \tau_0=0.
\end{equation}
Then we have the following statements:
\begin{itemize}
  \item [$(1)$] $M\in(\mathcal{M}_{\mathrm{loc}})^B$, and $(\tau_n,M^{(n)})$ is an FCS for $M\in(\mathcal{M}_{\mathrm{loc}})^B$. By induction, for any $n,k\in \mathbb{N}$ with $k\leq n$, we have $M^{(n)}\in\mathcal{M}_{\mathrm{loc}}$ and $(M^{(k)})^{\tau_{k}}=(M^{(n)})^{\tau_{k}}$. Remark \ref{remark-cs} shows $M\in(\mathcal{M}_{\mathrm{loc}})^B$ with the FCS $(\tau_n,M^{(n)})$.

  \item [$(2)$] From Theorem \ref{delta}, $(\tau_n,\Delta M^{(n)})$ is a CS for $\Delta M$. And from \eqref{x-expression},
  \begin{equation}\label{DM}
  \Delta M=\left(\sum\limits_{n=1}^{{+\infty}}
      \Delta M^{(n)}I_{\rrbracket{\tau_{n-1},\tau_n}
      \rrbracket}\right)\mathfrak{I}_B.
  \end{equation}

  \item [$(3)$] From Theorem \ref{th8.23}, $M$ admits the unique decomposition $M=M_0\mathfrak{I}_B+M^c+M^d$, where $M^c\in(\mathcal{M}^c_{\mathrm{loc},0})^B$ and $M^d\in(\mathcal{M}_{\mathrm{loc}})^B$ are given by
  \begin{align*}
  M^c=\left(\sum\limits_{n=1}^{{+\infty}}
      (M^{(n)})^cI_{\rrbracket{\tau_{n-1},\tau_n}
      \rrbracket}\right)\mathfrak{I}_B, \quad
  M^d=\left(\sum\limits_{n=1}^{{+\infty}}
      (M^{(n)})^dI_{\rrbracket{\tau_{n-1},\tau_n}
      \rrbracket}\right)\mathfrak{I}_B.
  \end{align*}
  From Theorem \ref{Mcn}, $(\tau_n,(M^{(n)})^c)$ is an FCS for $M^c\in(\mathcal{M}^c_{\mathrm{loc},0})^B$, and $(\tau_n,(M^{(n)})^d)$ is an FCS for $M^c\in(\mathcal{M}^d_{\mathrm{loc}})^B$.

  \item [$(4)$] By Corollary \ref{Aloc} and \eqref{x-expression}, $[M]$ can be expressed as
  \[
  [M]=\left(M^2_0I_{\llbracket{0}
      \rrbracket}+\sum\limits_{n=1}^{{+\infty}}
      [M^{(n)}]I_{\rrbracket{\tau_{n-1},\tau_n}
      \rrbracket}\right)\mathfrak{I}_B,
  \]
  and $(\tau_n,[M^{(n)}])$ is an FCS for $[M]\in(\mathcal{V}^+)^B$.
\end{itemize}
\end{example}

\subsection{Stochastic integrals on PSITs of predictable processes with respect to local martingales}
Based on quadratic covariations on PSITs, we extend the definition of $H.M$ in \eqref{hm}, and define stochastic integrals on PSITs of predictable processes with respect to local martingales.

\begin{definition}\label{HM}
Let $H\in \mathcal{P}^B$ and $M\in (\mathcal{M}_{\mathrm{loc}})^B$. If there exists a (unique) process $L\in(\mathcal{M}_{\mathrm{loc}})^B$ such that
\begin{equation}\label{deHM}
[L,N]=H_{\bullet}[M,N]
\end{equation}
holds for every process $N\in(\mathcal{M}_{\mathrm{loc}})^B$ (this naturally implies $H$ is integrable on $B$ w.r.t. $[M,N]$), then we say that $H$ is integrable on $B$ w.r.t. $M$. At this time, the process $L$, denoted by $H_{\bullet}M$, is called the stochastic integral on $B$ of $H$ w.r.t. $M$.
The collection of all predictable processes on $B$ which are integrable on $B$ w.r.t. $M$ is denoted by $\mathcal{L}_m^B(M)$.
\end{definition}

When we use the stochastic integrals $H_{\bullet}M$ and $H_{\bullet}A$, the classes to which $M$ and $V$ belong are always stated. Thus, there is no ambiguity in the two notations $H_{\bullet}M$ of Definition \ref{HM} and $H_{\bullet}A$ of Definition \ref{HA}.
Furthermore, we note that, if the process $L$ in (\ref{deHM}) exists, then it is unique. Indeed, suppose that $\widetilde{L}\in(\mathcal{M}_{\mathrm{loc}})^B$ is another process such that $[\widetilde{L},N]=H_{\bullet}[M,N]$ holds for every $N\in(\mathcal{M}_{\mathrm{loc}})^B$. Putting $N=L-\widetilde{L}\in(\mathcal{M}_{\mathrm{loc}})^B$, the relation $[L-\widetilde{L}]=0$ holds true. By the statement $(3)$ of Corollary \ref{Aloc}, we deduce $L=\widetilde{L}$, i.e., the uniqueness of $L$.

\begin{remark}\label{HMB=HM}
It is of much significance to note that the stochastic integral $H_{\bullet}M$ defined by \ref{deHM} degenerates to the stochastic integral $H.M$ defined by \eqref{hm} if $B=\llbracket{0,+\infty}\llbracket=\Omega\times\mathbb{R}^+$. More precisely, the following relation holds:
\begin{itemize}
  \item [] If $H\in \mathcal{P}^{\llbracket{0,+\infty}\llbracket}$ and $M\in (\mathcal{M}_{\mathrm{loc}})^{\llbracket{0,+\infty}\llbracket}$, then $H_{\bullet}M=H.M$.
\end{itemize}
Indeed, from Lemma \ref{stable} and Corollary \ref{cD=DB}, the relations $\mathcal{P}=\mathcal{P}^{\llbracket{0,+\infty}\llbracket}$ and $\mathcal{M}_{\mathrm{loc}}=(\mathcal{M}_{\mathrm{loc}})^{\llbracket{0,+\infty}\llbracket}$ hold true.
Then using Remark \ref{HAB=HA} and Definition \ref{HM}, it is easy to obtain the above statement.
\end{remark}

The following two theorems enable us to study the relation between $H_{\bullet}M$ in Definition \ref{HM} and related FCSs for $H\in\mathcal{P}$ and $M\in(\mathcal{M}_{\mathrm{loc}})^B$, which plays a crucial role in exploring properties of $H_{\bullet}M$. Theorem \ref{HM=} shows that the integrability on $B$ of $H$ w.r.t. $M$ essentially depends on the classic integrability of predictable processes w.r.t. local martingales. More precisely, Theorem \ref{eq-HM} characterizes the stochastic integral $H_{\bullet}M$ as a summation of a sequence of classic stochastic integrals of predictable processes w.r.t. local martingales.

\begin{theorem}\label{HM=}
Let $H\in \mathcal{P}^B$ and $M\in (\mathcal{M}_{\mathrm{loc}})^B$. Then the following statements are equivalent:
\begin{description}
  \item[$(i)$]$H\in\mathcal{L}_m^B(M)$.
  \item[$(ii)$]There exists an FS $(\tau_n)$ for $B$ satisfying $H^{\tau_n}\in\mathcal{L}_m(M^{\tau_n})$ for each $n\in \mathbb{N}^+$.
  \item[$(iii)$]There exist FCSs $(T_n,H^{(n)})$ for $H\in \mathcal{P}^B$ and $(T_n,M^{(n)})$ for $M\in(\mathcal{M}_{\mathrm{loc}})^B$ satisfying $H^{(n)}\in\mathcal{L}_m(M^{(n)})$ for each $n\in \mathbb{N}^+$.
  \item[$(iv)$]$\sqrt{{H^2}_{\bullet}[M]}\in (\mathcal{A}^+_{\mathrm{loc}})^B$.
\end{description}
\end{theorem}
\begin{proof}
$(i)\Rightarrow (ii).$ Suppose $H\in\mathcal{L}_m^B(M)$. Let $(\tau_n)$ be an FS for $B$, and $N=\widetilde{N}\mathfrak{I}_B$ for an arbitrary $\widetilde{N}\in \mathcal{M}_{\mathrm{loc}}$. Then Remark \ref{reprocess} shows $N\in (\mathcal{M}_{\mathrm{loc}})^B$, and Theorem \ref{fcs-p} shows that $(\tau_n,H^{\tau_n})$, $(\tau_n,M^{\tau_n})$ and $(\tau_n,\widetilde{N}^{\tau_n}=N^{\tau_n})$ are FCSs for $H\in \mathcal{P}^B$, $M\in (\mathcal{M}_{\mathrm{loc}})^B$ and $N\in (\mathcal{M}_{\mathrm{loc}})^B$, respectively. For each $n\in \mathbb{N}^+$, Theorem \ref{HA-FCS-p} and the existence of $H_{\bullet}[M,N]$ yield the existence of $H^{\tau_n}.[M,N]^{\tau_n}$, and the statement $(3)$ of Theorem \ref{condition-qr} yields $[M,N]^{\tau_n}=[M^{\tau_n},N^{\tau_n}]$. Using (\ref{deHM}) and Theorems \ref{HA-FCS-p} and \ref{condition-qr}, we deduce that for each $n\in \mathbb{N}^+$,
\[
[(H_{\bullet}M)^{\tau_n},N^{\tau_n}]I_{\llbracket{0,\tau_n}\rrbracket}
=[H_{\bullet}M,N]I_{\llbracket{0,\tau_n}\rrbracket}
=(H_{\bullet}[M,N])I_{\llbracket{0,\tau_n}\rrbracket}=(H^{\tau_n}.[M^{\tau_n},N^{\tau_n}])I_{\llbracket{0,\tau_n}\rrbracket},
\]
which implies $[(H_{\bullet}M)^{\tau_n},N^{\tau_n}]=H^{\tau_n}.[M^{\tau_n},N^{\tau_n}]$ and
 \begin{equation}\label{eqHM=-1}
[(H_{\bullet}M)^{\tau_n},\widetilde{N}]=[(H_{\bullet}M)^{\tau_n},\widetilde{N}^{\tau_n}]
=H^{\tau_n}.[M^{\tau_n},\widetilde{N}^{\tau_n}]=H^{\tau_n}.[M^{\tau_n},\widetilde{N}].
 \end{equation}
Therefore, for each $n\in \mathbb{N}^+$, $(H_{\bullet}M)^{\tau_n}$ is the local martingale such that (\ref{eqHM=-1}) holds for any $\widetilde{N}\in \mathcal{M}_{\mathrm{loc}}$, which implies $H^{\tau_n}\in\mathcal{L}_m(M^{\tau_n})$.

$(ii)\Rightarrow (i).$ Suppose the statement $(ii)$ holds. Define the process $L$ on $B$ as
\begin{equation}\label{eqHM=-2}
L=\left((H_0M_0)I_{\llbracket{0}\rrbracket}+\sum\limits_{n=1}^{+\infty}(H^{\tau_n}.M^{\tau_n})
I_{\rrbracket{\tau_{n-1},\tau_n}\rrbracket}\right)\mathfrak{I}_B,\quad \tau_0=0,
\end{equation}
and we prove $L$ is what we need in (\ref{deHM}).
For any $k,l\in \mathbb{N}^+$ with $k\leq l$, by Lemma \ref{property}, we deduce
\[
(H^{\tau_k}.M^{\tau_k})I_{\llbracket{0,\tau_k}\rrbracket}=
(H^{\tau_l}.M^{\tau_l})^{\tau_k}I_{\llbracket{0,\tau_k}\rrbracket}
=(H^{\tau_l}.M^{\tau_l})I_{\llbracket{0,\tau_k}\rrbracket}.
\]
Then Remark \ref{remark-cs} shows that $(\tau_n,H^{\tau_n}.M^{\tau_n})$ is a CS for $L$. Since $H^{\tau_n}.M^{\tau_n}\in \mathcal{M}_{\mathrm{loc}}$ for each $n\in \mathbb{N}^+$, we obtain
$L\in (\mathcal{M}_{\mathrm{loc}})^B$ with the FCS $(\tau_n,H^{\tau_n}.M^{\tau_n})$. For any $N\in(\mathcal{M}_{\mathrm{loc}})^B$, Theorem \ref{fcs-p} shows that $(\tau_n,N^{\tau_n})$ is an FCS for $N\in(\mathcal{M}_{\mathrm{loc}})^B$. For each $n\in \mathbb{N}^+$, the fact $H^{\tau_n}\in\mathcal{L}_m(M^{\tau_n})$ implies $[H^{\tau_n}.M^{\tau_n},N^{\tau_n}]
=H^{\tau_n}.[M^{\tau_n},N^{\tau_n}]$. Theorem \ref{condition-qr}  shows that  $(\tau_n,[M^{\tau_n},N^{\tau_n}])$ is an FCS for $[M,N]\in\mathcal{V}^B$, and Theorem \ref{HA-equivalent-p} together with the existence of $H^{\tau_n}.[M^{\tau_n},N^{\tau_n}]$ for each $n\in \mathbb{N}^+$  shows the existence of $H_{\bullet}[M,N]$, and Theorems \ref{HA-FCS-p} and \ref{condition-qr} show that $(\tau_n,H^{\tau_n}.[M^{\tau_n},N^{\tau_n}])$ and $(\tau_n,[H^{\tau_n}.M^{\tau_n},N^{\tau_n}])$ are FCSs for $H_{\bullet}[M,N]\in\mathcal{V}^B$ and $[L,N]\in\mathcal{V}^B$, respectively. Then for each $n\in \mathbb{N}^+$,
\[
[L,N]I_{\llbracket{0,\tau_n}\rrbracket}=[H^{\tau_n}.M^{\tau_n},N^{\tau_n}]I_{\llbracket{0,\tau_n}\rrbracket}
=(H^{\tau_n}.[M^{\tau_n},N^{\tau_n}])I_{\llbracket{0,T_n}\rrbracket}=(H_{\bullet}[M,N])I_{\llbracket{0,\tau_n}\rrbracket},
\]
which implies (\ref{deHM}). Hence, we obtain $(i)$.

$(ii)\Rightarrow (iii)$. Suppose the statement $(ii)$ holds. For each $n\in \mathbb{N}^+$, put $T_n=\tau_n$, $H^{(n)}=H^{\tau_n}$ and $M^{(n)}=M^{\tau_n}$. Then from Theorem \ref{fcs-p}, $(\tau_n,H^{(n)})$ and $(\tau_n,M^{(n)})$ are FCSs for $H\in \mathcal{P}^B$ and $M\in (\mathcal{M}_{\mathrm{loc}})^B$ respectively such that for each $n\in \mathbb{N}^+$, $H^{(n)}\in\mathcal{L}_m(M^{(n)})$ (because of $H^{\tau_n}\in\mathcal{L}_m(M^{\tau_n})$).

$(iii)\Rightarrow (ii).$ Suppose the statement $(iii)$ holds. Let $\widetilde{\tau}_n$ be an FS for $B$. Put $\tau_n=\widetilde{\tau}_n\wedge T_n$ for each $n\in \mathbb{N}^+$. From Theorem \ref{process-FS}, $(\tau_n)$ is an FS for $B$, and $(\tau_n,H^{\tau_n})$ and $(\tau_n,M^{\tau_n})$ are FCSs for $H\in \mathcal{P}^B$ and $M\in (\mathcal{M}_{\mathrm{loc}})^B$ respectively. Using the definition of FCS, for each $n\in \mathbb{N}^+$, we deduce
\[
H^{\tau_n}I_{\llbracket{0,\tau_n}\rrbracket}=HI_{\llbracket{0,\tau_n}\rrbracket}
=(HI_{B\llbracket{0,T_n}\rrbracket})I_{\llbracket{0,\widetilde{\tau}_n}\rrbracket}
=(H^{(n)}I_{B\llbracket{0,T_n}\rrbracket})I_{\llbracket{0,\widetilde{\tau}_n}\rrbracket}
=H^{(n)}I_{\llbracket{0,\tau_n}\rrbracket},
\]
which, by \eqref{XYT1}, implies $H^{\tau_n}=(H^{(n)})^{\tau_n}$. And similarly, we also obtain $M^{\tau_n}=(M^{(n)})^{\tau_n}$ for each $n\in \mathbb{N}^+$. By noticing
\[
H^{(n)}\in\mathcal{L}_m(M^{(n)}), \quad (H^{(n)}.M^{(n)})^{\tau_n}=(H^{(n)})^{\tau_n}.(M^{(n)})^{\tau_n},\; n\in N^{+},
\]
we deduce $H^{\tau_n}\in\mathcal{L}_m(M^{\tau_n})$ for each $n\in \mathbb{N}^+$, thereby proving the statement $(ii)$.

$(ii)\Rightarrow (iv).$ Suppose the statement $(ii)$ holds. We have proved $H\in\mathcal{L}_m^B(M)$. Definition \ref{HM} and Theorem \ref{HAproperty} show the relation
\[
[H_{\bullet}M]=H_{\bullet}[M,H_{\bullet}M]=H_{\bullet}(H_{\bullet}[M,M])={H^2}_{\bullet}[M]
\]
and the existence of ${H^2}_{\bullet}[M]$.
Then from the statement $(2)$ of Corollary \ref{Aloc}, we obtain the statement $(iv)$.

$(iv)\Rightarrow (ii).$ Suppose $\sqrt{{H^2}_{\bullet}[M]}\in (\mathcal{A}^+_{\mathrm{loc}})^B$. Let $(\tau_n)$ be an FS for $B$. From Theorem \ref{fcs-p}, $(\tau_n,\left(\sqrt{{H^2}_{\bullet}[M]}\right)^{\tau_n})$ is an FCS for $\sqrt{{H^2}_{\bullet}[M]}\in (\mathcal{A}^+_{\mathrm{loc}})^B$.  Theorem \ref{HA-FCS-p} implies that ${(H^2)^{\tau_n}}.[M]^{\tau_n}$ exists for each $n\in \mathbb{N}^+$, and that the sequence $(\tau_n,{(H^2)^{\tau_n}}.[M]^{\tau_n})$ is an FCS for ${H^2}_{\bullet}[M]\in \mathcal{V}^B$. For each $n\in \mathbb{N}^+$, by Theorem \ref{condition-qr}, we deduce
\[
\sqrt{{H^2}_{\bullet}[M]}I_{\llbracket{0,\tau_n}\rrbracket}
=\sqrt{{(H^2)^{\tau_n}}.[M]^{\tau_n}}I_{\llbracket{0,\tau_n}\rrbracket}
=\sqrt{{(H^{\tau_n})^2}.[M^{\tau_n}]}I_{\llbracket{0,\tau_n}\rrbracket},
\]
which, by \eqref{XYT1}, implies that $\sqrt{{(H^{\tau_n})^2}.[M^{\tau_n}]}=\left(\sqrt{{H^2}_{\bullet}[M]}\right)^{\tau_n}$. Then the relation $\left(\sqrt{{H^2}_{\bullet}[M]}\right)^{\tau_n}\in\mathcal{A}^+_{\mathrm{loc}}$ shows $\sqrt{{(H^{\tau_n})^2}.[M^{\tau_n}]}\in\mathcal{A}^+_{\mathrm{loc}}$ for each $n\in \mathbb{N}^+$. Therefore, by Theorem 9.2 in \cite{He}, we obtain the statement $(ii)$.
\end{proof}

\begin{remark}\label{HM==}
The condition $(iii)$ in Theorem \ref{HM=} can be changed equivalently to the following condition:
\begin{description}
  \item[$(iii')$] There exist FCSs $(T_n,H^{(n)})$ for $H\in \mathcal{P}^B$ and $(S_n,M^{(n)})$ for $M\in(\mathcal{M}_{\mathrm{loc}})^B$ satisfying $H^{(n)}\in\mathcal{L}_m(M^{(n)})$ for each $n\in \mathbb{N}^+$.
\end{description}
Suppose the statement $(iii)$ holds. Putting $S_n=T_n$ for each $n\in \mathbb{N}^+$, the statement $(iii')$ is valid obviously.
On the other hand, suppose the statement $(iii')$ holds. Put $\tau_n=T_n\wedge S_n$ for each $n\in \mathbb{N}^+$. Then from the statement (3) of Theorem \ref{process}, $(\tau_n, H^{(n)})$ is an FCS for $H\in \mathcal{P}^B$ and $(\tau_n,M^{(n)})$ is an FCS for $M\in(\mathcal{M}_{\mathrm{loc}})^B$, which yields $(iii)$.
\end{remark}

\begin{theorem}\label{eq-HM}
Let $M\in (\mathcal{M}_{\mathrm{loc}})^B$ and $H\in \mathcal{L}_m^B(M)$. Then we have the following statements:
\begin{itemize}
  \item [$(1)$] If $(T_n,H^{(n)})$ for $H\in\mathcal{P}^B$ and $(T_n,M^{(n)})$ for $M\in(\mathcal{M}_{\mathrm{loc}})^B$ are FCSs such that for each $n\in \mathbb{N}^+$, $H^{(n)}\in\mathcal{L}_m(M^{(n)})$, then $(T_n,H^{(n)}.M^{(n)})$ is an FCS for $H_{\bullet}M\in(\mathcal{M}_{\mathrm{loc}})^B$, and $H_{\bullet}M$ can be expressed as
  \begin{equation}\label{HM-expression-1}
      H_{\bullet}M=\left((H_0M_0)I_{\llbracket{0}\rrbracket}+\sum\limits_{n=1}^{{+\infty}}
      (H^{(n)}.M^{(n)})I_{\rrbracket{T_{n-1},T_n}
      \rrbracket}\right)\mathfrak{I}_B,\quad T_0=0.
      \end{equation}
  Furthermore, if $(S_n, \widetilde{H}^{(n)})$ for $H\in\mathcal{P}^{B}$ and $(\widetilde{S}_n,\widetilde{M}^{(n)})$ for $M\in(\mathcal{M}_{\mathrm{loc}})^B$ are FCSs such that for each $n\in \mathbb{N}^+$, $\widetilde{H}^{(n)}\in\mathcal{L}_m(\widetilde{M}^{(n)})$, then $H_{\bullet}M=\widetilde{X}$ where the process $\widetilde{X}$ is given by
      \begin{equation*}
      \widetilde{X}=\left((H_0M_0)I_{\llbracket{0}\rrbracket}
      +\sum\limits_{n=1}^{{+\infty}}(\widetilde{H}^{(n)}.\widetilde{M}^{(n)})
      I_{\rrbracket{\widetilde{T}_{n-1},\widetilde{T}_n}\rrbracket}\right)\mathfrak{I}_B,\quad \widetilde{T}_0=0,
      \end{equation*}
      and $\widetilde{T}_n=S_n\wedge \widetilde{S}_n,\;n\in \mathbb{N}^+$. In this case, we say that the expression of \eqref{HM-expression-1} is independent of the choice of FCSs $(T_n, H^{(n)})$ for $H\in\mathcal{P}^{B}$ and $(T_n,M^{(n)})$ for $M\in(\mathcal{M}_{\mathrm{loc}})^B$.
  \item [$(2)$] If $(\tau_n)$ is an FS for $B$, then $(\tau_n,H^{\tau_n}.M^{\tau_n})$ is an FCS for $H_{\bullet}M\in(\mathcal{M}_{\mathrm{loc}})^B$, and $H_{\bullet}M$ can be expressed as
  \begin{equation}\label{HM-expression-p}
      H_{\bullet}M=\left((H_0M_0)I_{\llbracket{0}\rrbracket}+\sum\limits_{n=1}^{{+\infty}}
      (H^{\tau_n}.M^{\tau_n})I_{\rrbracket{\tau_{n-1},\tau_n}
      \rrbracket}\right)\mathfrak{I}_B,\quad \tau_0=0.
      \end{equation}
  Furthermore, if $(\widetilde{\tau}_n)$ is also an FS for $B$, then $H_{\bullet}M=\widetilde{X}$ where the process $\widetilde{X}$ is given by
      \begin{equation*}
      \widetilde{X}=\left((H_0M_0)I_{\llbracket{0}\rrbracket}+\sum\limits_{n=1}^{{+\infty}}
      (H^{\widetilde{\tau}_n}.M^{\widetilde{\tau}_n})I_{\rrbracket{\widetilde{\tau}_{n-1},\widetilde{\tau}_n}
      \rrbracket}\right)\mathfrak{I}_B,\quad \widetilde{\tau}_0=0.
      \end{equation*}
      In this case, we say that the expression of \eqref{HM-expression-p} is independent of the choice of FS $(\tau_n)$ for $B$.
\end{itemize}
\end{theorem}
\begin{proof}
$(1)$ Define the process $L$ on $B$ as
\begin{equation*}
      L=\left((H_0M_0)I_{\llbracket{0}\rrbracket}+\sum\limits_{n=1}^{{+\infty}}
      (H^{(n)}.M^{(n)})I_{\rrbracket{T_{n-1},T_n}
      \rrbracket}\right)\mathfrak{I}_B,\quad T_0=0.
      \end{equation*}

We first prove that $(T_n,H^{(n)}.M^{(n)})$ is an FCS for $L\in(\mathcal{M}_{\mathrm{loc}})^B$.
Let $(\tau_n)$ be an FS for $B$, and $\tau_0=0$. For every $n\in \mathbb{N}^+$ and $i\in \mathbb{N}^+$, using FCSs for $H\in \mathcal{P}^B$, we have the relation
\begin{align*}
H^{(n)}I_{\llbracket{0,T_n}\rrbracket}I_{\llbracket{0,\tau_i}\rrbracket}
=(H^{(n)}I_{B\llbracket{0,T_n}\rrbracket})I_{\llbracket{0,\tau_i}\rrbracket}
=(HI_{B\llbracket{0,T_n}\rrbracket})I_{\llbracket{0,\tau_i}\rrbracket}
=HI_{\llbracket{0,T_n}\rrbracket}I_{\llbracket{0,\tau_i}\rrbracket}
\end{align*}
which, by \eqref{XYT1}, implies $(H^{(n)})^{T_n\wedge \tau_i}=H^{T_n\wedge \tau_i}$. And similarly, we also deduce $(M^{(n)})^{T_n\wedge \tau_i}=M^{T_n\wedge \tau_i}$ for every $n\in \mathbb{N}^+$ and $i\in \mathbb{N}^+$. For $i, k, n\in\mathbb{N}^+$ with $n\leq k$, by Lemma \ref{property}, it is easy to see
\begin{align*}
(H^{(n)}.M^{(n)})I_{\llbracket{0,T_n}\rrbracket}I_{\llbracket{0,\tau_i}\rrbracket}
=&(H^{T_n\wedge \tau_i}.M^{T_n\wedge \tau_i})I_{\llbracket{0,T_n}\rrbracket}I_{\llbracket{0,\tau_i}\rrbracket}\\
=&(H^{T_k\wedge \tau_i}.M^{T_k\wedge \tau_i})^{T_n}I_{\llbracket{0,T_n}\rrbracket}I_{\llbracket{0,\tau_i}\rrbracket}\\
=&(H^{(k)}.M^{(k)})I_{\llbracket{0,T_n}\rrbracket}I_{\llbracket{0,\tau_i}\rrbracket}.
\end{align*}
Using the fact
\[
B=\bigcup\limits_{i=1}^{{+\infty}}\llbracket{0,\tau_i}\rrbracket = \llbracket{0}\rrbracket\cup \left(\bigcup\limits_{i=1}^{{+\infty}}\rrbracket{\tau_{i-1},\tau_i}\rrbracket\right), \]
we deduce that for $k, n\in\mathbb{N}^+$ with $n\leq k$,
\begin{align*}
&(H^{(n)}.M^{(n)})I_{B\llbracket{0,T_n}\rrbracket}\\
=&H_0M_0I_{\llbracket{0}\rrbracket}+\sum_{i=1}^{\infty}(H^{(n)}.M^{(n)})I_{\llbracket{0,T_n}\rrbracket}I_{\rrbracket{\tau_{i-1},\tau_i}\rrbracket}\\
=&H_0M_0I_{\llbracket{0}\rrbracket}+\sum_{i=1}^{\infty}\left((H^{(n)}.M^{(n)})I_{\llbracket{0,T_n}\rrbracket}I_{\llbracket{0,\tau_i}\rrbracket}
-(H^{(n)}.M^{(n)})I_{\llbracket{0,T_n}\rrbracket}I_{\llbracket{0,\tau_{i-1}}\rrbracket}\right)\\
=&H_0M_0I_{\llbracket{0}\rrbracket}+\sum_{i=1}^{\infty}\left((H^{(k)}.M^{(k)})I_{\llbracket{0,T_n}\rrbracket}I_{\llbracket{0,\tau_i}\rrbracket}
-(H^{(k)}.M^{(k)})I_{\llbracket{0,T_n}\rrbracket}I_{\llbracket{0,\tau_{i-1}}\rrbracket}\right)\\
=&H_0M_0I_{\llbracket{0}\rrbracket}+\sum_{i=1}^{\infty}(H^{(k)}.M^{(k)})I_{\llbracket{0,T_n}\rrbracket}I_{\rrbracket{\tau_{i-1},\tau_i}\rrbracket}\\
=&(H^{(k)}.M^{(k)})I_{B\llbracket{0,T_n}\rrbracket}.
\end{align*}
Then Remark \ref{remark-cs} shows that $(T_n,H^{(n)}.M^{(n)})$ is a CS for $L$. Since $H^{(n)}.M^{(n)}\in \mathcal{M}_{\mathrm{loc}}$ for each $n\in \mathbb{N}^+$, we obtain that
$(T_n,H^{(n)}.M^{(n)})$ is an FCS for $L\in(\mathcal{M}_{\mathrm{loc}})^B$.

Next, we show the FCS $(T_n,H^{(n)}.M^{(n)})$ for $H_{\bullet}M\in(\mathcal{M}_{\mathrm{loc}})^B$ and the expression \eqref{HM-expression-1}.  Let $S_n=T_n\wedge \tau_n$ for each $n\in \mathbb{N}^+$, and  $N\in(\mathcal{M}_{\mathrm{loc}})^B$ be arbitrary. From Theorems \ref{process} and \ref{process-FS}, sequences $(S_n,H^{(n)})$, $(S_n,M^{(n)})$, $(S_n,H^{(n)}.M^{(n)})$ and $(S_n,N^{S_n})$ are FCSs for $H\in\mathcal{P}^B$, $M\in(\mathcal{M}_{\mathrm{loc}})^B$, $L\in(\mathcal{M}_{\mathrm{loc}})^B$ and $N\in(\mathcal{M}_{\mathrm{loc}})^B$, respectively.
Then using Theorems \ref{HA-FCS-p} and \ref{condition-qr}, we deduce
\begin{align*}
[L,N]I_{B\llbracket{0,S_n}\rrbracket}
=&[H^{(n)}.M^{(n)},N^{S_n}]I_{B\llbracket{0,S_n}\rrbracket}\\
=&H^{(n)}.[M^{(n)},N^{S_n}]I_{B\llbracket{0,S_n}\rrbracket}\\
=&(H_{\bullet}[M,N])I_{B\llbracket{0,S_n}\rrbracket}
\end{align*}
for each $n\in \mathbb{N}^+$ and $N\in(\mathcal{M}_{\mathrm{loc}})^B$, which, by \eqref{deHM}, yields $L=H_{\bullet}M$. Hence, $(T_n,H^{(n)}.M^{(n)})$ is an FCS for $H_{\bullet}M\in(\mathcal{M}_{\mathrm{loc}})^B$, and from \eqref{x-expression}, $H_{\bullet}M$ can be expressed as \eqref{HM-expression-1}.

Finally, we show the relation $H_{\bullet}M=\widetilde{X}$.  Suppose that $(S_n, \widetilde{H}^{(n)})$ for $H\in\mathcal{P}^{B}$ and $(\widetilde{S}_n,\widetilde{M}^{(n)})$ for $M\in(\mathcal{M}_{\mathrm{loc}})^B$ are FCSs such that for each $n\in \mathbb{N}^+$, $\widetilde{H}^{(n)}\in\mathcal{L}_m(\widetilde{M}^{(n)})$. From Theorem \ref{process}, $(\widetilde{T}_n, \widetilde{H}^{(n)})$ is an FCS for $H\in\mathcal{P}^{B}$, and $(\widetilde{T}_n,\widetilde{M}^{(n)})$ is an FCS for $M\in(\mathcal{M}_{\mathrm{loc}})^B$. Similarly, we can  prove that $(\widetilde{T}_n,\widetilde{H}^{(n)}.\widetilde{M}^{(n)})$ is an FCS for $H_{\bullet}M\in(\mathcal{M}_{\mathrm{loc}})^B$. Then using the independence property of \eqref{x-expression}, we have $H_{\bullet}M=\widetilde{X}$.

$(2)$ From Theorem \ref{fcs-p}, $(\tau_n,H^{\tau_n})$ is an FCS for $H\in \mathcal{P}^B$, and $(\tau_n,M^{\tau_n})$ is an FCS for $M\in (\mathcal{M}_{\mathrm{loc}})^B$. Using the proof $(i)\Rightarrow (ii)$ of Theorem \ref{HM=}, we have $H^{\tau_n}\in\mathcal{L}_m(M^{\tau_n})$ for each $n\in \mathbb{N}^+$. Then the statements are proved by $(1)$ easily.
\end{proof}

\begin{corollary}\label{bound-HM}
Let $H$ be a locally bounded predictable process on $B$, and $M\in (\mathcal{M}_{\mathrm{loc}})^B$. Then
$H\in \mathcal{L}_m^B(M)$, and both $(T_n,H^{(n)}.M^{(n)})$ and $(\tau_n,H^{\tau_n}.M^{\tau_n})$ are FCSs for $H_{\bullet}M\in(\mathcal{M}_{\mathrm{loc}})^B$, where $(T_n,H^{(n)})$ is an FCS for $H$ (a locally bounded predictable process on $B$), and $(T_n,M^{(n)})$ is an FCS for $M\in(\mathcal{M}_{\mathrm{loc}})^B$, and $(\tau_n)$ is an FS for $B$.
\end{corollary}
\begin{proof}
Suppose that $(T_n, H^{(n)})$ is an FCS for $H$ (a locally bounded predictable process on $B$) and $(T_n,M^{(n)})$ is an FCS for $M\in(\mathcal{M}_{\mathrm{loc}})^{B}$. For each $n\in \mathbb{N}^+$, $H^{(n)}$ is integrable w.r.t. $M^{(n)}$ (see, e.g., Theorem I.4.31 in \cite{Jacod}). Then, from Theorems \ref{HM=} and \ref{eq-HM}, the statements hold true.
\end{proof}

The following two theorems present fundamental properties of the stochastic integral $H_{\bullet}M$ in Definition \ref{HM}. Theorem \ref{HM-p} mainly reveals that $H_{\bullet}M$ admits linear properties as a consequence of symmetric and bilinear properties of quadratic covariations on $B$, and that $H_{\bullet}M$ also has the composite property (see \eqref{hHM}) as a result of Theorem \ref{HAproperty} and its definition. Theorem \ref{HM-property} considers the stochastic integral' jump process, continuous part and purely discontinuous part, and related stopped processes.
\begin{theorem}\label{HM-p}
Let $M,\widetilde{M}\in (\mathcal{M}_{\mathrm{loc}})^B$, and $H, K\in \mathcal{L}_m^B(M)$, and $H\in \mathcal{L}_m^B(\widetilde{M})$, and $(\tau_n)$ be an FS for $B$, and $a,b\in\mathbb{R}$. Then we have following statements:
\begin{itemize}
\item [$(1)$] $aH+bK\in\mathcal{L}_m^B(M)$, and in this case
  \begin{equation}\label{HM+}
      (aH+bK)_{\bullet}M=a(H_{\bullet}M)+b(K_{\bullet}M).
  \end{equation}
  Furthermore, $(\tau_n,(aH^{\tau_n}+bK^{\tau_n}).M^{\tau_n}=a(H^{\tau_n}.M^{\tau_n})+b(K^{\tau_n}.M^{\tau_n}))$ is an FCS for $(aH+bK)_{\bullet}M\in(\mathcal{M}_{\mathrm{loc}})^B$.

\item [$(2)$] $H\in\mathcal{L}_m^B(aM+b\widetilde{M})$, and in this case
  \begin{equation}\label{+HM}
      H_{\bullet}(aM+b\widetilde{M})=a(H_{\bullet}M)+b(H_{\bullet}\widetilde{M}).
  \end{equation}
  Furthermore, $(\tau_n,H^{\tau_n}.(aM^{\tau_n}+b\widetilde{M}^{\tau_n})=a(H^{\tau_n}.M^{\tau_n})+b(H^{\tau_n}.\widetilde{M}^{\tau_n}))$ is an FCS for $H_{\bullet}(aM+b\widetilde{M})\in(\mathcal{M}_{\mathrm{loc}})^B$.
  \item [$(3)$] $\widetilde{H}\in\mathcal{L}_m^B(H_{\bullet}M)\Leftrightarrow \widetilde{H}H\in\mathcal{L}_m^B(M)$. Furthermore, if $\widetilde{H}\in\mathcal{L}_m^B(H_{\bullet}M)$ (or equivalently, $\widetilde{H}H\in\mathcal{L}_m^B(M)$), then
  \begin{equation}\label{hHM}
     \widetilde{H}_{\bullet}(H_{\bullet}M)=(\widetilde{H}H)_{\bullet}M,
  \end{equation}
  and $(\tau_n,\widetilde{H}^{\tau_n}.(H^{\tau_n}.M^{\tau_n})=(\widetilde{H}^{\tau_n}H^{\tau_n}).M^{\tau_n})$ is an FCS for $(\widetilde{H}H)_{\bullet}M=\widetilde{H}_{\bullet}(H_{\bullet}M)\in(\mathcal{M}_{\mathrm{loc}})^B$.
\end{itemize}
\end{theorem}
\begin{proof}
$(1)$ From $H,\;K\in \mathcal{L}_m^B(M)$, for every process $N\in(\mathcal{M}_{\mathrm{loc}})^B$, we have
\[
[H_{\bullet}M,N]=H_{\bullet}[M,N],\quad [K_{\bullet}M,N]=K_{\bullet}[M,N],
\]
which, by Theorems \ref{HAproperty} and \ref{bilinear}, yields
\[
[a(H_{\bullet}M)+b(K_{\bullet}M),N]=(aH+bK)_{\bullet}[M,N].
\]
Then Definition \ref{HM} yields $aH+bK\in\mathcal{L}_m^B(M)$ and \eqref{HM+}.

From $aH+bK\in\mathcal{L}_m^B(M)$, Theorem \ref{eq-HM} implies that $(\tau_n,(aH+bK)^{\tau_n}.M^{\tau_n})$ is an FCS for $(aH+bK)_{\bullet}M\in(\mathcal{M}_{\mathrm{loc}})^B$.
Hence, by noticing $(aH+bK)^{\tau_n}=aH^{\tau_n}+bK^{\tau_n}$ for each $n\in \mathbb{N}^+$, we deduce that $(\tau_n,(aH^{\tau_n}+bK^{\tau_n}).M^{\tau_n})$ is an FCS for $(aH+bK)_{\bullet}M\in(\mathcal{M}_{\mathrm{loc}})^B$.

$(2)$ From $H\in \mathcal{L}_m^B(M)$ and $H\in \mathcal{L}_m^B(\widetilde{M})$, for every process $N\in(\mathcal{M}_{\mathrm{loc}})^B$, we have
\[
[H_{\bullet}M,N]=H_{\bullet}[M,N],\quad [H_{\bullet}\widetilde{M},N]=H_{\bullet}[\widetilde{M},N],
\]
which, by Theorems \ref{HAproperty} and \ref{bilinear}, yields
\[
[a(H_{\bullet}M)+b(H_{\bullet}\widetilde{M}),N]=H_{\bullet}[aM+b\widetilde{M},N].
\]
Then Definition \ref{HM} implies $H\in\mathcal{L}_m^B(aM+b\widetilde{M})$ and \eqref{+HM}.

From $aH+bK\in\mathcal{L}_m^B(M)$, Theorem \ref{eq-HM} implies that $(\tau_n,H^{\tau_n}.(aM+b\widetilde{M})^{\tau_n})$ is an FCS for $H_{\bullet}(aM+b\widetilde{M})\in(\mathcal{M}_{\mathrm{loc}})^B$.
Hence, by noticing $(aM+b\widetilde{M})^{\tau_n}=aM^{\tau_n}+b\widetilde{M}^{\tau_n}$ for each $n\in \mathbb{N}^+$, we deduce that $(\tau_n,H^{\tau_n}.(aM^{\tau_n}+b\widetilde{M}^{\tau_n}))$ is an FCS for $H_{\bullet}(aM+b\widetilde{M})\in(\mathcal{M}_{\mathrm{loc}})^B$.

$(3)$ Suppose that $N\in(\mathcal{M}_{\mathrm{loc}})^B$ is arbitrary.
Then by Theorem \ref{HAproperty} and Definition \ref{HM}, the equivalence can be obtained as follows:
\begin{align}
&\widetilde{H}\in\mathcal{L}_m^B(H_{\bullet}M)\nonumber\\
\Leftrightarrow\quad& [\widetilde{H}_{\bullet}(H_{\bullet}M),N]
=\widetilde{H}_{\bullet}[H_{\bullet}M,N]=\widetilde{H}_{\bullet}(H_{\bullet}[M,N])
=(\widetilde{H}H)_{\bullet}[M,N]\label{equi}\\
\Leftrightarrow\quad&\widetilde{H}H\in\mathcal{L}_m^B(M)\nonumber.
\end{align}

Suppose $\widetilde{H}H\in\mathcal{L}_m^B(M)$. \eqref{hHM} is easily obtained by \eqref{equi}.
Theorem \ref{eq-HM} implies that $(\tau_n,(\widetilde{H}H)^{\tau_n}.M^{\tau_n})$ is an FCS for $(\widetilde{H}H)_{\bullet}M\in(\mathcal{M}_{\mathrm{loc}})^B$.
Hence, by noticing $(\widetilde{H}H)^{\tau_n}=\widetilde{H}^{\tau_n}H^{\tau_n}$ for each $n\in \mathbb{N}^+$, we deduce that $(\tau_n,\widetilde{H}^{\tau_n}.(H^{\tau_n}.M^{\tau_n})=(\widetilde{H}^{\tau_n}H^{\tau_n}).M^{\tau_n})$ is an FCS for $(\widetilde{H}H)_{\bullet}M=\widetilde{H}_{\bullet}(H_{\bullet}M)\in(\mathcal{M}_{\mathrm{loc}})^B$.
\end{proof}

\begin{remark}
Let the conditions in Theorem \ref{HM-p} hold.
\begin{itemize}
  \item [$(1)$] From \eqref{HM+}, $(T_n,a(H^{(n)}.M^{(n)})+b(K^{(n)}.\widetilde{M}^{(n)}))$ is also an FCS for $(aH+bK)_{\bullet}M\in(\mathcal{M}_{\mathrm{loc}})^{B}$, where $(T_n,H^{(n)})$ and $(T_n,M^{(n)})$ are FCSs for $H\in\mathcal{P}^B$ and $M\in(\mathcal{M}_{\mathrm{loc}})^B$ respectively such that for each $n\in \mathbb{N}^+$, $H^{(n)}\in \mathcal{L}_m(M^{(n)})$, and where $(T_n,K^{(n)})$ and $(T_n,\widetilde{M}^{(n)})$ are FCSs for $K\in\mathcal{P}^B$ and $M\in(\mathcal{M}_{\mathrm{loc}})^B$ respectively such that for each $n\in \mathbb{N}^+$, $K^{(n)}\in \mathcal{L}_m(\widetilde{M}^{(n)})$.
  \item [$(2)$] From \eqref{+HM}, $(T_n,a(H^{(n)}.M^{(n)})+b(\widetilde{H}^{(n)}.N^{(n)}))$ is also an FCS for $H_{\bullet}(aM+bN)\in(\mathcal{M}_{\mathrm{loc}})^{B}$, where $(T_n,H^{(n)})$ and $(T_n,M^{(n)})$ are FCSs for $H\in\mathcal{P}^B$ and $M\in(\mathcal{M}_{\mathrm{loc}})^B$ respectively such that for each $n\in \mathbb{N}^+$, $H^{(n)}\in \mathcal{L}_m(M^{(n)})$, and where $(T_n,\widetilde{H}^{(n)})$ and $(T_n,N^{(n)})$ are FCSs for $H\in\mathcal{P}^B$ and $N\in(\mathcal{M}_{\mathrm{loc}})^B$ respectively such that for each $n\in \mathbb{N}^+$, $\widetilde{H}^{(n)}\in \mathcal{L}_m(N^{(n)})$.
  \item [$(3)$] Suppose $\widetilde{H}H\in\mathcal{L}_m^B(M)$. Then from \eqref{hHM}, $(T_n,(\widetilde{H}^{(n)}H^{(n)}).M^{(n)}=\widetilde{H}^{(n)}.(H^{(n)}.M^{(n)}))$ is also an FCS for $(\widetilde{H}H)_{\bullet}M=\widetilde{H}_{\bullet}(H_{\bullet}M)\in(\mathcal{M}_{\mathrm{loc}})^{B}$, where $(T_n,\widetilde{H}^{(n)})$, $(T_n,H^{(n)})$ and $(T_n,M^{(n)})$ are FCSs for $\widetilde{H}\in\mathcal{P}^B$, $H\in\mathcal{P}^B$ and $M\in(\mathcal{M}_{\mathrm{loc}})^B$ respectively such that for each $n\in \mathbb{N}^+$, $H^{(n)}\in \mathcal{L}_m(M^{(n)})$ and $\widetilde{H}^{(n)}H^{(n)}\in \mathcal{L}_m(M^{(n)})$.
\end{itemize}
\end{remark}

\begin{theorem}\label{HM-property}
Let $M\in (\mathcal{M}_{\mathrm{loc}})^B$, $H\in \mathcal{L}_m^B(M)$ and $M=M_0+M^c+M^d$ where $M^c\in (\mathcal{M}^c_{\mathrm{loc},0})^B$ and $M^d\in (\mathcal{M}^d_{\mathrm{loc}})^B$. Then we have the following statements:
\begin{itemize}
  \item [$(1)$] $\Delta (H_{\bullet}M)=H\Delta M$, and $(H_{\bullet}M)I_{\llbracket{0}\rrbracket}=HMI_{\llbracket{0}\rrbracket}$.
  \item [$(2)$] $\mathcal{L}^B_m(M)=\mathcal{L}_m^B(M^c)\bigcap\mathcal{L}_m^B(M^d)$, $H_{\bullet}M^c\in(\mathcal{M}^c_{\mathrm{loc}})^B$, $H_{\bullet}M^d\in(\mathcal{M}^d_{\mathrm{loc}})^B$, $(H_{\bullet}M)^c=H_{\bullet}M^c$, and $(H_{\bullet}M)^d=H_{\bullet}M^d$.
  \item [$(3)$] $(H_{\bullet}M)^\tau\mathfrak{I}_{B}=H_{\bullet}(M^\tau\mathfrak{I}_{B})
       =(H^\tau\mathfrak{I}_{B})_{\bullet}(M^\tau\mathfrak{I}_{B})
      =(HI_{\llbracket{0,\tau}\rrbracket}\mathfrak{I}_{B})_{\bullet}M$,
  where $\tau$ is a stopping time on $B$.
\end{itemize}
\end{theorem}
\begin{proof}
Suppose $(\tau_n)$ is an FS for $B$.

$(1)$ From \eqref{HM-expression-p}, the statement of $(H_{\bullet}M)I_{\llbracket{0}\rrbracket}=HMI_{\llbracket{0}\rrbracket}$ can be obtained easily. Using Theorems \ref{delta} and \ref{eq-HM}, we deduce that for each $n\in \mathbb{N}^+$,
\begin{align*}
\Delta (H_{\bullet}M)I_{\llbracket{0,\tau_n}\rrbracket}
=\Delta (H^{\tau_n}.M^{\tau_n})I_{\llbracket{0,\tau_n}\rrbracket}
=(H^{\tau_n}\Delta M^{\tau_n})I_{\llbracket{0,\tau_n}\rrbracket}
=(H\Delta M)I_{\llbracket{0,\tau_n}\rrbracket},
\end{align*}
which, by the statement $(1)$ of Theorem \ref{process-FS}, yields $\Delta (H_{\bullet}M)=H\Delta M$.

$(2)$ Theorem \ref{Mcn} shows that $(T_n,(M^{\tau_n})^c)$ is an FCS for $M^c\in(\mathcal{M}^c_{\mathrm{loc}})^B$ satisfying $(M^{\tau_n})^c=(M^c)^{\tau_n}$ for each $n\in \mathbb{N}^+$. For each $n\in \mathbb{N}^+$, from $H^{\tau_n}\in\mathcal{L}_m(M^{\tau_n})$ (Theorem \ref{eq-HM}), it is easy to obtain the relation
\[
H^{\tau_n}.(M^c)^{\tau_n}=H^{\tau_n}.(M^{\tau_n})^c=(H^{\tau_n}.M^{\tau_n})^c\in\mathcal{M}^c_{\mathrm{loc}},
\]
which, by Theorem \ref{HM=}, implies that $H\in\mathcal{L}_m^B(M^c)$ and $H_{\bullet}M^c\in(\mathcal{M}^c_{\mathrm{loc}})^B$. And the equality $(H_{\bullet}M)^c=H_{\bullet}M^c$ can be obtained by
\begin{equation*}
(H_{\bullet}M^c)I_{\llbracket{0,\tau_n}\rrbracket}
=(H^{\tau_n}.(M^c)^{\tau_n})I_{\llbracket{0,\tau_n}\rrbracket}
=(H^{\tau_n}.M^{\tau_n})^c I_{\llbracket{0,\tau_n}\rrbracket}
=(H_{\bullet}M)^c I_{\llbracket{0,\tau_n}\rrbracket},\quad n\in \mathbb{N}^+.
\end{equation*}
The statements of $H_{\bullet}M^d\in(\mathcal{M}^d_{\mathrm{loc}})^B$ and $(H_{\bullet}M)^d=H_{\bullet}M^d$ can be obtained similarly.

Finally, we prove $\mathcal{L}^B_m(M)=\mathcal{L}_m^B(M^c)\bigcap\mathcal{L}_m^B(M^d)$. From above proof, we have obtained the inclusion $\mathcal{L}^B_m(M)\subseteq\mathcal{L}_m^B(M^c)\bigcap\mathcal{L}_m^B(M^d)$. On the other hand, supposing $K\in\mathcal{L}_m^B(M^c)\bigcap\mathcal{L}_m^B(M^d)$, by Theorem \ref{HM-p}, we have $K\in\mathcal{L}^B_m(M)$, which implies $\mathcal{L}^B_m(M)\supseteq\mathcal{L}_m^B(M^c)\bigcap\mathcal{L}_m^B(M^d)$. Hence, the proof of (2) is completed.

$(3)$  From $H\in \mathcal{L}_m^B(M)$, we can suppose that $(T_n,H^{(n)})$ for $H\in\mathcal{P}^B$ and $(T_n,M^{(n)})$ for $M\in(\mathcal{M}_{\mathrm{loc}})^B$ are FCSs such that for each $n\in \mathbb{N}^+$, $H^{(n)}\in\mathcal{L}_m(M^{(n)})$. Then we have the following statements:
\begin{itemize}
  \item [$(a)$] From $H^{(n)}\in\mathcal{L}_m(M^{(n)})$ and the relation $(H^{(n)}.M^{(n)})^\tau
=H^{(n)}.(M^{(n)})^\tau$ (see Lemma \ref{property}) for each $n\in \mathbb{N}^+$, sequences $(T_n,H^{(n)})$ and $(T_n,(M^{(n)})^\tau)$ are FCSs for $H\in\mathcal{P}^B$ and $M^\tau \mathfrak{I}_B\in(\mathcal{M}_{\mathrm{loc}})^B$ (see Theorem \ref{fcs}) respectively such that
$H^{(n)}\in \mathcal{L}_m((M^{(n)})^\tau)$ for each $n\in \mathbb{N}^+$. Then Theorem \ref{HM=} shows $H\in \mathcal{L}_m^B(M^\tau \mathfrak{I}_B)$. From Theorem \ref{fcs}, \ref{process} and \ref{eq-HM}, the relations
\begin{align*}
(H_{\bullet}M)^\tau I_{B\llbracket{0,T_n}\rrbracket}
=(H^{(n)}.M^{(n)})^\tau I_{B\llbracket{0,T_n}\rrbracket}
=(H^{(n)}.(M^{(n)})^\tau)I_{B\llbracket{0,T_n}\rrbracket}
=(H_{\bullet}(M^\tau \mathfrak{I}_B))I_{B\llbracket{0,T_n}\rrbracket},\;n\in\mathbb{N}^+
\end{align*}
give $(H_{\bullet}M)^\tau\mathfrak{I}_{B}=H_{\bullet}(M^\tau\mathfrak{I}_{B})$.

  \item [$(b)$] From $H^{(n)}\in\mathcal{L}_m(M^{(n)})$ and the relation $(H^{(n)}.M^{(n)})^\tau
=(H^{(n)})^\tau.(M^{(n)})^\tau$ (see Lemma \ref{property}) for each $n\in \mathbb{N}^+$, sequences $(T_n,(H^{(n)})^\tau)$ and $(T_n,(M^{(n)})^\tau)$ are FCSs for $H^\tau \mathfrak{I}_B\in\mathcal{P}^B$ and $M^\tau \mathfrak{I}_B\in(\mathcal{M}_{\mathrm{loc}})^B$ (see Theorem \ref{fcs}) respectively such that $(H^{(n)})^\tau\in \mathcal{L}_m((M^{(n)})^\tau)$ for each $n\in \mathbb{N}^+$. Then Theorem \ref{HM=} shows $H^\tau\in \mathcal{L}_m^B(M^\tau \mathfrak{I}_B)$. From Theorem \ref{fcs}, \ref{process} and \ref{eq-HM}, the relations
\begin{align*}
(H_{\bullet}M)^\tau I_{B\llbracket{0,T_n}\rrbracket}
=(H^{(n)}.M^{(n)})^\tau I_{B\llbracket{0,T_n}\rrbracket}
=((H^{(n)})^\tau.(M^{(n)})^\tau)I_{B\llbracket{0,T_n}\rrbracket}
=((H^\tau \mathfrak{I}_B)_{\bullet}(M^\tau \mathfrak{I}_B))I_{B\llbracket{0,T_n}\rrbracket},\;n\in\mathbb{N}^+
\end{align*}
give $(H_{\bullet}M)^\tau\mathfrak{I}_{B}=(H^\tau\mathfrak{I}_{B})_{\bullet}(M^\tau\mathfrak{I}_{B})$.

  \item [$(c)$] From $H^{(n)}\in\mathcal{L}_m(M^{(n)})$ and the relation $(H^{(n)}.M^{(n)})^\tau
=(H^{(n)}I_{\llbracket{0,\tau}\rrbracket}).M^{(n)}$ (see Lemma \ref{property}) for each $n\in \mathbb{N}^+$, sequences $(T_n,H^{(n)}I_{\llbracket{0,\tau}\rrbracket})$ and $(T_n,M^{(n)})$ are FCSs for $HI_{\llbracket{0,\tau}\rrbracket}\mathfrak{I}_B\in\mathcal{P}^B$ (because of \eqref{HFCS} and $H^{(n)}I_{\llbracket{0,\tau}\rrbracket}\in\mathcal{P}$ for each $n\in \mathbb{N}^+$) and $M\in(\mathcal{M}_{\mathrm{loc}})^B$ respectively such that $H^{(n)}I_{\llbracket{0,\tau}\rrbracket}\in \mathcal{L}_m(M^{(n)})$ for each $n\in \mathbb{N}^+$. Then Theorem \ref{HM=} shows $HI_{\llbracket{0,\tau}\rrbracket}\mathfrak{I}_B\in \mathcal{L}_m^B(M)$. From Theorem \ref{fcs}, \ref{process} and \ref{eq-HM}, the relations
\begin{align*}
(H_{\bullet}M)^\tau I_{B\llbracket{0,T_n}\rrbracket}
=(H^{(n)}.M^{(n)})^\tau I_{B\llbracket{0,T_n}\rrbracket}
=((H^{(n)}I_{\llbracket{0,\tau}\rrbracket}).M^{(n)})I_{B\llbracket{0,T_n}\rrbracket}
=((HI_{\llbracket{0,\tau}\rrbracket}\mathfrak{I}_B)_{\bullet}M)I_{B\llbracket{0,T_n}\rrbracket},\;n\in\mathbb{N}^+
\end{align*}
give $(H_{\bullet}M)^\tau\mathfrak{I}_{B}
      =(HI_{\llbracket{0,\tau}\rrbracket}\mathfrak{I}_{B})_{\bullet}M$.
\end{itemize}
Summarizing, we deduce $(3)$.
\end{proof}

Finally, we give a general example of the stochastic integral $H_{\bullet}M$ defined in Definition \ref{HM}.
\begin{example}
Let $\widetilde{H}$ be a locally bounded predictable process, $H=\widetilde{H}\mathfrak{I}_B$, and $M\in(\mathcal{M}_{\mathrm{loc}})^B$ be given by \eqref{gen-eM} in Example \ref{gen-M}. It is easy to see that $H$ is a locally bounded predictable process on $B$ with the FCS $(\tau_n,\widetilde{H})$.
Then we have the following statements:
\begin{itemize}
  \item [$(1)$] From Corollary \ref{bound-HM}, $H\in \mathcal{L}_m^B(M)$, and $(\tau_n,\widetilde{H}.M^{(n)})$ is an FCS for $H_{\bullet}M\in(\mathcal{M}_{\mathrm{loc}})^B$.
      By Theorem \ref{eq-HM}, $H_{\bullet}M$ can be expressed as
      \begin{equation*}
      H_{\bullet}M=\left((H_0M_0)I_{\llbracket{0}\rrbracket}+\sum\limits_{n=1}^{{+\infty}}
      (\widetilde{H}.M^{(n)})I_{\rrbracket{\tau_{n-1},\tau_n}
      \rrbracket}\right)\mathfrak{I}_B.
      \end{equation*}

  \item [$(2)$] From \eqref{DM} and the statement (1) of Theorem \ref{HM-property},
  \[
  \Delta (H_{\bullet}M)=H\Delta M=\left(\sum\limits_{n=1}^{{+\infty}}
      \widetilde{H}\Delta M^{(n)}I_{\rrbracket{\tau_{n-1},\tau_n}
      \rrbracket}\right)\mathfrak{I}_B.
  \]
  Equivalently, above expression of $\Delta (H_{\bullet}M)$ can be also obtained by using \eqref{x-expression} and the CS $(\tau_n,\Delta(\widetilde{H}.M^{(n)})=\widetilde{H}\Delta M^{(n)})$ for $\Delta (H_{\bullet}M)$ (see Theorem \ref{delta}).

  \item [$(3)$] From Theorem \ref{HM-property} and Example \ref{gen-M}, $H\in \mathcal{L}_m^B(M^c)$ and $H\in \mathcal{L}_m^B(M^d)$ satifying
  \begin{equation*}
  \left\{
  \begin{aligned}
  H_{\bullet}M^c&=\left(\sum\limits_{n=1}^{{+\infty}}
      (\widetilde{H}.(M^{(n)})^c)I_{\rrbracket{\tau_{n-1},\tau_n}
      \rrbracket}\right)\mathfrak{I}_B
      =\left(\sum\limits_{n=1}^{{+\infty}}
      (\widetilde{H}.M^{(n)})^cI_{\rrbracket{\tau_{n-1},\tau_n}
      \rrbracket}\right)\mathfrak{I}_B
      =(H_{\bullet}M)^c,\\
  H_{\bullet}M^d&=\left(\sum\limits_{n=1}^{{+\infty}}
      (\widetilde{H}.(M^{(n)})^d)I_{\rrbracket{\tau_{n-1},\tau_n}
      \rrbracket}\right)\mathfrak{I}_B
      =\left(\sum\limits_{n=1}^{{+\infty}}
      (\widetilde{H}.M^{(n)})^dI_{\rrbracket{\tau_{n-1},\tau_n}
      \rrbracket}\right)\mathfrak{I}_B
      =(H_{\bullet}M)^d.
  \end{aligned}
  \right.
  \end{equation*}

  \item [$(4)$] Let $\tau$ be a stopping time on $B$. Using the definition of $(H_{\bullet}M)^\tau$ (or Theorem \ref{fcs}), it is easy to see
      \begin{align*}
      (H_{\bullet}M)^\tau&=(H_0M_0)I_{\llbracket{0}\rrbracket}+\sum\limits_{n=1}^{{+\infty}}
      (\widetilde{H}.M^{(n)})^\tau I_{\rrbracket{\tau_{n-1},\tau_n}\rrbracket}\\
      &=(H_0M_0)I_{\llbracket{0}\rrbracket}+\sum\limits_{n=1}^{{+\infty}}
      (\widetilde{H}.M^{(n)}) I_{\rrbracket{\tau_{n-1}\wedge\tau,\tau_n\wedge\tau}\rrbracket}.
      \end{align*}
      Then from Theorem \ref{HM-property},
      \begin{align*}
      (H_{\bullet}M)^\tau\mathfrak{I}_{B}&=H_{\bullet}(M^\tau\mathfrak{I}_{B})
      =(H^\tau\mathfrak{I}_{B})_{\bullet}(M^\tau\mathfrak{I}_{B})
      =(HI_{\llbracket{0,\tau}\rrbracket}\mathfrak{I}_{B})_{\bullet}M\\
      &=\left((H_0M_0)I_{\llbracket{0}\rrbracket}+\sum\limits_{n=1}^{{+\infty}}
      (\widetilde{H}.M^{(n)}) I_{\rrbracket{\tau_{n-1}\wedge\tau,\tau_n\wedge\tau}\rrbracket}\right)\mathfrak{I}_B.
      \end{align*}
\end{itemize}
\end{example}

\section{Semimartingales on PSITs and stochastic integrals on PSITs of predictable processes with respect to semimartingales}\label{section5}\noindent
\setcounter{equation}{0}
In this section, we investigate semimartingales on PSITs and stochastic integrals on PSITs of predictable processes with respect to semimartingales. And then It\^{o}'s formula for semimartingales on PSITs is developed from such stochastic integrals and their fundamental properties.

Let $H\in \mathcal{P}$ and $X\in \mathcal{S}$. Recall that $H$ is $X$-integrable if there exists a decomposition $X=M+A\; (M\in \mathcal{M}_{\mathrm{loc}},A\in \mathcal{V}_0)$ such that $H\in \mathcal{L}_m(M)$ and $H.A$ exists. And in this case, the stochastic integral of $H$ w.r.t. $X$, denoted by $H.X$, is defined by
\begin{equation}\label{0-HX}
H.X=H.M+H.A.
\end{equation}
The collection of all predictable processes which are integrable w.r.t. $X$ is denoted by $\mathcal{L}(X)$.

\subsection{Semimartingales on PSITs}
The following two theorems consider the decomposition of a semimartingale on $B$. Theorem \ref{XMA} shows that a semimartingale on $B$ can be equivalently defined as a summation of a local martingale on $B$ and an adapted process on $B$ with finite variation, and such a decomposition is analogous to that of a semimartingale. Theorem \ref{uXc} introduces another decomposition of a semimartingale on $B$ which enables us to define the continuous part of a semimartingale on $B$.

\begin{theorem}\label{XMA}
Let $X$ be a process on $B$. Then $X\in \mathcal{S}^B$ if and only if $X$ admits a decomposition
    \begin{equation}\label{eq-XMA}
      X=M+A,
    \end{equation}
where $M\in (\mathcal{M}_{\mathrm{loc}})^B$ and $A\in (\mathcal{V}_0)^B$.
\end{theorem}
\begin{proof}
{\it Sufficiency}. Suppose that $X$ admits the decomposition \eqref{eq-XMA}. Let $(\tau_n)$ be an FS for $B$. From Theorem \ref{fcs-p}, $(\tau_n,M^{\tau_n})$ is an FCS for $M\in (\mathcal{M}_{\mathrm{loc}})^B$, and $(\tau_n,A^{\tau_n})$ is an FCS for $A\in (\mathcal{V}_0)^B$. Then for each $n\in \mathbb{N}^+$, $X^{\tau_n}=M^{\tau_n}+A^{\tau_n}\in \mathcal{S}$ holds, and the relation
\[
X^{\tau_n}I_{\llbracket{0,\tau_n}\rrbracket}=XI_{\llbracket{0,\tau_n}\rrbracket}
\]
implies that $X\in \mathcal{S}^B$ with the FCS $(\tau_n,X^{\tau_n})$.

{\it Necessity}. Suppose $X\in \mathcal{S}^B$. Let $(\tau_n)$ be an FS for $B$. Then from Theorem \ref{fcs-p}, $(\tau_n,X^{\tau_n})$ is an FCS for $X\in \mathcal{S}^B$. For each $n\in \mathbb{N}^+$, $X^{\tau_n}$ admits a decomposition $X^{\tau_n}=M^{(n)}+A^{(n)}$ with $M^{(n)}\in \mathcal{M}_{\mathrm{loc}}$ and $A^{(n)}\in \mathcal{V}_0$.
Put $\widetilde{M}^{(1)}=(M^{(1)})^{\tau_1}, \widetilde{A}^{(1)}=(A^{(1)})^{\tau_1}$ and for $n\in \mathbb{N}^+$,
\begin{equation}\label{MMAA}
\left\{
\begin{aligned}
\widetilde{M}^{(n+1)}&=\widetilde{M}^{(n)}+(M^{(n+1)})^{\tau_{n+1}}
-(M^{(n+1)})^{\tau_{n}},\\
\widetilde{A}^{(n+1)}&=\widetilde{A}^{(n)}+(A^{(n+1)})^{\tau_{n+1}}
-(A^{(n+1)})^{\tau_{n}}.
\end{aligned}
\right.
\end{equation}
For any $n,k\in \mathbb{N}^+$ with $n\leq k$, by induction, we deduce   $X^{\tau_n}=\widetilde{M}^{(n)}+\widetilde{A}^{(n)}$ ($\widetilde{M}^{(n)}\in\mathcal{M}_{\mathrm{loc}}$ and $\widetilde{A}^{(n)}\in\mathcal{V}_0$) and
\begin{equation}\label{tm2}
(\widetilde{M}^{(k)})^{\tau_n}=(\widetilde{M}^{(n)})^{\tau_n},\quad (\widetilde{A}^{(k)})^{\tau_n}=(\widetilde{A}^{(n)})^{\tau_n}.
\end{equation}
Define the following two processes on $B$:
\begin{equation}\label{MA-p}
\left\{
\begin{aligned}
 M&=\left(X_0I_{\llbracket{0}\rrbracket}+\sum\limits_{n=1}^{{+\infty}}\widetilde{M}^{(n)}I_{\rrbracket{\tau_{n-1},\tau_n}
      \rrbracket}\right)\mathfrak{I}_B,\\
 A&=\left(\sum\limits_{n=1}^{{+\infty}}\widetilde{A}^{(n)}I_{\rrbracket{\tau_{n-1},\tau_n}
      \rrbracket}\right)\mathfrak{I}_B,\quad \tau_0=0.
\end{aligned}
\right.
\end{equation}
Using (\ref{tm2}) and Remark \ref{remark-cs}, we obtain $M\in (\mathcal{M}_{\mathrm{loc}})^B$ with the FCS $(\tau_n,\widetilde{M}^{(n)})$ and $A\in (\mathcal{V}_0)^B$ with the FCS $(\tau_n,\widetilde{A}^{(n)})$. Then for each $k\in \mathbb{N}^+$, the relation
\begin{align*}
(M+A)I_{\llbracket{0,\tau_k}\rrbracket}
&=X_0I_{\llbracket{0}\rrbracket}+\sum\limits_{n=1}^{k}(\widetilde{M}^{(n)}+\widetilde{A}^{(n)})I_{\rrbracket{\tau_{n-1},\tau_n}
      \rrbracket}\\
&=X_0I_{\llbracket{0}\rrbracket}+\sum\limits_{n=1}^{k}X^{\tau_n}I_{\rrbracket{\tau_{n-1},\tau_n}
      \rrbracket}\\
&=X_0I_{\llbracket{0}\rrbracket}+\sum\limits_{n=1}^{k}XI_{\rrbracket{\tau_{n-1},\tau_n}
      \rrbracket}\\
&=XI_{\llbracket{0,\tau_k}\rrbracket}
\end{align*}
gives $X=M+A$ with $M\in (\mathcal{M}_{\mathrm{loc}})^B$ and $A\in (\mathcal{V}_0)^B$.
\end{proof}

In general, the decomposition \eqref{eq-XMA} of a semimartingale on $B$ is not unique. To see this, suppose that $\widetilde{X}\in \mathcal{S}$ (e.g., a martingale with integrable variation, see Definition 6.1 in \cite{He}) admits two different decompositions $\widetilde{X}=\widetilde{M}+\widetilde{A}$ and $\widetilde{X}=\widetilde{N}+\widetilde{V}$ ($\widetilde{M},\widetilde{N}\in \mathcal{M}_{\mathrm{loc}}$ and $\widetilde{A},\widetilde{V}\in \mathcal{V}_0$). Put $X=\widetilde{X}\mathfrak{I}_B$, $M=\widetilde{M}\mathfrak{I}_B$, $A=\widetilde{A}\mathfrak{I}_B$, $N=\widetilde{N}\mathfrak{I}_B$, and $V=\widetilde{V}\mathfrak{I}_B$. Then from Remark \ref{reprocess}, $X\in \mathcal{S}^B$ could admits two different decompositions $X=M+A$ and $X=N+V$ ($M,N\in (\mathcal{M}_{\mathrm{loc}})^B$ and $A,V\in (\mathcal{V}_0)^B$).

Combining (\ref{eq-XMA}) with (\ref{con-M}), we have that $X\in \mathcal{S}^B$ admits a further decomposition
\begin{equation}\label{decomX}
X=X_0\mathfrak{I}_B+M^c+M^d+A,
\end{equation}
where $M^c\in (\mathcal{M}^c_{\mathrm{loc},0})^B$, $M^d\in (\mathcal{M}^d_{\mathrm{loc}})^B$ and $A\in (\mathcal{V}_0)^B$.

\begin{theorem}\label{uXc}
Let $X\in \mathcal{S}^B$, and $(\ref{decomX})$ be a decomposition of $X$. Then $M^c$ is uniquely determined by $X$, i.e., $M^c=N^c$ if $X$ admits another decomposition
\begin{equation*}
X=X_0\mathfrak{I}_B+N^c+N^d+V,
\end{equation*}
where $N^c\in (\mathcal{M}^c_{\mathrm{loc},0})^B$, $N^d\in (\mathcal{M}^d_{\mathrm{loc}})^B$ and $V\in (\mathcal{V}_0)^B$. At this time, $M^c$ in the decomposition (\ref{decomX}) is called the continuous part of $X$, and is denoted by $X^c$ as well.
\end{theorem}
\begin{proof}
Let $(\tau_n)$ be an FS for $B$. From Theorem \ref{fcs-p}, $(\tau_n,X^{\tau_n})$ is an FCS for $X\in \mathcal{S}^B$, and from \eqref{decomX}, for each $n\in \mathbb{N}^+$, $X^{\tau_n}\in \mathcal{S}$ admits the following decompositions:
\begin{align*}
X^{\tau_n}&=X_0\mathfrak{I}_B+(M^c)^{\tau_n}+(M^d)^{\tau_n}+A^{\tau_n}=X_0\mathfrak{I}_B+(N^c)^{\tau_n}+(N^d)^{\tau_n}+V^{\tau_n},
\end{align*}
where $(M^c)^{\tau_n},(N^c)^{\tau_n}\in \mathcal{M}^c_{\mathrm{loc},0}$, and $(M^d)^{\tau_n},(N^d)^{\tau_n}\in \mathcal{M}^d_{\mathrm{loc}}$, and $A^{\tau_n}, V^{\tau_n}\in \mathcal{V}_0$. On the other hand, for each $n\in \mathbb{N}^+$, the continuous part of semimartingale $X^{\tau_n}$ is unique (see, e.g., Proposition I.4.27 in \cite{Jacod}), and it follows $(M^c)^{\tau_n}=(N^c)^{\tau_n}$. Hence, by Theorem \ref{process-FS}, we deduce $M^c=N^c$, i.e., the uniqueness of $M^c$.
\end{proof}

We present fundamental properties of the continuous part of a semimartingale on $B$ in the following theorem.
\begin{theorem}\label{XTc}
Let $X\in \mathcal{S}^B$. Then we have the following statements:
\begin{itemize}
  \item [$(1)$]If $(T_n,X^{(n)})$ is an FCS for $X\in \mathcal{S}^B$, then $(T_n,(X^{(n)})^c)$ is an FCS for $X^c\in (\mathcal{M}^c_{\mathrm{loc},0})^B$.
  \item [$(2)$]If $\tau$ is a stopping time on $B$, then
  \begin{equation}\label{Xc-eq}
  (X^\tau)^c=(X^c)^\tau
  \end{equation}
  and
  \begin{equation}\label{XcB-eq}
  (X^c)^\tau\mathfrak{I}_B=(X^\tau)^c\mathfrak{I}_B=(X^\tau\mathfrak{I}_B)^c.
  \end{equation}
  \item [$(3)$] If $(\tau_n)$ is an FS for $B$, then $(\tau_n,(X^{\tau_n})^c)$ is an FCS for $X^c\in (\mathcal{M}^c_{\mathrm{loc},0})^B$.
\end{itemize}
\end{theorem}
\begin{proof}
$(1)$ The proof is analogous to that of Theorem \ref{Mcn}. Fix $n\in \mathbb{N}^+$, and put $B_n=B\llbracket{0,T_n}\rrbracket$. Then Remark \ref{reprocess} shows that $X\mathfrak{I}_{B_n}$ and $X^{(n)}\mathfrak{I}_{B_n}$ are both semimartingales on $B_n$.
Assume that $X$ admits the decomposition \eqref{decomX}, and that $X^{(n)}$ admits the decomposition $X^{(n)}=X^{(n)}_0+N^c+N^d+V$, where $N^c\in \mathcal{M}^c_{\mathrm{loc},0}$, $N^d\in \mathcal{M}^d_{\mathrm{loc}}$ and $V\in \mathcal{V}_0$. Then $X^c=M^c$ and $(X^{(n)})^c=N^c$, and the following relations hold true:
\[
\left\{
\begin{aligned}
X\mathfrak{I}_{B_n}&=(X_0\mathfrak{I}_B+M^c+M^d+A)\mathfrak{I}_{B_n}=X_0\mathfrak{I}_{B_n}
    +M^c\mathfrak{I}_{B_n}+M^d\mathfrak{I}_{B_n}+A\mathfrak{I}_{B_n},\\
X^{(n)}\mathfrak{I}_{B_n}&=(X^{(n)}_0+N^c+N^d+V)\mathfrak{I}_{B_n}=X^{(n)}_0\mathfrak{I}_{B_n}
+N^c\mathfrak{I}_{B_n}+N^d\mathfrak{I}_{B_n}+V\mathfrak{I}_{B_n},\\
\end{aligned}
\right.
\]
where $M^c\mathfrak{I}_{B_n},N^c\mathfrak{I}_{B_n}\in (\mathcal{M}^c_{\mathrm{loc},0})^{B_n}$, and $M^d\mathfrak{I}_{B_n},N^d\mathfrak{I}_{B_n}\in (\mathcal{M}^d_{\mathrm{loc}})^{B_n}$, and $A\mathfrak{I}_{B_n}, V\mathfrak{I}_{B_n}\in (\mathcal{V}_0)^{B_n}$.
Noticing $X\mathfrak{I}_{B_n}=X^{(n)}\mathfrak{I}_{B_n}$ and using the uniqueness of their continuous parts, we deduce
\begin{equation}\label{Xc-eq2}
X^c\mathfrak{I}_{B_n}=M^c\mathfrak{I}_{B_n}=(X\mathfrak{I}_{B_n})^c=(X^{(n)}\mathfrak{I}_{B_n})^c
=N^c\mathfrak{I}_{B_n}=(X^{(n)})^c\mathfrak{I}_{B_n}.
\end{equation}
Since (\ref{Xc-eq2}) holds for each $n\in \mathbb{N}^+$, we obtain $X^cI_{B\llbracket{0,T_n}\rrbracket}=(X^{(n)})^cI_{B\llbracket{0,T_n}\rrbracket}$ which yields that $(T_n,(X^{(n)})^c)$ is a CS for $X^c$. From $(X^{(n)})^c\in \mathcal{M}^c_{\mathrm{loc},0}$ for each $n\in \mathbb{N}^+$, the sequence $(T_n,(X^{(n)})^c)$ is an FCS for $X^c\in (\mathcal{M}^c_{\mathrm{loc},0})^B$.

$(2)$ Let $(\ref{decomX})$ be a decomposition of $X$. From Theorem \ref{fcs}, $X^\tau\in \mathcal{S}$ and its unique continuous part is $(X^\tau)^c$. However, using \eqref{decomX}, we deduce
\[
X^\tau=(X_0+M^c+M^d+A)^{\tau}=X_0+(M^c)^{\tau}+(M^d)^{\tau}+A^{\tau},
\]
which, by Theorem \ref{uXc}, implies $(X^\tau)^c=(M^c)^{\tau}$. Hence, (\ref{Xc-eq}) is obtained by $X^c=M^c$.
As for (\ref{XcB-eq}), the first equality has been proved by (\ref{Xc-eq}), and it suffices to prove the second equality. The continuous part of $X^\tau\mathfrak{I}_B$ is $(X^\tau\mathfrak{I}_B)^c$. On the other hand, the relation
\[
 X^\tau\mathfrak{I}_B=X_0\mathfrak{I}_B+(M^c)^{\tau}\mathfrak{I}_B+(M^d)^{\tau}\mathfrak{I}_B+A^{\tau}\mathfrak{I}_B
\]
shows that the continuous part of $X^\tau\mathfrak{I}_B$ can also expressed as $(M^c)^{\tau}\mathfrak{I}_B=(X^c)^{\tau}\mathfrak{I}_B$. Hence, the second equality of (\ref{XcB-eq}) is obtained by the uniqueness of the continuous part of $X^\tau\mathfrak{I}_B$.

$(3)$ Theorem \ref{fcs-p} shows that $(\tau_n,X^{\tau_n})$ is an FCS for $X\in \mathcal{S}^B$, and then the statement is a direct result of $(1)$.
\end{proof}

Now we can define the quadratic covariation of two semimartingales on a PSIT, and then study its fundamental properties.
\begin{definition}\label{[X,Y]}
Let $X,Y\in\mathcal{S}^B$. Put
\[
[X,Y]=X_0Y_0\mathfrak{I}_B+\langle X^c,Y^c\rangle+\Sigma(\Delta X\Delta Y),
\]
where $X^c\in(\mathcal{M}^c_{\mathrm{loc},0})^B$ and $Y^c\in(\mathcal{M}^c_{\mathrm{loc},0})^B$ are the continuous parts of $X$ and $Y$, respectively.
Then $[X,Y]\in \mathcal{V}^B$ is called the quadratic covariation on $B$ of $X$ and $Y$. And in the case of $X=Y$, the process $[X,X]$ (or simply, $[X]$) is called the quadratic variation on $B$ of $X$.
\end{definition}

The Definition \ref{[X,Y]} is the same as Definition 8.2 in \cite{He} if $B=\llbracket{0,+\infty}\llbracket$.  The following theorem presents fundamental properties of the quadratic covariation $[X,Y]$ in Definition \ref{[X,Y]}.

\begin{theorem}\label{[X,Y]-fcs}
Let $X,\; Y\in \mathcal{S}^B$.
\begin{itemize}
  \item [$(1)$] If $Z\in \mathcal{S}^B$ and $a,b\in \mathbb{R}$, then
  \[
    [X,Y]=[Y,X],\quad [aX+bY,Z]=a[X,Z]+b[Y,Z].
  \]
  \item [$(2)$] If $(T_n,X^{(n)})$ and $(T_n,Y^{(n)})$ are FCSs for $X\in \mathcal{S}^B$ and $Y\in \mathcal{S}^B$ respectively, then $(T_n,[X^{(n)},Y^{(n)}])$ is an FCS for $[X,Y]\in \mathcal{V}^B$.
  \item [$(3)$] If $(\tau_n)$ is an FS for $B$, then $(\tau_n,[X^{\tau_n},Y^{\tau_n}])$ is an FCS for $[X,Y]\in \mathcal{V}^B$.
  \item [$(4)$] If $\tau$ be a stopping time on $B$, then
\begin{equation}\label{XY}
      [X^{\tau},Y^{\tau}]=[X,Y]^{\tau}
\end{equation}
and
\begin{equation}\label{XYB}
[X^{\tau}\mathfrak{I}_B,Y^{\tau}\mathfrak{I}_B]
=[X,Y]^{\tau}\mathfrak{I}_B=[X^{\tau},Y^{\tau}]\mathfrak{I}_B=[X^{\tau}\mathfrak{I}_B,Y].
\end{equation}
\end{itemize}
\end{theorem}
\begin{proof}
$(1)$ From the statement (1) of Theorem \ref{property-qr-p}, we deduce
\[
[X,Y]=X_0Y_0\mathfrak{I}_B+\langle X^c,Y^c\rangle+\Sigma(\Delta X\Delta Y)
=Y_0X_0\mathfrak{I}_B+\langle Y^c,X^c\rangle+\Sigma(\Delta Y\Delta X)=[Y,X]
\]
and
\begin{align*}
[aX+bY,Z]
&=(aX+bY)_0Z_0\mathfrak{I}_B+\langle (aX+bY)^c,Z^c\rangle+\Sigma(\Delta (aX+bY)\Delta Z)\\
&=(aX_0+bY_0)Z_0\mathfrak{I}_B+\langle aX^c+bY^c,Z^c\rangle+\Sigma(a\Delta X+b\Delta Y)\Delta Z)\\
&=a\left(X_0Z_0\mathfrak{I}_B+\langle X^c,Z^c\rangle+\Sigma(\Delta X\Delta Z)\right)+b\left(Y_0Z_0\mathfrak{I}_B+\langle Y^c,Z^c\rangle+\Sigma(\Delta Y\Delta Z)\right)\\
&=a[X,Z]+b[Y,Z],
\end{align*}
which yields the result.

$(2)$ Using Theorems \ref{Mcn} and \ref{thin}, and Corollary \ref{qr-pc}, we deduce $[X^{(n)},Y^{(n)}]\in \mathcal{V}$ and
\begin{align*}
&[X^{(n)},Y^{(n)}]I_{B\llbracket{0,T_n}\rrbracket}\\
=&\left(X_0^{(n)}Y_0^{(n)}+\langle(X^{(n)})^c,(Y^{(n)})^c\rangle+\Sigma(\Delta X^{(n)}\Delta Y^{(n)})\right)I_{B\llbracket{0,T_n}\rrbracket}\\
=&\left(X_0Y_0\mathfrak{I}_B+\langle X^c,Y^c\rangle+\Sigma(\Delta X\Delta Y)\right)I_{B\llbracket{0,T_n}\rrbracket}\\
=&[X,Y]I_{B\llbracket{0,T_n}\rrbracket}
\end{align*}
for each $n\in \mathbb{N}^+$, which gives the statement.

$(3)$ From Theorem \ref{fcs-p}, $(\tau_n,X^{\tau_n})$ is an FCS for $X\in \mathcal{S}^B$, and $(\tau_n,Y^{\tau_n})$ is an FCS for $Y\in \mathcal{S}^B$. Then using the statement $(1)$, we deduce that $(\tau_n,[X^{\tau_n},Y^{\tau_n}])$ is an FCS for $[X,Y]\in \mathcal{V}^B$.

$(4)$ We start to prove \eqref{XY}. Theorems \ref{property-qr-p} and \ref{XTc} show
\begin{equation*}
      \langle (X^{\tau})^c,(Y^{\tau})^c\rangle=\langle (X^c)^{\tau},(Y^c)^{\tau}\rangle=\langle X^c,Y^c\rangle^{\tau},
\end{equation*}
and the relations \eqref{eqXT} and \eqref{sigmaX} yield
\[
\Sigma(\Delta (X^{\tau})\Delta (Y^{\tau}))
=\Sigma((\Delta X\Delta Y)I_{\llbracket{0,\tau}\rrbracket})
=(\Sigma(\Delta X\Delta Y))^\tau.
\]
Then \eqref{XY} can be obtained by
\begin{align*}
[X^{\tau},Y^{\tau}]
&=X_0Y_0+\langle (X^{\tau})^c,(Y^{\tau})^c\rangle+\Sigma(\Delta (X^{\tau})\Delta (Y^{\tau}))\\
&=X_0Y_0+\langle X^c,Y^c\rangle^{\tau}+(\Sigma(\Delta X\Delta Y))^\tau\\
&=(X_0Y_0\mathfrak{I}_B+\langle X^c,Y^c\rangle+\Sigma(\Delta X\Delta Y))^\tau\\
&=[X,Y]^{\tau}.
\end{align*}

Next, we turn to \eqref{XYB}. Assume that $(T_n,X^{(n)})$ and $(T_n,Y^{(n)})$ are FCSs for $X\in\mathcal{S}^B$ and $Y\in\mathcal{S}^B$ respectively. From Theorem \ref{fcs},
$(T_n,(X^{(n)})^\tau)$ and $(T_n,(Y^{(n)})^\tau)$ are FCSs for $X^\tau\mathfrak{I}_B\in\mathcal{S}^B$ and $Y^\tau\mathfrak{I}_B\in\mathcal{S}^B$, respectively. By Theorem \ref{fcs} and the statement (2), we deduce
\begin{align*}
[X^{\tau}\mathfrak{I}_B,Y^{\tau}\mathfrak{I}_B]I_{B\llbracket{0,T_n}\rrbracket}
&=[(X^{(n)})^\tau,(Y^{(n)})^\tau]I_{B\llbracket{0,T_n}\rrbracket}
=[X^{(n)},Y^{(n)}]^\tau I_{B\llbracket{0,T_n}\rrbracket}\\
&=[X,Y]^\tau I_{B\llbracket{0,T_n}\rrbracket}
=([X,Y]^\tau \mathfrak{I}_B) I_{B\llbracket{0,T_n}\rrbracket}, \quad n\in \mathbb{N}^+,
\end{align*}
and this proves the first equality of \eqref{XYB}.
The second equality of \eqref{XYB} can be obtained by \eqref{XY} easily. Using Theorem \ref{fcs} and the statement (2) again, the last equality of \eqref{XYB} can be obtained by
\begin{align*}
[X^{\tau}\mathfrak{I}_B,Y]I_{B\llbracket{0,T_n}\rrbracket}
&=[(X^{(n)})^\tau,Y^{(n)}]I_{B\llbracket{0,T_n}\rrbracket}
=[X^{(n)},Y^{(n)}]^\tau I_{B\llbracket{0,T_n}\rrbracket}\\
&=[X,Y]^\tau I_{B\llbracket{0,T_n}\rrbracket}
=([X,Y]^\tau \mathfrak{I}_B) I_{B\llbracket{0,T_n}\rrbracket}, \quad n\in \mathbb{N}^+.
\end{align*}
Therefore, \eqref{XYB} is valid.
\end{proof}

\begin{corollary}\label{[X]}
Let $X\in \mathcal{S}^B$ with an FCS $(T_n,X^{(n)})$. Then $[X]\in (\mathcal{V}^+)^B$, and $(T_n,[X^{(n)}])$ is an FCS for $[X]\in (\mathcal{V}^+)^B$. Specially, if $(\tau_n)$ is an FS for $B$, then $(\tau_n,[X^{\tau_n}])$ is an FCS for $[X]\in (\mathcal{V}^+)^B$.
\end{corollary}
\begin{proof}
From Theorem \ref{[X,Y]-fcs}, $(T_n,[X^{(n)}])$ is an FCS for $[X]\in \mathcal{V}^B$, which leads to
\[
[X]I_{B\llbracket{0,T_n}\rrbracket}
=[X^{(n)}]I_{B\llbracket{0,T_n}\rrbracket}, \quad n\in \mathbb{N}^+.
\]
Since $[X^{(n)}]\in\mathcal{V}^+$ holds (see Definition 8.2 in \cite{He}) for each $n\in \mathbb{N}^+$, we obtain $[X]\in (\mathcal{V}^+)^B$ with the FCS $(T_n,[X^{(n)}])$. From Theorem \ref{fcs-p}, $(\tau_n,X^{\tau_n})$ is an FCS for $X\in \mathcal{S}^B$, and then it is easy to see that $(\tau_n,[X^{\tau_n}])$ is an FCS for $[X]\in (\mathcal{V}^+)^B$.
\end{proof}

Similar to Example \ref{gen-M}, we give an example to show that a sequence of semimartingales can be used to construct a general semimartingale on $B$.
\begin{example}\label{general-X}
Assume that $(\tau_n)$ is an FCS for $B$, and that $(Y^{(n)})$ is a sequence of semimartingales. Put
\[
X^{(1)}=Y^{(1)}, \quad
X^{(n+1)}=Y^{(n+1)}+(X^{(n)}-Y^{(n+1)})^{\tau_{n}},\quad n\in \mathbb{N}^+
\]
and
\begin{equation}\label{general-eX}
X=\left(X_0^{(1)}I_{\llbracket{0}\rrbracket}+\sum\limits_{n=1}^{{+\infty}}
      X^{(n)}I_{\rrbracket{\tau_{n-1},\tau_n}
      \rrbracket}\right)\mathfrak{I}_B,\quad \tau_0=0.
\end{equation}
Then we have the following statements:
\begin{itemize}
  \item [$(1)$] $X\in\mathcal{S}^B$, and $(\tau_n,X^{(n)})$ is an FCS for $X\in\mathcal{S}^B$.  For any $n,k\in \mathbb{N}^+$ with $k\leq n$, by induction, we deduce $X^{(n)}\in\mathcal{S}$ and $(X^{(k)})^{\tau_{k}}=(X^{(n)})^{\tau_{k}}$. Remark \ref{remark-cs} shows $X\in\mathcal{S}^B$ with the FCS $(\tau_n,X^{(n)})$.

  \item [$(2)$] From Theorem \ref{delta}, $(\tau_n,\Delta X^{(n)})$ is a CS for $\Delta X$. And from \eqref{x-expression},
  \begin{equation}\label{DX}
  \Delta X=\left(\sum\limits_{n=1}^{{+\infty}}
      \Delta X^{(n)}I_{\rrbracket{\tau_{n-1},\tau_n}
      \rrbracket}\right)\mathfrak{I}_B.
  \end{equation}

  \item [$(3)$] From Theorem \ref{XTc} and \eqref{x-expression}, $(\tau_n,(X^{(n)})^c)$ is an FCS for $X^c\in(\mathcal{M}^c_{\mathrm{loc},0})^B$, and $X^c$ can be expressed as
  \[
  X^c=\left(\sum\limits_{n=1}^{{+\infty}}
      (X^{(n)})^cI_{\rrbracket{\tau_{n-1},\tau_n}
      \rrbracket}\right)\mathfrak{I}_B.
  \]

  \item [$(4)$] From Corollary \ref{[X]} and \eqref{x-expression}, $[X]$ can be expressed as
  \[
  [X]=\left(X^2_0I_{\llbracket{0}
      \rrbracket}+\sum\limits_{n=1}^{{+\infty}}
      [X^{(n)}]I_{\rrbracket{\tau_{n-1},\tau_n}
      \rrbracket}\right)\mathfrak{I}_B,
  \]
  and $(\tau_n,[X^{(n)}])$ is an FCS for $[X]\in(\mathcal{V}^+)^B$.
\end{itemize}
\end{example}

\subsection{Stochastic integrals on PSITs of predictable processes with respect to semimartingales}
The classic stochastic integral \eqref{0-HX} is based on a decomposition of the semimartingale, and is independent of the choice of decompositions. Hence, we first give the following lemma which enables us to develop stochastic integrals on PSITs of predictable processes w.r.t. semimartingales.

\begin{lemma}\label{HX-unique}
Let $H\in \mathcal{P}^B$ and $X\in \mathcal{S}^B$. Assume that $X=M+A$ and $X=N+V$ are both decompositions of $X$, where $M\in (\mathcal{M}_{\mathrm{loc}})^B$, $A\in (\mathcal{V}_0)^B$, $N\in (\mathcal{M}_{\mathrm{loc}})^B$ and $V\in (\mathcal{V}_0)^B$. If $H\in \mathcal{L}^B_m(M)\bigcap\mathcal{L}^B_m(N)$ and if both $H_{\bullet}A$ and $H_{\bullet}V$ exist, then
\begin{equation}\label{HX-unique-1}
H_{\bullet}M+H_{\bullet}A=H_{\bullet}N+H_{\bullet}V.
\end{equation}
\end{lemma}
\begin{proof}
Let $(\tau_n)$ be an FS for $B$. Theorem \ref{HA-FCS-p} shows that $(\tau_n,H^{\tau_n}.A^{\tau_n})$ and $(T_n,H^{\tau_n}.V^{\tau_n})$ are FCSs for $H_{\bullet}A\in\mathcal{V}^B$ and $H_{\bullet}V\in\mathcal{V}^B$ respectively, and Theorem \ref{eq-HM} shows that $(\tau_n,H^{\tau_n}.M^{\tau_n})$ and $(\tau_n,H^{\tau_n}.N^{\tau_n})$ are FCSs for $H_{\bullet}M\in (\mathcal{M}_{\mathrm{loc}})^B$ and $H_{\bullet}N\in (\mathcal{M}_{\mathrm{loc}})^B$ respectively.
Since $X^{\tau_n}=M^{\tau_n}+A^{\tau_n}=N^{\tau_n}+V^{\tau_n}$ for each $n\in \mathbb{N}^+$, by Theorem 9.12 in \cite{He}, we deduce that
\[
H^{\tau_n}.M^{\tau_n}+H^{\tau_n}.A^{\tau_n}=H^{\tau_n}.N^{\tau_n}+H^{\tau_n}.V^{\tau_n},\quad n\in \mathbb{N}^+.
\]
Then it follows that for each $n\in \mathbb{N}^+$,
\begin{align*}
(H_{\bullet}M+H_{\bullet}A)I_{\llbracket{0,\tau_n}\rrbracket}
&=(H^{\tau_n}.M^{\tau_n}+H^{\tau_n}.A^{\tau_n})I_{\llbracket{0,\tau_n}\rrbracket}\\
&=(H^{\tau_n}.N^{\tau_n}+H^{\tau_n}.V^{\tau_n})I_{\llbracket{0,\tau_n}\rrbracket}\\
&=(H_{\bullet}N+H_{\bullet}V)I_{\llbracket{0,\tau_n}\rrbracket},
\end{align*}
which, by the statement $(1)$ of Theorem \ref{process-FS}, yields \eqref{HX-unique-1}.
\end{proof}

\begin{definition}\label{de-HX}
Let $H\in \mathcal{P}^B$ and $X\in \mathcal{S}^B$. We say that $H$ is integrable on $B$ w.r.t. $X$ in the domain of semimartingales (or simply, $H$ is $X$-integrable on $B$), if there exists a decomposition $X=M+A$ ($M\in (\mathcal{M}_{\mathrm{loc}})^B$ and $A\in (\mathcal{V}_0)^B$) such that $H\in \mathcal{L}_m^{B}(M)$ and $H_{\bullet}A$ exists.
At this time, the process defined by
\begin{equation}\label{HX}
H_{\bullet}X:=H_{\bullet}M+H_{\bullet}A
\end{equation}
is called the stochastic integral on $B$ of $H$ w.r.t. $X$, and $X=M+A$ is an $H$-decomposition on $B$ of $X$. The collection of all predictable processes on $B$ which are $X$-integrable on $B$ is denoted by $\mathcal{L}^{B}(X)$.
\end{definition}

Lemma \ref{HX-unique} guarantees that the stochastic integral $H_{\bullet}X$ defined by \eqref{HX} is independent of $H$-decompositions of $X$, which is analogous with the classic stochastic integral \eqref{0-HX}.
\begin{remark}\label{HXB=HX}
From Corollary \ref{cD=DB}, the relations
\begin{equation*}
\mathcal{S}=\mathcal{S}^{\llbracket{0,+\infty}\llbracket}, \quad \mathcal{M}_{\mathrm{loc}}=(\mathcal{M}_{\mathrm{loc}})^{\llbracket{0,+\infty}\llbracket},
\quad \text{and}\quad \mathcal{V}_0=(\mathcal{V}_0)^{\llbracket{0,+\infty}\llbracket}
\end{equation*}
hold true.
Then from Remarks \ref{HAB=HA} and \ref{HMB=HM}, it is easy to see that the stochastic integral $H_{\bullet}X$ defined by \eqref{HX} degenerates to the stochastic integral $H.X$ defined by \eqref{0-HX} if $B=\llbracket{0,+\infty}\llbracket=\Omega\times\mathbb{R}^+$. More precisely, the following relation holds:
\begin{itemize}
  \item [] If $H\in \mathcal{P}^{\llbracket{0,+\infty}\llbracket}$ and $X\in \mathcal{S}^{\llbracket{0,+\infty}\llbracket}$, then $H_{\bullet}X=H.X$.
\end{itemize}
\end{remark}

It is also of significance to reveal the relation between the stochastic integrals \eqref{HX} and \eqref{0-HX}, and the results are presented in the following two theorems.
Theorem \ref{HX=} presents the sufficient and necessary conditions of the existence of the stochastic integral $H_{\bullet}X$. Theorem \ref{eq-HX} characterizes the stochastic integral $H_{\bullet}X$ as a summation of a sequence of stochastic integrals of predictable processes w.r.t. semimartigales.

\begin{theorem}\label{HX=}
Let $H\in \mathcal{P}^B$ and $X\in \mathcal{S}^B$.
Then the following statements are equivalent:
\begin{description}
  \item[$(i)$] $H\in\mathcal{L}^B(X)$.
  \item[$(ii)$] There exist an FS $(\tau_n)$ for $B$ and a decomposition $X=M+A$ ($M\in (\mathcal{M}_{\mathrm{loc}})^B$ and $A\in (\mathcal{V}_0)^B$) such that for each $n\in \mathbb{N}^+$, $H^{\tau_n}\in\mathcal{L}_m(M^{\tau_n})$ and $H^{\tau_n}.A^{\tau_n}$ exists.
  \item[$(iii)$] There exist a decomposition $X=M+A$ ($M\in (\mathcal{M}_{\mathrm{loc}})^B$ and $A\in (\mathcal{V}_0)^B$) and FCSs $(T_n,H^{(n)})$ for $H\in\mathcal{P}^B$, $(T_n,M^{(n)})$ for $M\in (\mathcal{M}_{\mathrm{loc}})^B$ and $(T_n,A^{(n)})$ for $A\in (\mathcal{V}_0)^B$ such that for each $n\in \mathbb{N}^+$, $H^{(n)}\in\mathcal{L}_m(M^{(n)})$ and $H^{(n)}.A^{(n)}$ exists.
  \item[$(iv)$] There exists an FS $(\tau_n)$ for $B$ such that for each $n\in \mathbb{N}^+$, $H^{\tau_n}\in\mathcal{L}(X^{\tau_n})$.
  \item[$(v)$] There exist FCSs $(T_n,H^{(n)})$ for $H\in\mathcal{P}^B$ and $(T_n,X^{(n)})$ for $X\in \mathcal{S}^B$ such that for each $n\in \mathbb{N}^+$, $H^{(n)}\in\mathcal{L}(X^{(n)})$.
\end{description}
\end{theorem}
\begin{proof}
$(i)\Rightarrow (ii)$. Suppose $H\in\mathcal{L}^B(X)$.  Let $(\tau_n)$ be an FS for $B$. From Definition \ref{de-HX}, there exists a decomposition $X=M+A$ ($M\in (\mathcal{M}_{\mathrm{loc}})^B$ and $A\in (\mathcal{V}_0)^B$) such that $H\in \mathcal{L}_m^{B}(M)$ and $H_{\bullet}A$ exists. Then for each $n\in \mathbb{N}^+$, Theorem \ref{HA-FCS-p} shows the existence of $H^{\tau_n}.A^{\tau_n}$, and Theorem \ref{eq-HM} shows $H^{\tau_n}\in\mathcal{L}_m(M^{\tau_n})$. Hence, $(ii)$ is obtained.

$(ii)\Rightarrow (iii)$. Suppose $(ii)$ holds. Put $T_n=\tau_n$, $H^{(n)}=H^{\tau_n}$, $M^{(n)}=M^{\tau_n}$ and $A^{(n)}=A^{\tau_n}$ for each $n\in \mathbb{N}^+$. Then by Theorem \ref{fcs-p}, we deduce $(iii)$.

$(iii)\Rightarrow (i)$. Suppose $(iii)$ holds. By Theorems \ref{HA-equivalent-p} and \ref{HM=}, we deduce that $H\in \mathcal{L}_m^{B}(M)$ and $H_{\bullet}A$ exists. Then Definition \ref{de-HX} shows $H\in\mathcal{L}^B(X)$.

$(ii)\Rightarrow (iv)$. Suppose $(ii)$ holds. For each $n\in \mathbb{N}^+$, Theorem \ref{fcs-p} shows $M^{\tau_n}\in \mathcal{M}_{\mathrm{loc}}$ and $A^{\tau_n}\in\mathcal{V}_0$, and then $X^{\tau_n}=M^{\tau_n}+A^{\tau_n}$ is an $H^{\tau_n}$-decomposition of $X^{\tau_n}$, i.e., $H^{\tau_n}\in\mathcal{L}(X^{\tau_n})$. Hence, we deduce $(iv)$.

$(iv)\Rightarrow (ii)$. Suppose $(iv)$ holds. For each $n\in \mathbb{N}^+$, let $X^{\tau_n}=M^{(n)}+A^{(n)}$ be a decomposition of $X^{\tau_n}$ such that $H^{\tau_n}\in\mathcal{L}_m(M^{(n)})$ and $H^{\tau_n}.A^{(n)}$ exists, where $M^{(n)}\in \mathcal{M}_{\mathrm{loc}}$ and $A^{(n)}\in \mathcal{V}_0$.
Put
\[
\widetilde{H}:=HI_B=H_0I_{\llbracket{0}\rrbracket}+\sum\limits_{n=1}^{{+\infty}}
H^{\tau_n}I_{\rrbracket{\tau_{n-1},\tau_n}\rrbracket}.
\]
Then $\widetilde{H}$ is a predictable process satisfying $\widetilde{H}^{\tau_n}=H^{\tau_n}$ for each $n\in \mathbb{N}^+$.

We first prove that for each $n\in \mathbb{N}^+$, $\widetilde{H}\in\mathcal{L}_m((M^{(n)})^{\tau_n})$ and that $\widetilde{H}.(A^{(n)})^{\tau_n}$ exists.
Let $n\in \mathbb{N}^+$ be fixed. From Theorem 9.2 in \cite{He} and Lemma \ref{property}, the relations $H^{\tau_n}\in\mathcal{L}_m(M^{(n)})$ and $(H^{\tau_n}.M^{(n)})^{\tau_n}=H^{\tau_n}.(M^{(n)})^{\tau_n}$ yield
\[
H^{\tau_n}\in\mathcal{L}_m((M^{(n)})^{\tau_n})\quad \text{and}\quad \sqrt{(H^{\tau_n})^2.\left[(M^{(n)})^{\tau_n}\right]}\in \mathcal{A}^+_{\mathrm{loc}}.
\]
 It follows that
\[
\sqrt{\widetilde{H}^2.\left[(M^{(n)})^{\tau_n}\right]}=\sqrt{(\widetilde{H}^{\tau_n})^2.
\left[(M^{(n)})^{\tau_n}\right]}
=\sqrt{(H^{\tau_n})^2.\left[(M^{(n)})^{\tau_n}\right]}\in \mathcal{A}^+_{\mathrm{loc}},
\]
and by Theorem 9.2 in \cite{He}, we deduce $\widetilde{H}\in\mathcal{L}_m((M^{(n)})^{\tau_n})$. On the other hand, the existence of $H^{\tau_n}.A^{(n)}$ implies the existence of $H^{\tau_n}.(A^{(n)})^{\tau_n}$, and then the relation
\begin{align*}
\int_{[0,t]}|\widetilde{H}_s(\omega)||d(A^{(n)})^{\tau_n}_s(\omega)|
&=\int_{[0,t]}|\widetilde{H}^{\tau_n}_s(\omega)||d(A^{(n)})^{\tau_n}_s(\omega)|\\
&=\int_{[0,t]}|H^{\tau_n}_s(\omega)||d(A^{(n)})^{\tau_n}_s(\omega)|<\infty,\; (\omega,t)\in \Omega\times \mathbb{R}^+
\end{align*}
shows the existence of $\widetilde{H}.(A^{(n)})^{\tau_n}$.

Now, we can prove $(ii)$. From Theorem \ref{XMA}, $X$ admits a decomposition $X=M+A$, where $M\in (\mathcal{M}_{\mathrm{loc}})^B$ and $A\in (\mathcal{V}_0)^B$ are given by (\ref{MA-p}).
Then $(\tau_n,\widetilde{M}^{(n)})$ is an FCS for $M\in (\mathcal{M}_{\mathrm{loc}})^B$ satisfying $(\widetilde{M}^{(n)})^{\tau_n}=M^{\tau_n}$ for each $n\in \mathbb{N}^+$, and $(\tau_n,\widetilde{A}^{(n)})$ is an FCS for $A\in (\mathcal{V}_0)^B$ satisfying $(\widetilde{A}^{(n)})^{\tau_n}=A^{\tau_n}$ for each $n\in \mathbb{N}^+$.
For each $k\in \mathbb{N}^+$, the facts $\widetilde{H}\in\mathcal{L}_m((M^{(k)})^{\tau_k})$ and $(\widetilde{H}.(M^{(k)})^{\tau_k})^{\tau_{k-1}}=\widetilde{H}.(M^{(k)})^{\tau_{k-1}}$ yield $\widetilde{H}\in\mathcal{L}_m((M^{(k)})^{\tau_{k-1}})$, and the existence of $\widetilde{H}.(A^{(k)})^{\tau_k}$ and $(\widetilde{H}.(A^{(k)})^{\tau_k})^{\tau_{k-1}}=\widetilde{H}.(A^{(k)})^{\tau_{k-1}}$ yield the existence of $\widetilde{H}.(A^{(k)})^{\tau_{k-1}}$. Then for each $n\in \mathbb{N}^+$, using the relations
\begin{equation*}
\left\{
\begin{aligned}
(\widetilde{M}^{(n)})^{\tau_n}&=\sum_{k=1}^n\left((M^{(k)})^{\tau_{k}}
-(M^{(k)})^{\tau_{k-1}}\right)+(M^{(1)})^{\tau_{0}},\\
(\widetilde{A}^{(n)})^{\tau_n}&=\sum_{k=1}^n\left((A^{(k)})^{\tau_{k}}
-(A^{(k)})^{\tau_{k-1}}\right)+(A^{(1)})^{\tau_{0}}.
\end{aligned}
\right.
\end{equation*}
we deduce $\widetilde{H}\in\mathcal{L}_m((\widetilde{M}^{(n)})^{\tau_n})$ and the existence of $\widetilde{H}.(\widetilde{A}^{(n)})^{\tau_n}$.
Consequently, from Lemma \ref{property}, the relations
\[
\left\{
\begin{aligned}
(\widetilde{H}.(\widetilde{M}^{(n)})^{\tau_n})^{\tau_n}
&=\widetilde{H}^{\tau_n}.(\widetilde{M}^{(n)})^{\tau_n}=H^{\tau_n}.M^{\tau_n},\\
(\widetilde{H}.(\widetilde{A}^{(n)})^{\tau_n})^{\tau_n}
&=\widetilde{H}^{\tau_n}.(\widetilde{A}^{(n)})^{\tau_n}=H^{\tau_n}.A^{\tau_n}
\end{aligned}
\right.
\]
imply $H^{\tau_n}\in\mathcal{L}_m(M^{\tau_n})$ and the existence of $H^{\tau_n}.A^{\tau_n}$  for each $n\in \mathbb{N}^+$. Hence, $(ii)$ is obtained.

$(iv)\Rightarrow (v)$. Suppose $(iv)$ holds. For each $n\in \mathbb{N}^+$, put $H^{(n)}=H^{\tau_n}$ and $X^{(n)}=X^{\tau_n}$. Then from Theorem \ref{fcs-p}, $(v)$ holds true.

$(v)\Rightarrow (iv)$. Suppose $(v)$ holds. Let $(\alpha_n)$ be an FS for $B$. For each $n\in \mathbb{N}^+$, put $\tau_n=T_n\wedge\alpha_n$. From Theorem \ref{process-FS}, $(\tau_n)$ is an FS for $B$. For each $n\in \mathbb{N}^+$, using the definition of FCSs for $H\in \mathcal{P}^B$, we deduce that
\begin{align}\label{HnH}
H^{(n)}I_{\llbracket{0,\tau_n}\rrbracket}
=(H^{(n)}I_{B\llbracket{0,T_n}\rrbracket})I_{\llbracket{0,\alpha_n}\rrbracket}
=(HI_{B\llbracket{0,T_n}\rrbracket})I_{\llbracket{0,\alpha_n}\rrbracket}
=HI_{\llbracket{0,\tau_n}\rrbracket}
=H^{\tau_n}I_{\llbracket{0,\tau_n}\rrbracket},
\end{align}
which implies $(H^{(n)})^{\tau_n}=H^{\tau_n}$. And similarly, we also deduce $(X^{(n)})^{\tau_n}=X^{\tau_n}$ for each $n\in \mathbb{N}^+$. Then for each $n\in \mathbb{N}^+$, by Lemma \ref{property}, we have
\[
(H^{(n)}.X^{(n)})^{\tau_n}=(H^{(n)})^{\tau_n}.(X^{(n)})^{\tau_n}=H^{\tau_n}.X^{\tau_n},
\]
which indicates $H^{\tau_n}\in\mathcal{L}(X^{\tau_n})$.
\end{proof}

\begin{remark}
\begin{itemize}
  \item [$(1)$] The condition $(iii)$ in Theorem \ref{HX=} can be changed equivalently to the following condition:
\begin{description}
  \item[$(iii')$] There exist a decomposition $X=M+A$ ($M\in (\mathcal{M}_{\mathrm{loc}})^B$ and $A\in (\mathcal{V}_0)^B$) and FCSs $(T_n,H^{(n)})$ for $H\in\mathcal{P}^B$, $(\widetilde{T}_n,\widetilde{H}^{(n)})$ for $H\in\mathcal{P}^B$, $(S_n,M^{(n)})$ for $M\in (\mathcal{M}_{\mathrm{loc}})^B$ and $(\widetilde{S}_n,A^{(n)})$ for $A\in (\mathcal{V}_0)^B$ such that for each $n\in \mathbb{N}^+$, $H^{(n)}\in\mathcal{L}_m(M^{(n)})$ and $\widetilde{H}^{(n)}.A^{(n)}$ exists.
\end{description}
This is because the condition $(iii')$ is equivalent to the relation $H\in\mathcal{L}^B(X)$ from Definition \ref{de-HX} and Remarks \ref{HA==} and \ref{HM==}.

  \item [$(2)$] The condition $(v)$ in Theorem \ref{HX=} can be changed equivalently to the following condition:
\begin{description}
  \item[$(v')$] There exist FCSs $(T_n,H^{(n)})$ for $H\in\mathcal{P}^B$ and $(S_n,X^{(n)})$ for $X\in \mathcal{S}^B$ such that for each $n\in \mathbb{N}^+$, $H^{(n)}\in\mathcal{L}(X^{(n)})$.
\end{description}
Suppose the statement $(v)$ holds. Putting $S_n=T_n$ for each $n\in \mathbb{N}^+$, the statement $(v')$ is obtained obviously.
On the other hand, suppose the statement $(v')$ holds. Put $\tau_n=T_n\wedge S_n$ for each $n\in \mathbb{N}^+$. Then from the statement (3) of Theorem \ref{process}, $(\tau_n, H^{(n)})$ is an FCS for $H\in \mathcal{P}^B$ and $(\tau_n,X^{(n)})$ is an FCS for $X\in\mathcal{S}^B$, which yields $(v)$.
\end{itemize}
\end{remark}

\begin{theorem}\label{eq-HX}
Let $X\in \mathcal{S}^B$ and $H\in \mathcal{L}^B(X)$.
\begin{itemize}
  \item [$(1)$] If $(\tau_n)$ is an FS for $B$, then $(\tau_n,H^{\tau_n}.X^{\tau_n})$ is an FCS for $H_{\bullet}X\in\mathcal{S}^B$, and $H_{\bullet}X$ can be expressed as
  \begin{equation}\label{HX-expression-2}
      H_{\bullet}X=\left((H_0X_0)I_{\llbracket{0}\rrbracket}+\sum\limits_{n=1}^{{+\infty}}
      (H^{\tau_n}.X^{\tau_n})I_{\rrbracket{\tau_{n-1},\tau_n}
      \rrbracket}\right)\mathfrak{I}_B,\quad \tau_0=0.
      \end{equation}
  Furthermore, if $(\widetilde{\tau}_n)$ is also an FS for $B$, then $H_{\bullet}X=Z$ where the process $Z$ is given by
      \begin{equation*}
      Z=\left((H_0X_0)I_{\llbracket{0}\rrbracket}+\sum\limits_{n=1}^{{+\infty}}
      (H^{\widetilde{\tau}_n}.X^{\widetilde{\tau}_n})I_{\rrbracket{\widetilde{\tau}_{n-1},\widetilde{\tau}_n}
      \rrbracket}\right)\mathfrak{I}_B,\quad \widetilde{\tau}_0=0.
      \end{equation*}
  In this case, we say the expression of \eqref{HX-expression-2} is independent of the choice of FS $(\tau_n)$ for $B$.
  \item [$(2)$] If $(T_n,H^{(n)})$ for $H\in\mathcal{P}^{B}$ and $(T_n,X^{(n)})$ for $X\in\mathcal{S}^B$ are FCSs such that for each $n\in \mathbb{N}^+$, $H^{(n)}\in\mathcal{L}(X^{(n)})$, then $(T_n,H^{(n)}.X^{(n)})$ is an FCS for $H_{\bullet}X\in\mathcal{S}^B$, and $H_{\bullet}X$ can be expressed as
  \begin{equation}\label{HX-expression-4}
      H_{\bullet}X=\left((H_0X_0)I_{\llbracket{0}\rrbracket}+\sum\limits_{n=1}^{{+\infty}}
      (H^{(n)}.X^{(n)})I_{\rrbracket{T_{n-1},T_n}
      \rrbracket}\right)\mathfrak{I}_B,\quad T_0=0.
      \end{equation}
  Furthermore, if $(S_n, \widetilde{H}^{(n)})$ for $H\in\mathcal{P}^{B}$ and $(\widetilde{S}_n,\widetilde{X}^{(n)})$ for $X\in\mathcal{S}^B$ are FCSs such that for each $n\in \mathbb{N}^+$, $\widetilde{H}^{(n)}\in\mathcal{L}(\widetilde{X}^{(n)})$, then $H_{\bullet}X=Z$ where the process Z is given by
      \begin{equation*}
      Z=\left((H_0X_0)I_{\llbracket{0}\rrbracket}
      +\sum\limits_{n=1}^{{+\infty}}(\widetilde{H}^{(n)}.\widetilde{X}^{(n)})
      I_{\rrbracket{\widetilde{T}_{n-1},\widetilde{T}_n}\rrbracket}\right)\mathfrak{I}_B,\quad \widetilde{T}_0=0,
      \end{equation*}
      and $\widetilde{T}_n=S_n\wedge \widetilde{S}_n,\;n\in \mathbb{N}^+$. In this case, we say the expression of \eqref{HX-expression-4} is independent of the choice of FCSs $(T_n, H^{(n)})$ for $H\in\mathcal{P}^{B}$ and $(T_n,X^{(n)})$ for $X\in\mathcal{S}^B$.
\end{itemize}
\end{theorem}
\begin{proof}
$(1)$ Suppose that $X=M+A$ ($M\in (\mathcal{M}_{\mathrm{loc}})^B$ and $A\in (\mathcal{V}_0)^B$) is a decomposition such that $H\in \mathcal{L}_m^{B}(M)$ and $H_{\bullet}A$ exists. From Theorem \ref{HA-FCS-p},  $(\tau_n,H^{\tau_n}.A^{\tau_n})$ is an FCS for $H_{\bullet}A\in\mathcal{V}^B$, and from Theorem \ref{eq-HM}, $(\tau_n,H^{\tau_n}.M^{\tau_n})$ is an FCS for $H_{\bullet}M\in(\mathcal{M}_{\mathrm{loc}})^B$.  For each $n\in \mathbb{N}^+$, $H^{\tau_n}\in \mathcal{L}(X^{\tau_n})$ with $H^{\tau_n}.X^{\tau_n}\in\mathcal{S}$, and
\[
(H_{\bullet}X)I_{\llbracket{0,\tau_n}\rrbracket}
=(H_{\bullet}M+H_{\bullet}A)I_{\llbracket{0,\tau_n}\rrbracket}
=(H^{\tau_n}.M^{\tau_n}+H^{\tau_n}.A^{\tau_n})I_{\llbracket{0,\tau_n}\rrbracket}
=(H^{\tau_n}.X^{\tau_n})I_{\llbracket{0,\tau_n}\rrbracket}.
\]
Hence, $(\tau_n,H^{\tau_n}.X^{\tau_n})$ is an FCS for $H_{\bullet}X\in\mathcal{S}^B$.
The expression of \eqref{HX-expression-2}, as well as its independence property, can be obtained by  \eqref{x-expression}.

$(2)$ Firstly, we prove $L=H_{\bullet}X$, where $L$ is given by
\[
L=\left((H_0X_0)I_{\llbracket{0}\rrbracket}+\sum\limits_{n=1}^{{+\infty}}
      (H^{(n)}.X^{(n)})I_{\rrbracket{T_{n-1},T_n}
      \rrbracket}\right)\mathfrak{I}_B,\quad T_0=0.
\]
Let $(\tau_n)$ be an FS for $B$, and put $\beta_n=T_n\wedge\tau_n$ for each $n\in \mathbb{N}^+$. Then from Corollary \ref{process-FS}, $(\beta_n)$ is also an FS for $B$.
Similar to \eqref{HnH}, we deduce that for each $n\in \mathbb{N}^+$, $(H^{(n)})^{\beta_n}=H^{\beta_n}$ and $(X^{(n)})^{\beta_n}=X^{\beta_n}$.
Then the statement $(1)$ shows
\begin{align*}
&(H_{\bullet}X)I_{\llbracket{0,\beta_n}\rrbracket}\\
=&\left((H_0X_0)I_{\llbracket{0}\rrbracket}+\sum\limits_{k=1}^{n}
      (H^{\beta_k}.X^{\beta_k})I_{\rrbracket{\beta_{k-1},\beta_k}
      \rrbracket}\right)I_{\llbracket{0,\beta_n}\rrbracket}\\
=&\left((H_0X_0)I_{\llbracket{0}\rrbracket}+\sum\limits_{k=1}^{n}
      ((H^{(k)})^{\beta_k}.(X^{(k)})^{\beta_k})I_{\rrbracket{\beta_{k-1},\beta_k}
      \rrbracket}\right)I_{\llbracket{0,\beta_n}\rrbracket}\\
=&\left((H_0X_0)I_{\llbracket{0}\rrbracket}+\sum\limits_{k=1}^{n}
      (H^{(k)}.X^{(k)})I_{\rrbracket{\beta_{k-1},\beta_k}
      \rrbracket}\right)I_{\llbracket{0,\beta_n}\rrbracket}\\
=&LI_{\llbracket{0,\beta_n}\rrbracket},
\end{align*}
which, by the statement $(1)$ of Theorem \ref{process-FS}, implies $L=H_{\bullet}X$.

Next, we show that $(T_n,H^{(n)}.X^{(n)})$ is an FCS for $H_{\bullet}X\in\mathcal{S}^B$, thereby obtaining \eqref{HX-expression-4}. For every $n\in \mathbb{N}^+$ and $i\in \mathbb{N}^+$, using FCSs for $H\in \mathcal{P}^B$, we deduce the relation
\begin{align*}
H^{(n)}I_{\llbracket{0,T_n}\rrbracket}I_{\llbracket{0,\tau_i}\rrbracket}
=(H^{(n)}I_{B\llbracket{0,T_n}\rrbracket})I_{\llbracket{0,\tau_i}\rrbracket}
=(HI_{B\llbracket{0,T_n}\rrbracket})I_{\llbracket{0,\tau_i}\rrbracket}
=HI_{\llbracket{0,T_n}\rrbracket}I_{\llbracket{0,\tau_i}\rrbracket}
\end{align*}
which, by \eqref{XYT1}, implies $(H^{(n)})^{T_n\wedge \tau_i}=H^{T_n\wedge \tau_i}$. And similarly, we also have $(X^{(n)})^{T_n\wedge \tau_i}=X^{T_n\wedge \tau_i}$ for every $n\in \mathbb{N}^+$ and $i\in \mathbb{N}^+$. For $i, k, n\in\mathbb{N}^+$ with $n\leq k$, it is easy to see
\begin{align*}
(H^{(n)}.X^{(n)})I_{\llbracket{0,T_n}\rrbracket}I_{\llbracket{0,\tau_i}\rrbracket}
=&(H^{T_n\wedge \tau_i}.X^{T_n\wedge \tau_i})I_{\llbracket{0,T_n}\rrbracket}I_{\llbracket{0,\tau_i}\rrbracket}\\
=&(H^{T_k\wedge \tau_i}.X^{T_k\wedge \tau_i})^{T_n}I_{\llbracket{0,T_n}\rrbracket}I_{\llbracket{0,\tau_i}\rrbracket}\\
=&(H^{(k)}.X^{(k)})I_{\llbracket{0,T_n}\rrbracket}I_{\llbracket{0,\tau_i}\rrbracket}.
\end{align*}
Using the fact
\[
B=\bigcup\limits_{i=1}^{{+\infty}}\llbracket{0,\tau_i}\rrbracket = \llbracket{0}\rrbracket\cup \left(\bigcup\limits_{i=1}^{{+\infty}}\rrbracket{\tau_{i-1},\tau_i}\rrbracket\right), \]
we deduce that for $k, n\in\mathbb{N}^+$ with $n\leq k$,
\begin{align*}
&(H^{(n)}.X^{(n)})I_{B\llbracket{0,T_n}\rrbracket}\\
=&H_0X_0I_{\llbracket{0}\rrbracket}+\sum_{i=1}^{\infty}(H^{(n)}.X^{(n)})I_{\llbracket{0,T_n}\rrbracket}I_{\rrbracket{\tau_{i-1},\tau_i}\rrbracket}\\
=&H_0X_0I_{\llbracket{0}\rrbracket}+\sum_{i=1}^{\infty}\left((H^{(n)}.X^{(n)})I_{\llbracket{0,T_n}\rrbracket}I_{\llbracket{0,\tau_i}\rrbracket}
-(H^{(n)}.X^{(n)})I_{\llbracket{0,T_n}\rrbracket}I_{\llbracket{0,\tau_{i-1}}\rrbracket}\right)\\
=&H_0X_0I_{\llbracket{0}\rrbracket}+\sum_{i=1}^{\infty}\left((H^{(k)}.X^{(k)})I_{\llbracket{0,T_n}\rrbracket}I_{\llbracket{0,\tau_i}\rrbracket}
-(H^{(k)}.X^{(k)})I_{\llbracket{0,T_n}\rrbracket}I_{\llbracket{0,\tau_{i-1}}\rrbracket}\right)\\
=&H_0X_0I_{\llbracket{0}\rrbracket}+\sum_{i=1}^{\infty}(H^{(k)}.X^{(k)})I_{\llbracket{0,T_n}\rrbracket}I_{\rrbracket{\tau_{i-1},\tau_i}\rrbracket}\\
=&(H^{(k)}.X^{(k)})I_{B\llbracket{0,T_n}\rrbracket}.
\end{align*}
Then Remark \ref{remark-cs} shows that $(T_n,H^{(n)}.X^{(n)})$ is a CS for $L$. From $H^{(n)}.X^{(n)}\in \mathcal{S}$ for each $n\in \mathbb{N}^+$,
$(T_n,H^{(n)}.X^{(n)})$ is an FCS for $H_{\bullet}X\in\mathcal{S}^B$.

Finally, we prove the independence property of \eqref{HX-expression-4}. Suppose that $(S_n, \widetilde{H}^{(n)})$ for $H\in\mathcal{P}^{B}$ and $(\widetilde{S}_n,\widetilde{X}^{(n)})$ for $X\in\mathcal{S}^B$ are FCSs such that for each $n\in \mathbb{N}^+$, $\widetilde{H}^{(n)}\in\mathcal{L}(\widetilde{X}^{(n)})$. From Theorem \ref{process}, $(\widetilde{T}_n, \widetilde{H}^{(n)})$ is an FCS for $H\in\mathcal{P}^{B}$, and $(\widetilde{T}_n,\widetilde{X}^{(n)})$ is an FCS for $X\in\mathcal{S}^B$. Similarly, we can  prove that $(\widetilde{T}_n,\widetilde{H}^{(n)}.\widetilde{X}^{(n)})$ is an FCS for $H_{\bullet}X\in\mathcal{S}^B$. Then using the independence property of \eqref{x-expression}, we have $H_{\bullet}X=Z$.
\end{proof}

\begin{remark}
Let $X\in \mathcal{S}^B$ and $H\in \mathcal{L}^B(X)$.
\begin{itemize}
  \item [$(1)$] Let $(\tau_n)$ be an FS for $B$, and $X=M+A$ ($M\in (\mathcal{M}_{\mathrm{loc}})^B$ and $A\in (\mathcal{V}_0)^B$) be a decomposition of $X$. Suppose that $H^{\tau_n}\in\mathcal{L}_m(M^{\tau_n})$ and $H^{\tau_n}.A^{\tau_n}$ exists for each $n\in \mathbb{N}^+$. Using (1) of Theorem \ref{eq-HX} and noticing $H^{\tau_n}.X^{\tau_n}=H^{\tau_n}.M^{\tau_n}+H^{\tau_n}.A^{\tau_n}$ for each $n\in \mathbb{N}^+$, we deduce that $(\tau_n,H^{\tau_n}.M^{\tau_n}+H^{\tau_n}.A^{\tau_n})$ is an FCS for $H_{\bullet}X\in\mathcal{S}^B$, and $H_{\bullet}X$ can be expressed as
  \begin{equation*}
      H_{\bullet}X=\left((H_0X_0)I_{\llbracket{0}\rrbracket}+\sum\limits_{n=1}^{{+\infty}}
      (H^{\tau_n}.M^{\tau_n}+H^{\tau_n}.A^{\tau_n})I_{\rrbracket{\tau_{n-1},\tau_n}
      \rrbracket}\right)\mathfrak{I}_B,\quad \tau_0=0.
      \end{equation*}

  \item [$(2)$] Let $X=M+A$ ($M\in (\mathcal{M}_{\mathrm{loc}})^B$ and $A\in (\mathcal{V}_0)^B$) be a decomposition of $X$, and the sequences $(T_n,H^{(n)})$, $(T_n,M^{(n)})$ and $(T_n,A^{(n)})$ for  be FCSs for $H\in \mathcal{P}^B$, $M\in (\mathcal{M}_{\mathrm{loc}})^B$ and $A\in (\mathcal{V}_0)^B$, respectively. Suppose that $H^{(n)}\in\mathcal{L}_m(M^{(n)})$ and $H^{(n)}.A^{(n)}$ exists for each $n\in \mathbb{N}^+$. Using (2) of Theorem \ref{eq-HX} (with $X^{(n)}=M^{(n)}+A^{(n)}$ for each $n\in \mathbb{N}^+$), we deduce that   $(T_n,H^{(n)}.M^{(n)}+H^{(n)}.A^{(n)})$ is an FCS for $H_{\bullet}X\in\mathcal{S}^B$, and $H_{\bullet}X$ can be expressed as
  \begin{equation*}
      H_{\bullet}X=\left((H_0X_0)I_{\llbracket{0}\rrbracket}+\sum\limits_{n=1}^{{+\infty}}
      (H^{(n)}.M^{(n)}+H^{(n)}.A^{(n)})I_{\rrbracket{T_{n-1},T_n}
      \rrbracket}\right)\mathfrak{I}_B,\quad T_0=0.
      \end{equation*}
  \end{itemize}
\end{remark}

\begin{corollary}\label{bound-HX}
Let $H$ be a locally bounded predictable process on $B$, and $X\in\mathcal{S}^B$. Then
$H\in \mathcal{L}^B(X)$, and both $(T_n,H^{(n)}.X^{(n)})$ and $(\tau_n,H^{\tau_n}.X^{\tau_n})$ are FCSs for $H_{\bullet}X\in\mathcal{S}^B$, where $(T_n,H^{(n)})$ is an FCS for $H$ (a locally bounded predictable process on $B$), and $(T_n,X^{(n)})$ is an FCS for $X\in\mathcal{S}^B$, and $(\tau_n)$ is an FS for $B$.
\end{corollary}
\begin{proof}
Suppose that $(T_n, H^{(n)})$ is an FCS for $H$ (a locally bounded predictable process on $B$), and that $(T_n,X^{(n)})$ is an FCS for $X\in\mathcal{S}^{B}$. For each $n\in \mathbb{N}^+$, $H^{(n)}$ is integrable w.r.t. $X^{(n)}$ (see, e.g., Theorem I.4.31 in \cite{Jacod}). Then, by Theorems \ref{HX=} and \ref{eq-HX}, the statements hold true.
\end{proof}

From Theorems \ref{HAproperty} and  \ref{HM-p}, both the stochastic integrals $H_{\bullet}A$ and $H_{\bullet}M$ in \eqref{HX} have the linear and composite properties. Such linear and composite properties of the stochastic integral $H_{\bullet}X$ defined by \eqref{HX} are still valid, as the following theory illustrates.

\begin{theorem}\label{HX-p}
Let $X,\;Y\in \mathcal{S}^B$, and $H,\;K\in \mathcal{L}^B(X)$, and $H\in \mathcal{L}^B(Y)$, and $(\tau_n)$ be an FS for $B$, and $a,b\in\mathbb{R}$. Then we have following statements:
\begin{itemize}
  \item [$(1)$] $aH+bK\in\mathcal{L}^B(X)$, and in this case
  \begin{equation}\label{HX+}
      (aH+bK)_{\bullet}X=a(H_{\bullet}X)+b(K_{\bullet}X).
  \end{equation}
  Furthermore, $(\tau_n,(aH^{\tau_n}+bK^{\tau_n}).X^{\tau_n}=a(H^{\tau_n}.X^{\tau_n})+b(K^{\tau_n}.X^{\tau_n}))$ is an FCS for $(aH+bK)_{\bullet}X\in\mathcal{S}^B$.
  \item [$(2)$] $H\in\mathcal{L}^B(aX+bY)$, and in this case
  \begin{equation}\label{+HX}
      H_{\bullet}(aX+bY)=a(H_{\bullet}X)+b(H_{\bullet}Y).
  \end{equation}
  Furthermore, $(\tau_n,H^{\tau_n}.(aX^{\tau_n}+bY^{\tau_n})=a(H^{\tau_n}.X^{\tau_n})+b(H^{\tau_n}.Y^{\tau_n}))$ is an FCS for $H_{\bullet}(aX+bY)\in\mathcal{S}^B$.
  \item [$(3)$] $\widetilde{H}\in\mathcal{L}^B(H_{\bullet}X)\Leftrightarrow \widetilde{H}H\in\mathcal{L}^B(X)$.
  Furthermore, if $\widetilde{H}\in\mathcal{L}^B(H_{\bullet}X)$ (or equivalently, $\widetilde{H}H\in\mathcal{L}^B(X)$), then
  \begin{equation}\label{hHX}
     \widetilde{H}_{\bullet}(H_{\bullet}X)=(\widetilde{H}H)_{\bullet}X,
  \end{equation}
  and $(\tau_n,\widetilde{H}^{\tau_n}.(H^{\tau_n}.X^{\tau_n})=(\widetilde{H}^{\tau_n}H^{\tau_n}).X^{\tau_n})$ is an FCS for $(\widetilde{H}H)_{\bullet}X=\widetilde{H}_{\bullet}(H_{\bullet}X)\in\mathcal{S}^B$.
\end{itemize}
\end{theorem}
\begin{proof}
$(1)$ From Theorem \ref{eq-HX},  $(\tau_n,H^{\tau_n}.X^{\tau_n})$ is an FCS for $H_{\bullet}X\in\mathcal{S}^B$, and $(\tau_n,K^{\tau_n}.X^{\tau_n})$ is an FCS for $K_{\bullet}X\in\mathcal{S}^B$. For each $n\in \mathbb{N}^+$, $aH^{\tau_n}+bK^{\tau_n}\in\mathcal{L}(X^{\tau_n})$ (see Lemma \ref{property}). Noticing $(aH+bK)^{\tau_n}=aH^{\tau_n}+bK^{\tau_n}$ and using Theorems \ref{HX=} and \ref{eq-HX}, we deduce that $aH+bK\in\mathcal{L}^B(X)$, and that $(\tau_n,(aH^{\tau_n}+bK^{\tau_n}).X^{\tau_n})$ is an FCS for $(aH+bK)_{\bullet}X\in\mathcal{S}^B$. In this case, we can obtain that for each $n\in \mathbb{N}^+$,
\begin{align*}
((aH+bK)_{\bullet}X)I_{\llbracket{0,\tau_n}\rrbracket}
&=((aH^{\tau_n}+bK^{\tau_n}).X^{\tau_n})I_{\llbracket{0,\tau_n}\rrbracket}\\
&=(a(H^{\tau_n}.X^{\tau_n})+b(K^{\tau_n}.X^{\tau_n}))I_{\llbracket{0,\tau_n}\rrbracket}\\
&=(a(H_{\bullet}X)+b(K_{\bullet}X))I_{\llbracket{0,\tau_n}\rrbracket},
\end{align*}
which, by the statement $(1)$ of Theorem \ref{process-FS}, indicates \eqref{HX+}.

$(2)$ From Theorem \ref{eq-HX},  $(\tau_n,H^{\tau_n}.X^{\tau_n})$ is an FCS for $H_{\bullet}X\in\mathcal{S}^B$, and $(\tau_n,H^{\tau_n}.Y^{\tau_n})$ is an FCS for $H_{\bullet}Y\in\mathcal{S}^B$. For each $n\in \mathbb{N}^+$, $H^{\tau_n}\in\mathcal{L}(aX^{\tau_n}+bY^{\tau_n})$ (see Lemma \ref{property}). Noticing $(aX+bY)^{\tau_n}=aX^{\tau_n}+bY^{\tau_n}$ and using Theorems \ref{HX=} and \ref{eq-HX}, we deduce that $H\in\mathcal{L}^B(aX+bY)$, and that $(\tau_n,H^{\tau_n}.(aX^{\tau_n}+bY^{\tau_n}))$ is an FCS for $H_{\bullet}(aX+bY)\in\mathcal{S}^B$. In this case, we have that for each $n\in \mathbb{N}^+$,
\begin{align*}
(H_{\bullet}(aX+bY))I_{\llbracket{0,\tau_n}\rrbracket}
&=(H^{\tau_n}.(aX^{\tau_n}+bY^{\tau_n}))I_{\llbracket{0,\tau_n}\rrbracket}\\
&=(a(H^{\tau_n}.X^{\tau_n})+b(H^{\tau_n}.Y^{\tau_n}))I_{\llbracket{0,\tau_n}\rrbracket}\\
&=(a(H_{\bullet}X)+b(H_{\bullet}Y))I_{\llbracket{0,\tau_n}\rrbracket},
\end{align*}
which, by the statement $(1)$ of Theorem \ref{process-FS}, indicates \eqref{+HX}.

$(3)$ From Theorem \ref{eq-HX},  $(\tau_n,H^{\tau_n}.X^{\tau_n})$ is an FCS for $H_{\bullet}X\in\mathcal{S}^B$, and then $(H_{\bullet}X)^{\tau_n}=H^{\tau_n}.X^{\tau_n}$ for each $n\in \mathbb{N}^+$. Noticing $(\widetilde{H}H)^{\tau_n}=\widetilde{H}^{\tau_n}H^{\tau_n}$ and using Theorems \ref{HX=} and \ref{eq-HX}, the relation $\widetilde{H}\in\mathcal{L}^B(H_{\bullet}X)\Leftrightarrow \widetilde{H}H\in\mathcal{L}^B(X)$ can be proved by the following relations
\[
\widetilde{H}^{\tau_n}\in\mathcal{L}((H_{\bullet}X)^{\tau_n})\Leftrightarrow
\widetilde{H}^{\tau_n}\in\mathcal{L}(H^{\tau_n}.X^{\tau_n})\Leftrightarrow
\widetilde{H}^{\tau_n}H^{\tau_n}\in\mathcal{L}(X^{\tau_n})\Leftrightarrow
(\widetilde{H}H)^{\tau_n}\in\mathcal{L}(X^{\tau_n}),\quad n\in \mathbb{N}^+.
\]

Suppose $\widetilde{H}H\in\mathcal{L}^B(X)$. In this case, using Lemma \ref{property}, we deduce that for each $n\in \mathbb{N}^+$,
\begin{align*}
(\widetilde{H}_{\bullet}(H_{\bullet}X))I_{\llbracket{0,\tau_n}\rrbracket}
&=(\widetilde{H}^{\tau_n}.(H_{\bullet}X)^{\tau_n})I_{\llbracket{0,\tau_n}\rrbracket}\\
&=(\widetilde{H}^{\tau_n}.(H^{\tau_n}.X^{\tau_n}))I_{\llbracket{0,\tau_n}\rrbracket}\\
&=((\widetilde{H}^{\tau_n}H^{\tau_n}).X^{\tau_n})I_{\llbracket{0,\tau_n}\rrbracket}\\
&=((\widetilde{H}H)^{\tau_n}.X^{\tau_n})I_{\llbracket{0,\tau_n}\rrbracket}\\
&=((\widetilde{H}H)_{\bullet}X)I_{\llbracket{0,\tau_n}\rrbracket},
\end{align*}
which, by the statement $(1)$ of Theorem \ref{process-FS}, indicates \eqref{hHX}.
From Theorems \ref{HX=} and \ref{eq-HX}, for each $n\in \mathbb{N}^+$, $(\widetilde{H}H)^{\tau_n}\in\mathcal{L}(X^{\tau_n})$ and $H^{\tau_n}\in\mathcal{L}(X^{\tau_n})$.
For each $n\in \mathbb{N}^+$, it is easy to see that $(\widetilde{H}H)^{\tau_n}=\widetilde{H}^{\tau_n}H^{\tau_n}$ and $(\widetilde{H}^{\tau_n}H^{\tau_n}).X^{\tau_n}=\widetilde{H}^{\tau_n}.(H^{\tau_n}.X^{\tau_n})$. Then  Theorems \ref{HX=} and \ref{eq-HX} show that $(\tau_n,\widetilde{H}^{\tau_n}.(H^{\tau_n}.X^{\tau_n})=(\widetilde{H}^{\tau_n}H^{\tau_n}).X^{\tau_n})$ is an FCS for $(\widetilde{H}H)_{\bullet}M=\widetilde{H}_{\bullet}(H_{\bullet}M)\in\mathcal{S}^B$.
\end{proof}

\begin{remark}
Let the conditions in Theorem \ref{HX-p} hold.
\begin{itemize}
  \item [$(1)$] From \eqref{HX+}, $(T_n,a(H^{(n)}.X^{(n)})+b(K^{(n)}.\widetilde{X}^{(n)}))$ is also an FCS for $(aH+bK)_{\bullet}X\in\mathcal{S}^{B}$, where $(T_n,H^{(n)})$ and $(T_n,X^{(n)})$ are FCSs for $H\in\mathcal{P}^B$ and $X\in\mathcal{S}^{B}$ respectively such that for each $n\in \mathbb{N}^+$, $H^{(n)}\in \mathcal{L}(X^{(n)})$, and where $(T_n,K^{(n)})$ and $(T_n,\widetilde{X}^{(n)})$ are FCSs for $K\in\mathcal{P}^B$ and $X\in\mathcal{S}^{B}$ respectively such that for each $n\in \mathbb{N}^+$, $K^{(n)}\in \mathcal{L}(\widetilde{X}^{(n)})$.
  \item [$(2)$] From \eqref{+HX}, $(T_n,a(H^{(n)}.X^{(n)})+b(\widetilde{H}^{(n)}.Y^{(n)}))$ is also an FCS for $H_{\bullet}(aX+bY)\in\mathcal{S}^{B}$, where $(T_n,H^{(n)})$ and $(T_n,X^{(n)})$ are FCSs for $H\in\mathcal{P}^B$ and $X\in\mathcal{S}^{B}$ respectively such that for each $n\in \mathbb{N}^+$, $H^{(n)}\in \mathcal{L}(X^{(n)})$, and where $(T_n,\widetilde{H}^{(n)})$ and $(T_n,Y^{(n)})$ are FCSs for $H\in\mathcal{P}^B$ and $Y\in\mathcal{S}^{B}$ respectively such that for each $n\in \mathbb{N}^+$, $\widetilde{H}^{(n)}\in \mathcal{L}(Y^{(n)})$.
  \item [$(3)$] Suppose $\widetilde{H}H\in\mathcal{L}^B(X)$. Then from \eqref{hHX}, $(T_n,(\widetilde{H}^{(n)}H^{(n)}).X^{(n)}=\widetilde{H}^{(n)}.(H^{(n)}.X^{(n)}))$ is also an FCS for $(\widetilde{H}H)_{\bullet}X=\widetilde{H}_{\bullet}(H_{\bullet}X)\in\mathcal{S}^{B}$, where $(T_n,\widetilde{H}^{(n)})$, $(T_n,H^{(n)})$ and $(T_n,X^{(n)})$ are FCSs for $\widetilde{H}\in\mathcal{P}^B$, $H\in\mathcal{P}^B$ and $X\in\mathcal{S}^{B}$ respectively such that for each $n\in \mathbb{N}^+$, $H^{(n)}\in \mathcal{L}(X^{(n)})$ and $\widetilde{H}^{(n)}H^{(n)}\in \mathcal{L}(X^{(n)})$.
\end{itemize}
\end{remark}

Other fundamental properties of the stochastic integral $H_{\bullet}X$ in Definition \ref{de-HX} are presented in the following theorem.
\begin{theorem}\label{HX-property}
Let $X\in \mathcal{S}^B$ and $H\in\mathcal{L}^B(X)$. Then we have following statements:
\begin{itemize}
  \item [$(1)$] $(H_{\bullet}X)^c=H_{\bullet}X^c$, $\Delta (H_{\bullet}X)=H\Delta X$, and $(H_{\bullet}X)I_{\llbracket{0}\rrbracket}=HXI_{\llbracket{0}\rrbracket}$.
  \item [$(2)$] $(H_{\bullet}X)^\tau\mathfrak{I}_{B}=H_{\bullet}(X^\tau\mathfrak{I}_{B})
      =(HI_{\llbracket{0,\tau}\rrbracket}\mathfrak{I}_{B})_{\bullet}X=(H^\tau\mathfrak{I}_{B})_{\bullet}(X^\tau\mathfrak{I}_{B})$,
  where $\tau$ is a stopping time on $B$.
  \item [$(3)$] If $\widetilde{H}$ is a predictable process on $B$ satisfying $|\widetilde{H}|\leq |H|$, then $\widetilde{H}\in\mathcal{L}^B(X)$.
  \item [$(4)$] For any $Y\in \mathcal{S}^B$, we have
\begin{equation}\label{HXY-p}
  [H_{\bullet}X,Y]=H_{\bullet}[X,Y].
\end{equation}
\end{itemize}
\end{theorem}
\begin{proof}
Let $(\tau_n)$ be an FS for $B$, and let $X=M+A$ ($M\in (\mathcal{M}_{\mathrm{loc}})^B$ and $A\in (\mathcal{V}_0)^B$) be a decomposition such that $H\in \mathcal{L}_m^{B}(M)$ and $H_{\bullet}A$ exists. Definition \ref{de-HX} shows $H_{\bullet}X=H_{\bullet}M+H_{\bullet}A$.

$(1)$ The fact $(H_{\bullet}X)I_{\llbracket{0}\rrbracket}=HXI_{\llbracket{0}\rrbracket}$ can be obtained by \eqref{HX-expression-2} easily, and by Theorem \ref{HAproperty1} and \ref{HM-property}, the fact $\Delta (H_{\bullet}X)=H\Delta X$ can be proved by
\[
\Delta (H_{\bullet}X)=\Delta (H_{\bullet}M)+\Delta (H_{\bullet}A)=H\Delta M+H\Delta A=H\Delta X.
\]
Then it remains to prove $(H_{\bullet}X)^c=H_{\bullet}X^c$. Theorem \ref{uXc} shows $X^c=M^c$, and Theorem \ref{HM-p} and \ref{HM-property} imply
\[
H_{\bullet}X=H_0X_0\mathfrak{I}_B+H_{\bullet}M^c+H_{\bullet}M^d+H_{\bullet}A
\]
with $H_{\bullet}M^c\in(\mathcal{M}^c_{\mathrm{loc},0})^B$, $H_{\bullet}M^d\in(\mathcal{M}^d_{\mathrm{loc}})^B$, and $(H_{\bullet}M)^c=H_{\bullet}M^c$. By Theorem \ref{uXc}, the continuous part of $H_{\bullet}X\in \mathcal{S}^B$ can be expressed as $(H_{\bullet}M)^c=H_{\bullet}M^c$. Thus, by the uniqueness, $(H_{\bullet}X)^c=H_{\bullet}M^c=H_{\bullet}X^c$.

$(2)$ Theorem \ref{HAproperty1} yields
\[
(H_{\bullet}A)^\tau\mathfrak{I}_{B}=H_{\bullet}(A^\tau\mathfrak{I}_{B})
      =(HI_{\llbracket{0,\tau}\rrbracket}\mathfrak{I}_{B})_{\bullet}A=(H^\tau\mathfrak{I}_{B})_{\bullet}(A^\tau\mathfrak{I}_{B})
\]
and Theorem \ref{HM-property} yields
\[
(H_{\bullet}M)^\tau\mathfrak{I}_{B}=H_{\bullet}(M^\tau\mathfrak{I}_{B})
      =(HI_{\llbracket{0,\tau}\rrbracket}\mathfrak{I}_{B})_{\bullet}M=(H^\tau\mathfrak{I}_{B})_{\bullet}(M^\tau\mathfrak{I}_{B}).
\]
Then we can obtain the equalities by
\[
(H_{\bullet}X)^\tau\mathfrak{I}_{B}=(H_{\bullet}M)^\tau\mathfrak{I}_{B}
+(H_{\bullet}A)^\tau\mathfrak{I}_{B}=\left\{
\begin{aligned}
&H_{\bullet}(M^\tau\mathfrak{I}_{B})+H_{\bullet}(A^\tau\mathfrak{I}_{B})=H_{\bullet}(X^\tau\mathfrak{I}_{B}),\\
&(HI_{\llbracket{0,\tau}\rrbracket}\mathfrak{I}_{B})_{\bullet}M+(HI_{\llbracket{0,\tau}\rrbracket}\mathfrak{I}_{B})_{\bullet}A
=(HI_{\llbracket{0,\tau}\rrbracket}\mathfrak{I}_{B})_{\bullet}X,\\
&(H^\tau\mathfrak{I}_{B})_{\bullet}(A^\tau\mathfrak{I}_{B})+(H^\tau\mathfrak{I}_{B})_{\bullet}(M^\tau\mathfrak{I}_{B})
=(H^\tau\mathfrak{I}_{B})_{\bullet}(X^\tau\mathfrak{I}_{B}).
\end{aligned}
\right.
\]

$(3)$ Theorem \ref{eq-HX} shows that $(\tau_n,H^{\tau_n}.X^{\tau_n})$ is an FCS for $H_{\bullet}X\in\mathcal{S}^B$, and Theorem \ref{fcs-p} shows that $(\tau_n,\widetilde{H}^{\tau_n})$ and $(\tau_n,X^{\tau_n})$ are FCSs for $\widetilde{H}\in\mathcal{P}^B$ and $X\in\mathcal{S}^B$ respectively.
For each $n\in \mathbb{N}^+$, the relation $|\widetilde{H}|\leq |H|$ implies $|\widetilde{H}^{\tau_n}|\leq|H^{\tau_n}|$, and then $\widetilde{H}^{\tau_n}\in \mathcal{L}(X^{\tau_n})$ is obtained by Theorem 9.15 in \cite{He}. Consequently, Theorem \ref{HX=} yields $\widetilde{H}\in\mathcal{L}^B(X)$.

$(4)$ Theorem \ref{fcs-p} shows that $(\tau_n,Y^{\tau_n})$ is an FCS for $Y\in\mathcal{S}^B$, and Theorem \ref{eq-HX} shows that $(\tau_n,H^{\tau_n}.X^{\tau_n})$ is an FCS for $H_{\bullet}X\in\mathcal{S}^B$.
For each $n\in \mathbb{N}^+$, by Theorem 9.15 in \cite{He} and Theorem \ref{[X,Y]-fcs}, we deduce
\begin{equation}\label{HXY-1}
[H_{\bullet}X,Y]I_{\llbracket{0,\tau_n}\rrbracket}
=[H^{\tau_n}.X^{\tau_n},Y^{\tau_n}]I_{\llbracket{0,\tau_n}\rrbracket}
=\left(H^{\tau_n}.[X^{\tau_n},Y^{\tau_n}]\right)I_{\llbracket{0,\tau_n}\rrbracket}
=\left(H^{\tau_n}.[X,Y]^{\tau_n}\right)I_{\llbracket{0,\tau_n}\rrbracket}.
\end{equation}
From the existence of $H^{\tau_n}.[X,Y]^{\tau_n}$, Theorem \ref{HA-equivalent-p} shows the existence of $H_{\bullet}[X,Y]$, and then Theorem \ref{HA-FCS-p} shows that $(\tau_n,H^{\tau_n}.[X,Y]^{\tau_n})$ is an FCS for $H_{\bullet}[X,Y]\in\mathcal{V}^B$ satisfying
\begin{equation}\label{HXY-2}
\left(H^{\tau_n}.[X,Y]^{\tau_n}\right)I_{\llbracket{0,\tau_n}\rrbracket}
=\left(H_{\bullet}[X,Y]\right)I_{\llbracket{0,\tau_n}\rrbracket}, \quad n\in \mathbb{N}^+.
\end{equation}
Combining \eqref{HXY-1} and \eqref{HXY-2} leads to
\[
[H_{\bullet}X,Y]I_{\llbracket{0,\tau_n}\rrbracket}=\left(H_{\bullet}[X,Y]\right)I_{\llbracket{0,\tau_n}\rrbracket}, \quad n\in \mathbb{N}^+,
\]
which, by the statement $(1)$ of Theorem \ref{process-FS}, implies \eqref{HXY-p}.
\end{proof}

It\^{o} Formula, or the change-of-variable formula, is one of the most important tools in the study of stochastic calculus. In the following theory, we present the It\^{o} Formula for semimartingales on $B$ which not only states that a ``smooth function" of a semimartingale on $B$ is still a semimartingale on $B$, but also provides its decomposition.
\begin{theorem}\label{ito}
Let $X_{1},X_{2}\cdots, X_{d}$ ($d\in \mathbb{N}^+$) be semimartingales on $B$, and $F$ be a $C^2$-function on $\mathbb{R}^d$ (i.e., $F$ has continuous partial derivatives of the first and the second orders). Put $Z=(X_{1},X_{2}\cdots X_{d})$. Then
\begin{align}\label{ito-eq}
F(Z)-F(Z(0))\mathfrak{I}_B=\sum_{i=1}^d D_iF(Z_{-})_{\bullet}(X_i-X_i(0)\mathfrak{I}_B)+\eta+
\frac{1}{2}\sum_{i,j=1}^d D_{ij}F(Z_{-})_{\bullet}\langle X_i^c,X_j^c\rangle,
\end{align}
where $D_iF=\frac{\partial F}{\partial x_i}$, $D_{ij}F=\frac{\partial^2 F}{\partial x_i\partial x_j}$ and
\[
\eta=\Sigma\bigg(F(Z)-F(Z_{-})-\sum_{i=1}^d D_iF(Z_{-})\Delta X_i\bigg).
\]
\end{theorem}
\begin{proof}
From Corollaries \ref{SP-} and \ref{bound-HX}, stochastic integrals in \eqref{ito-eq} are well defined. Let $(\tau_n)$ be an FS for $B$ and $Z^{\tau_n}=(X_{1}^{\tau_n},X_{2}^{\tau_n}\cdots X_{d}^{\tau_n})$. For each $n\in \mathbb{N}^+$, from \eqref{eqXT} and \eqref{sigmaX}, it is easy to see
\begin{align*}
\eta^{\tau_n}I_{\llbracket{0,\tau_n}\rrbracket}
&=\Sigma\bigg(\Delta F(Z)I_{\llbracket{0,\tau_n}\rrbracket}-\sum_{i=1}^d D_iF(Z_{-})\Delta X_iI_{\llbracket{0,\tau_n}\rrbracket}\bigg)I_{\llbracket{0,\tau_n}\rrbracket}\\
&=\Sigma\bigg(\Delta F(Z^{\tau_n})-\sum_{i=1}^d D_iF((Z^{\tau_n})_{-})\Delta X_i^{\tau_n}\bigg)I_{\llbracket{0,\tau_n}\rrbracket}.
\end{align*}
Then we deduce that for each $n\in \mathbb{N}^+$,
\begin{align*}
&\big(F(Z)-F(Z(0))\mathfrak{I}_B\big)I_{\llbracket{0,\tau_n}\rrbracket}\\
=&\big(F(Z^{\tau_n})-F(Z^{\tau_n}(0))\big)I_{\llbracket{0,\tau_n}\rrbracket}\\
=&\bigg(\sum_{i=1}^d D_iF((Z^{\tau_n})_{-}).(X_i^{\tau_n}-X_i^{\tau_n}(0))+\eta^{\tau_n}+
\frac{1}{2}\sum_{i,j=1}^d D_{ij}F((Z^{\tau_n})_{-}).\langle(X_i^{\tau_n})^c,(X_j^{\tau_n})^c\rangle\bigg)I_{\llbracket{0,\tau_n}\rrbracket}\\
=&\bigg(\sum_{i=1}^d (D_iF((Z^{\tau_n})_{-})I_{\llbracket{0,\tau_n}\rrbracket}).(X_i^{\tau_n}-X_i^{\tau_n}(0))+\eta^{\tau_n}+
\frac{1}{2}\sum_{i,j=1}^d (D_{ij}F((Z^{\tau_n})_{-})I_{\llbracket{0,\tau_n}\rrbracket}).\langle(X_i^{\tau_n})^c,(X_j^{\tau_n})^c\rangle\bigg)I_{\llbracket{0,\tau_n}\rrbracket}\\
=&\bigg(\sum_{i=1}^d (D_iF(Z_{-}))^{\tau_n}.(X_i-X_i(0)\mathfrak{I}_B)^{\tau_n}+\eta^{\tau_n}+
\frac{1}{2}\sum_{i,j=1}^d (D_{ij}F(Z_{-}))^{\tau_n}.\langle(X_i)^c,(X_j)^c\rangle^{\tau_n}\bigg)I_{\llbracket{0,\tau_n}\rrbracket}\\
=&\bigg(\sum_{i=1}^d D_iF(Z_{-})_{\bullet}(X_i-X_i(0)\mathfrak{I}_B)+\eta+
\frac{1}{2}\sum_{i,j=1}^d D_{ij}F(Z_{-})_{\bullet}\langle X_i^c,X_j^c\rangle\bigg)I_{\llbracket{0,\tau_n}\rrbracket},
\end{align*}
where the second equality comes from the It\^{o} Formula for semimartingales (see Theorem 9.35 in \cite{He}), the third equality from Lemma \ref{property}, the forth equality from Theorems \ref{X-left} and \ref{property-qr-p}, and the last equality from Theorems \ref{HA-FCS-p} and \ref{eq-HX}.
Thus, (\ref{ito-eq}) is obtained from Theorem \ref{process-FS}.
\end{proof}

In Theorem \ref{ito}, the process $Z=(X_{1},X_{2}\cdots X_{d})$ is called a $d$-dimensional semimartingale on $B$. By means of Definition \ref{processB}, we can also define a $d$-dimensional semimartingale on $B$ equivalently as follows: $Z$ is called a $d$-dimensional semimartingale on $B$, if there exists a CS $(T_n,Z^{(n)})$ for $Z$ such that for each $n\in \mathbb{N}^+$, $Z^{(n)}$ is a $d$-dimensional semimartingale. Specially, if $d=1$, then It\^{o} formula \eqref{ito-eq} becomes
\[
f(X)-f(X_0)\mathfrak{I}_B=f'(X_{-})_{\bullet}(X-X(0)\mathfrak{I}_B)+\Sigma\big(f(X)-f(X_{-})-f'(X_{-})\Delta X\big)+\frac{1}{2}f''(X_{-})_{\bullet}\langle X^c\rangle,
\]
where $X$ is a semimartingale on $B$, and $f$ is a $C^2$-function on $\mathbb{R}$ (i.e., $f$ has continuous derivatives of the first order $f'$ and the second order $f''$).

The following two corollaries are important applications of It\^{o} formula \eqref{ito-eq}: the fomer presents the formula of integral by Parts for two semimartingales on $B$, and the later studies a simple stochastic differential equation on $B$.
\begin{corollary}\label{IbP-p}
Let $X,\;Y\in\mathcal{S}^B$. Then
\begin{equation}\label{IbP-eq}
XY=(X_{-})_{\bullet}Y+(Y_{-})_{\bullet}X+[X,Y]-2X_0Y_0\mathfrak{I}_B.
\end{equation}
\end{corollary}
\begin{proof}
Applying Theorem \ref{ito} with $d=2$, $Z=(X,Y)$ and $F(x,y)=xy$ yields
\[
XY-X_0Y_0\mathfrak{I}_B=(X_{-})_{\bullet}(Y-Y_0\mathfrak{I}_B)+(Y_{-})_{\bullet}(X-X_0\mathfrak{I}_B)
+\langle X^c,Y^c\rangle+\Sigma(XY-X_{-}Y_{-}-X_{-}\Delta Y-Y_{-}\Delta X).
\]
From \eqref{HX-expression-2}, it is easy to see
\[
(X_{-})_{\bullet}(Y_0\mathfrak{I}_B)=(Y_{-})_{\bullet}(X_0\mathfrak{I}_B)=X_0Y_0\mathfrak{I}_B,
\]
and from the relation $\Delta X\Delta Y=XY-X_{-}Y_{-}-X_{-}\Delta Y-Y_{-}\Delta X$, Definition
\ref{[X,Y]} yields
\[
\langle X^c,Y^c\rangle+\Sigma(XY-X_{-}Y_{-}-X_{-}\Delta Y-Y_{-}\Delta X)
=[X,Y]-X_0Y_0\mathfrak{I}_B.
\]
Then \eqref{IbP-eq} is obtained by the statement $(1)$ of Theorem \ref{HX-p}.
\end{proof}

\begin{corollary}\label{expZ}
Let $Z$ be a continuous semimartingale (see pp. 337 in \cite{Kallenberg}) on $B$ satisfying $Z_0=0$. Put
 \begin{equation*}
     S=s_0\exp\left\{Z-\frac{1}{2}\langle Z^c\rangle\right\},
 \end{equation*}
where $s_0$ is a positive constant.
Then $S$ is the unique semimartingale on $B$ satisfying
 \begin{equation}\label{SDE}
    S=s_0\mathfrak{I}_B+S_{\bullet}Z.
 \end{equation}
\end{corollary}
\begin{proof}
Put $X=Z-\frac{1}{2}\langle Z^c\rangle$, and it is obvious that $X\in \mathcal{S}^B$ satisfying $\Delta X=0$ (see Corollary \ref{qr-pc}), $X^c=Z^c$ and $X_0=0$. Applying Theorem \ref{ito} with $d=1$ and $F(x)=s_0e^x$, and using Theorem \ref{HX-p}, we deduce
\begin{align*}
F(X)=&F(X_0)\mathfrak{I}_B+F'(X)_{\bullet}X+\frac{1}{2}F''(X)_{\bullet}\langle X^c\rangle\\
=&s_0\mathfrak{I}_B+S_{\bullet}\left(Z-\frac{1}{2}\langle Z^c\rangle\right)+\frac{1}{2}S_{\bullet}\langle Z^c\rangle\\
=&s_0\mathfrak{I}_B+S_{\bullet}Z,
\end{align*}
which shows that $S$ is a semimartingale on $B$ satisfying \eqref{SDE}.

Assume that $\widetilde{S}$ is another semimartingale on $B$ satisfying $\widetilde{S}=s_0\mathfrak{I}_B+\widetilde{S}_{\bullet}Z$. Let $(\tau_n)$ is an FS for $B$. Theorem \ref{eq-HX} shows that $(\tau_n,S^{\tau_n}.Z^{\tau_n})$ is an FCS for $S_{\bullet}Z\in S^B$, and that $(\tau_n,\widetilde{S}^{\tau_n}.Z^{\tau_n})$ is an FCS for $\widetilde{S}_{\bullet}Z\in S^B$. Then for each $n\in \mathbb{N}^+$, $S^{\tau_n}$ is a semimartingale satisfying $S^{\tau_n}=s_0+{S^{\tau_n}}.Z^{\tau_n}$, and $\widetilde{S}^{\tau_n}$ is a semimartingale satisfying $\widetilde{S}^{\tau_n}=s_0+{\widetilde{S}^{\tau_n}}.Z^{\tau_n}$. However, Dol\'{e}an-Dade exponential formula (see Theorem 9.39 in \cite{He}) implies $S^{\tau_n}=\widetilde{S}^{\tau_n}$ for each $n\in \mathbb{N}^+$. Therefore, by Theorem \ref{process-FS}, we obtain the relation $S=\widetilde{S}$, i.e., the uniqueness of $S\in\mathcal{S}^B$.
\end{proof}

Finally, we give two examples of the stochastic integral $H_{\bullet}X$ defined in Definition \ref{de-HX}.
\begin{example}
Let $\widetilde{H}$ be a locally bounded predictable process, $H=\widetilde{H}\mathfrak{I}_B$, and $X\in\mathcal{S}^B$ be given by \eqref{general-eX} in Example \ref{general-X}. Obviously, $H$ is a locally bounded predictable process on $B$ with the FCS $(\tau_n,\widetilde{H})$.
Then we have the following statements:
\begin{itemize}
  \item [$(1)$] From Corollary \ref{bound-HX}, $H\in \mathcal{L}^B(X)$, and $(\tau_n,\widetilde{H}.X^{(n)})$ is an FCS for $H_{\bullet}X\in\mathcal{S}^B$.
      From Theorem \ref{eq-HX}, $H_{\bullet}X$ can be expressed as
      \begin{equation*}
      H_{\bullet}X=\left((H_0X_0)I_{\llbracket{0}\rrbracket}+\sum\limits_{n=1}^{{+\infty}}
      (\widetilde{H}.X^{(n)})I_{\rrbracket{\tau_{n-1},\tau_n}
      \rrbracket}\right)\mathfrak{I}_B.
      \end{equation*}

  \item [$(2)$] From \eqref{DX} and the statement (1) of Theorem \ref{HX-property},
  \[
  \Delta (H_{\bullet}X)=H\Delta X=\left(\sum\limits_{n=1}^{{+\infty}}
      \widetilde{H}\Delta X^{(n)}I_{\rrbracket{\tau_{n-1},\tau_n}
      \rrbracket}\right)\mathfrak{I}_B.
  \]
  Equivalently, above expression of $\Delta (H_{\bullet}X)$ can be also obtained by using \eqref{x-expression} and the CS $(\tau_n,\Delta(\widetilde{H}.X^{(n)})=\widetilde{H}\Delta X^{(n)})$ for $\Delta (H_{\bullet}X)$ (see Theorem \ref{delta}).

  \item [$(3)$] From Theorem \ref{HX-property} and Example \ref{general-X}, the following relation is valid:
  \begin{equation*}
  H_{\bullet}X^c=\left(\sum\limits_{n=1}^{{+\infty}}
      (\widetilde{H}.(X^{(n)})^c)I_{\rrbracket{\tau_{n-1},\tau_n}
      \rrbracket}\right)\mathfrak{I}_B
      =\left(\sum\limits_{n=1}^{{+\infty}}
      (\widetilde{H}.X^{(n)})^cI_{\rrbracket{\tau_{n-1},\tau_n}
      \rrbracket}\right)\mathfrak{I}_B
      =(H_{\bullet}X)^c.
  \end{equation*}

  \item [$(4)$] Let $\tau$ be a stopping time on $B$. Using the definition of $(H_{\bullet}X)^\tau$ (or Theorem \ref{fcs}), it is easy to see
      \begin{align*}
      (H_{\bullet}X)^\tau&=(H_0X_0)I_{\llbracket{0}\rrbracket}+\sum\limits_{n=1}^{{+\infty}}
      (\widetilde{H}.X^{(n)})^\tau I_{\rrbracket{\tau_{n-1},\tau_n}\rrbracket}\\
      &=(H_0X_0)I_{\llbracket{0}\rrbracket}+\sum\limits_{n=1}^{{+\infty}}
      (\widetilde{H}.X^{(n)}) I_{\rrbracket{\tau_{n-1}\wedge\tau,\tau_n\wedge\tau}\rrbracket}.
      \end{align*}
      Then from Theorem \ref{HX-property},
      \begin{align*}
      (H_{\bullet}X)^\tau\mathfrak{I}_{B}&=H_{\bullet}(X^\tau\mathfrak{I}_{B})
      =(HI_{\llbracket{0,\tau}\rrbracket}\mathfrak{I}_{B})_{\bullet}X
      =(H^\tau\mathfrak{I}_{B})_{\bullet}(X^\tau\mathfrak{I}_{B})\\
      &=\left((H_0X_0)I_{\llbracket{0}\rrbracket}+\sum\limits_{n=1}^{{+\infty}}
      (\widetilde{H}.X^{(n)}) I_{\rrbracket{\tau_{n-1}\wedge\tau,\tau_n\wedge\tau}\rrbracket}\right)\mathfrak{I}_B.
      \end{align*}
\end{itemize}
\end{example}

\begin{example}
Let the assumptions in Example \ref{ex-absolute} hold true, and $X$ be an adapted c\`{a}dl\`{a}g process. Put $\widetilde{Z}=Z\mathfrak{I}_B$ and $\widetilde{X}=X\mathfrak{I}_B$. Let $\mathcal{M}_{\mathrm{loc}}(\mathbb{Q})$ be the set of all $\mathbb{Q}$-local martingales. Then we have the following relations:
   \begin{align}
      X\in\mathcal{M}_{\mathrm{loc}}(\mathbb{Q})\Leftrightarrow&\widetilde{X}\in\mathcal{S}^B\; \text{and}\; \widetilde{X}+ \bigg[\frac{1}{\widetilde{Z}_{-}}{}_{\bullet}\widetilde{X},\widetilde{Z}\bigg]\in (\mathcal{M}_{\mathrm{loc}})^B\nonumber\\
      \Leftrightarrow&
      \widetilde{X}\in\mathcal{S}^B\; \text{and}\; \widetilde{X}+\frac{1}{\widetilde{Z}_{-}} {}_{\bullet}[\widetilde{X},\widetilde{Z}]\in (\mathcal{M}_{\mathrm{loc}})^B,\label{ex-mloc-1}
  \end{align}
From Theorem 12.18 in \cite{He}, the condition $X\in\mathcal{M}_{\mathrm{loc}}(\mathbb{Q})$ is equivalent to
\begin{equation*}
 \widetilde{X}\in\mathcal{S}^B\; \text{and}\;\bigg(X+\frac{1}{Z_{-}} {}.[X,Z]\bigg)\mathfrak{I}_B\in (\mathcal{M}_{\mathrm{loc}})^B.
\end{equation*}
Since \eqref{HXY-p} shows
\[
\widetilde{X}+ \bigg[\frac{1}{\widetilde{Z}_{-}}{}_{\bullet}\widetilde{X},\widetilde{Z}\bigg]=\widetilde{X}+\frac{1}{\widetilde{Z}_{-}} {}_{\bullet}[\widetilde{X},\widetilde{Z}],
\]
it suffices to prove
\begin{equation}\label{Q1}
\widetilde{X}+\frac{1}{\widetilde{Z}_{-}} {}_{\bullet}[\widetilde{X},\widetilde{Z}]=\bigg(X+\frac{1}{Z_{-}} {}.[X,Z]\bigg)\mathfrak{I}_B.
\end{equation}
Let $(\tau_n)$ be an FS for $B$. Theorems \ref{HA-FCS-p} and \ref{[X,Y]-fcs} implies that the relation
\begin{align*}
\bigg(\widetilde{X}+\frac{1}{\widetilde{Z}_{-}} {}_{\bullet}[\widetilde{X},\widetilde{Z}]\bigg)I_{\llbracket{0,\tau_n}\rrbracket}
&=\bigg(\widetilde{X}^{\tau_n}+\frac{1}{(\widetilde{Z}^{\tau_n})_{-}} {}.[\widetilde{X}^{\tau_n},\widetilde{Z}^{\tau_n}]\bigg)I_{\llbracket{0,\tau_n}\rrbracket}\\
&=\bigg(X^{\tau_n}+\frac{1}{(Z^{\tau_n})_{-}} {}.[X^{\tau_n},Z^{\tau_n}]\bigg)I_{\llbracket{0,\tau_n}\rrbracket}\\
&=\bigg(X+\frac{1}{Z_{-}} {}.[X,Z]\bigg)I_{\llbracket{0,\tau_n}\rrbracket}
\end{align*}
holds for each $n\in \mathbb{N}^+$. This yields \eqref{Q1}, and proves the relations \eqref{ex-mloc-1}.
\end{example}

\section{Applications in finance}\label{section6}\noindent
\setcounter{equation}{0}
In this section, we apply stochastic integrals on PSITs to the study of investment in financial markets. Our main aim is to construct a financial market where the time-horizon of the investor is characterized by a PSIT, and where the dynamic price of the risky asset is extended to a semimartingale on such a PSIT. Recall that $(\Omega,\mathcal{F},\mathbb{P})$ is a probability space and that $\mathbb{F}=(\mathcal{F}_t,t\geq 0)$ is a filtration on that space.

\subsection{Financial markets on PSITs}

We first recall the classic financial market $(Y,\mathbb{F})$ where $Y$ is an $\mathbb{F}$-semimartingale representing the price of a risky asset, and where for simplicity, the savings account is an asset with constant price. At the time $t\in \mathbb{R}^+$, the investor holds $\vartheta_t$ shares of the risky asset and invests the rest of his wealth in the savings account, and then his wealth $X_t$ can be expressed as $X=\vartheta Y+(X-\vartheta Y)$, where $X_0=x$ is the initial wealth and $\vartheta$ is a predictable process with $\vartheta_0=0$. Let $L(Y,\mathbb{F})$ be the collection of all such $\mathbb{F}$-strategies $\vartheta\in \mathcal{L}(Y)$. Then we have the following essentials of mathematical finance (see, e.g. Subsection 1.4 in \cite{Aksamit}):
\begin{itemize}
  \item A strategy $\vartheta\in L(Y,\mathbb{F})$ is said to be self-financing if the wealth can be expressed as $X=x+\vartheta.Y$.

  \item Let $a>0$ be a constant. A strategy $\vartheta\in L(Y,\mathbb{F})$ is said to be $a$-admissible on the time horizon $[0,T\rangle$ (where $[0,T\rangle=[0,T]$ if $T\in]0,+\infty[$, and $[0,T\rangle=\mathbb{R}^+$ if $T=+\infty$), if $(\vartheta.Y)_t\geq -a$, $\mathbb{P}$-a.s. for all $t\in[0,T\rangle$.
      Let $l_a(Y,\mathbb{F},T)$ be the set of all $a$-admissible strategies on $[0,T\rangle$.

  \item A strategy $\vartheta\in L(Y,\mathbb{F})$ is said to be admissible on $[0,T\rangle$, if $\vartheta\in\bigcup_{a\in \mathbb{R}^+}l_a(Y,\mathbb{F},T)$. Let $l_0(Y,\mathbb{F},T)$ be the set of all admissible strategies on $[0,T\rangle$.

  \item The financial market $(Y,\mathbb{F})$ is said to satisfy no arbitrage (NA) on $[0,T\rangle$ if there does not exist any strategy  $\vartheta\in l_0(Y,\mathbb{F},T)$ such that
\[
(\vartheta.Y)_T\geq 0,\; \mathbb{P}\text{-\;}a.s. \quad\text{and}\quad \mathbb{P}((\vartheta.Y)_T>0)>0.
\]
\end{itemize}
If $T=\infty$, then a strategy which is $a$-admissible (resp. admissible) on $[0,+\infty[$ is also said to be $a$-admissible (resp. admissible), and a market which satisfies NA on $[0,+\infty[$ is also said to satisfy NA.

Now we start to construct a new financial market. Same as the classic financial market, a risky asset $S$ is traded in such a market, and the savings account is an asset with constant price. On the other hand,
the time-horizon of an investor in the market is uncertain but can be characterized by a predictable set $B$ of interval type, and the dynamic of the risky asset is a semimartingale on $B$, i.e., $S\in \mathcal{S}^B$. We denote such a financial market by the triplet $(S,\mathbb{F},B)$.

We use an example to explain the time-horizon. Assume that a risky asset with default is traded in the financial market. Let a predictable time $\tau$ represent the time when a default occurs in the credit risk setting, and the positive constant $T$ be the deterministic terminal time for the investor. Define the following PSIT
\begin{equation}\label{time}
  B=\llbracket{0,T}\rrbracket\llbracket{0,\tau}\llbracket.
\end{equation}
From the investor's point of view, the investment should be made strictly before the default time $\tau$ and not exceeding the terminal time $T$, and hence his time-horizon is characterized by $B$. The information (such as the price of the risky asset, the investment strategies and so on) on $B$ is sufficient for the investor to consider portfolio problems, but the information outside $B$ does not matter. Therefore, we can assume that $B$ is the stochastic time-horizon of the investor.

In the financial market $(S,\mathbb{F},B)$, at the time $t$ satisfying $(\omega,t)\in B$, the investor holds ${\vartheta}(\omega,t)$ shares of the risky asset and invests the rest of his wealth in the savings account. Then his wealth ${X}(\omega,t)$ can be expressed as
\begin{equation*}
{X}(\omega,t)={\vartheta}(\omega,t){S}(\omega,t)+({X}(\omega,t)-{\vartheta}(\omega,t){S}(\omega,t)), \quad (\omega,t)\in B,
\end{equation*}
or equivalently,
\begin{equation}\label{wealth}
{X}={\vartheta}{S}+({X}-{\vartheta}{S}),
\end{equation}
where ${\vartheta}I_{\llbracket{0}\rrbracket}=0$ and ${X}I_{\llbracket{0}\rrbracket}={x}_0$, $x_0>0$ is a constant, and ${\vartheta}$ is an $\mathbb{F}$-predictable process on $B$ and is called an $(\mathbb{F},B)$-strategy.

Let $(T_n,S^{(n)})$ be an FCS for $S\in\mathcal{S}^B$, and fix $n\in \mathbb{N}^+$. Then $(S^{(n)},\mathbb{F})$ remains a financial market with infinite time span which is associated with $(S,\mathbb{F},B)$. On $B\llbracket{0,T_n}\rrbracket$, the stock price $S^{(n)}$ in the financial market $(S^{(n)},\mathbb{F})$ is the same as in $(S,\mathbb{F},B)$, thereby leading to the same portfolio strategies of the investor (i.e., the strategy $\vartheta$ and the wealth $X$ in \eqref{wealth}). Therefore, it is reasonable to assume that, in the financial market $(S^{(n)},\mathbb{F})$, the investor's strategy $\vartheta^{(n)}$ and wealth $X^{(n)}$ satisfy
\begin{equation}\label{ConditionWealth}
\vartheta^{(n)}I_{B\llbracket{0,T_n}\rrbracket}
=\vartheta I_{B\llbracket{0,T_n}\rrbracket},\quad X^{(n)}I_{B\llbracket{0,T_n}\rrbracket}
=XI_{B\llbracket{0,T_n}\rrbracket}.
\end{equation}

Using the relationship between the financial market $(S,\mathbb{F},B)$ and the classic financial markets $(S^{(n)},\mathbb{F})$, $n\in \mathbb{N}^+$, we can define self-financing strategies and admissible strategies in the financial market $(S,\mathbb{F},B)$, and study whether the financial market  $(S,\mathbb{F},B)$ satisfy NA.

\begin{definition}
In the financial market $({S},\mathbb{F},B)$, suppose that the investor's wealth ${X}$ and trading strategy ${\vartheta}$ are given by (\ref{wealth}). Let $L({S},\mathbb{F},B)$ be the collection of all $(\mathbb{F},B)$-strategies ${\vartheta}\in \mathcal{L}^B({S})$ with ${\vartheta}I_{\llbracket{0}\rrbracket}=0$.
\begin{itemize}
  \item[$(1)$] A strategy ${\vartheta}\in L({S},\mathbb{F},B)$ is said to be self-financing if there exist FCSs $(T_n,\vartheta^{(n)})$ for ${\vartheta}\in \mathcal{P}^B$ and $(T_n,S^{(n)})$ for $S\in \mathcal{S}^B$ such that for each $n\in \mathbb{N}^+$, the strategy ${\vartheta}^{(n)}$ is self-financing in the financial market $({S}^{(n)},\mathbb{F})$.

  \item[$(2)$]  Let $a>0$ be a constant. A strategy ${\vartheta}\in L({S},\mathbb{F},B)$ is said to be $a$-admissible, if there exist FCSs $(T_n,\vartheta^{(n)})$ for ${\vartheta}\in \mathcal{P}^B$ and $(T_n,S^{(n)})$ for $S\in \mathcal{S}^B$ such that for each $n\in \mathbb{N}^+$, the strategy ${\vartheta}^{(n)}$ is $a$-admissible in the financial market $({S}^{(n)},\mathbb{F})$.
      Let $l_a({S},\mathbb{F},B)$ be the set of all $a$-admissible strategies in the financial market $({S},\mathbb{F},B)$.
      A strategy ${\vartheta}\in L({S},\mathbb{F},B)$ is said to be admissible, if $\vartheta\in\bigcup_{a\in \mathbb{R}^+}l_a({S},\mathbb{F},B)$.

  \item [$(3)$] The financial market $({S},\mathbb{F},B)$ is said to satisfy NA if there exists an FCS $(T_n,S^{(n)})$ for $S\in \mathcal{S}^B$ such that for each $n\in \mathbb{N}^+$, the financial market $({S}^{(n)},\mathbb{F})$ satisfies NA.
\end{itemize}
\end{definition}

Note that, ${\vartheta}\in L({S},\mathbb{F},B)$ if and only if there exists FCSs $(T_n,\vartheta^{(n)})$ for ${\vartheta}\in \mathcal{P}^B$ and $(T_n,S^{(n)})$ for $S\in \mathcal{S}^B$ satisfying ${\vartheta}^{(n)}\in L({S}^{(n)},\mathbb{F})$ for each $n\in \mathbb{N}^+$.

\begin{remark}
From Corollary \ref{cD=DB}, $\mathcal{P}=\mathcal{P}^{\llbracket{0,+\infty}\llbracket}$ and $\mathcal{S}=\mathcal{S}^{\llbracket{0,+\infty}\llbracket}$.  It is not hard to see that the financial market $(S,\mathbb{F},B)$ degenerates to the classical financial market $(S,\mathbb{F})$ if $B=\llbracket{0,+\infty}\llbracket=\Omega\times\mathbb{R}^+$:
\begin{enumerate}
  \item [$(1)$] A strategy $\vartheta$ is a self-financing strategy in the financial market $(S,\mathbb{F},\llbracket{0,+\infty}\llbracket)$ if and only if it is a self-financing strategy in the financial market $(S,\mathbb{F})$. The sufficiency is trivial, and we just show the necessity.
      Assume that $(T_n,\vartheta^{(n)})$ for ${\vartheta}\in \mathcal{P}^{\llbracket{0,+\infty}\llbracket}$ and $(T_n,S^{(n)})$ for $S\in \mathcal{S}^{\llbracket{0,+\infty}\llbracket}$ are FCSs such that for each $n\in \mathbb{N}^+$, the strategy ${\vartheta}^{(n)}$ is self-financing in the financial market $({S}^{(n)},\mathbb{F})$. Then $X^{(n)}=x+\vartheta^{(n)}.S^{(n)}$ for each $n\in \mathbb{N}^+$. Theorems \ref{HX=} and \ref{eq-HX} shows that $\vartheta\in\mathcal{L}^{\llbracket{0,+\infty}\llbracket}(S)$, and that $(T_n,\vartheta^{(n)}.S^{(n)})$ is an FCS for ${\vartheta}_{\bullet}{S}\in\mathcal{S}^{\llbracket{0,+\infty}\llbracket}$. From the assumption \eqref{ConditionWealth} and Remark \ref{HXB=HX}, we deduce $X=x\mathfrak{I}_{\llbracket{0,+\infty}\llbracket}+{\vartheta}_{\bullet}{S}=x+{\vartheta}.{S}$, which shows that $\vartheta$ is a self-financing strategy in the financial market $(S,\mathbb{F})$.

  \item [$(2)$]Let $a>0$ be a constant. A strategy $\vartheta$ is an $a$-admissible (resp. admissible) strategy in the financial market $(S,\mathbb{F},\llbracket{0,+\infty}\llbracket)$ if and only if it is an $a$-admissible (resp. admissible) in the financial market $(S,\mathbb{F})$. The proof is analogous to that of $(1)$.

  \item [$(3)$] The financial market $({S},\mathbb{F},\llbracket{0,+\infty}\llbracket)$ satisfies NA if and only if the financial market $({S},\mathbb{F})$ satisfies NA. This statement is a direct result of $(2)$.
\end{enumerate}
\end{remark}

The following theorem reveals that self-financing strategies and admissible strategies in the financial market $({S},\mathbb{F},B)$ can be characterized through stochastic integrals on $B$.
For simplicity, we say $U\geq a,\; (\mathbb{P},B)\text{-\;}a.s.$ if there exists an FCS $(T_n,U^{(n)})$ for $U\in \mathcal{D}^B$ such that for each $n\in \mathbb{N}^+$ and for all $t\in \mathbb{R}^+$,
\[
U^{(n)}_t\geq a,\; \mathbb{P}\text{-\;} a.s.,
\]
where $\mathcal{D}$ is a collection of processes, $a\in \mathbb{R}$, and $U\in\mathcal{D}^B$.

\begin{theorem}
In the financial market $({S},\mathbb{F},B)$, suppose that the investor's wealth $X$ and strategy ${\vartheta}\in L({S},\mathbb{F},B)$ are given by (\ref{wealth}).
\begin{itemize}
  \item [$(1)$] ${\vartheta}$ is self-financing if and only if
the wealth ${X}$ can be expressed as
  \begin{equation}\label{wealth1}
      {X}={x}_0\mathfrak{I}_B+{\vartheta}_{\bullet}{S}.
  \end{equation}
  \item [$(2)$] ${\vartheta}$ is $a$-admissible if and only if
${\vartheta}$ satisfies
      \[
      {\vartheta}_{\bullet}{S}\geq -a,\;(\mathbb{P},B)\text{-\;}a.s..
      \]
\end{itemize}
\end{theorem}
\begin{proof}
$(1)$ {\it Necessity.} Suppose that ${\vartheta}$ is self-financing. Let $(T_n,\vartheta^{(n)})$  and $(T_n,S^{(n)})$ be the FCSs for ${\vartheta}\in \mathcal{P}^B$ and $S\in \mathcal{S}^B$ respectively such that for each $n\in \mathbb{N}^+$, the strategy ${\vartheta}^{(n)}$ is self-financing in the financial market $({S}^{(n)},\mathbb{F})$. Then for each $n\in \mathbb{N}^+$, we deduce ${\vartheta}^{(n)}\in L({S}^{(n)},\mathbb{F})$ and
\begin{equation}\label{wealth2}
{X}^{(n)}={x}_0+{\vartheta}^{(n)}.{S}^{(n)},
\end{equation}
where ${X}^{(n)}$ is the investor's wealth.
Now using \eqref{ConditionWealth}, \eqref{wealth2} and Theorem \ref{eq-HX}, the expression \eqref{wealth1} can be obtained easily by
\begin{align*}
{X}I_{B\llbracket{0,T_n}\rrbracket}
&={X}^{(n)}I_{B\llbracket{0,T_n}\rrbracket}\\
&=({x}_0+{\vartheta}^{(n)}.{S}^{(n)})I_{B\llbracket{0,T_n}\rrbracket}\\
&=({x}_0\mathfrak{I}_B+{\vartheta}_{\bullet}{S})I_{B\llbracket{0,T_n}\rrbracket}
\end{align*}
for each $n\in \mathbb{N}^+$.

{\it Sufficiency.} Suppose \eqref{wealth1} holds. Let $(\tau_n)$ be an FS for $B$. For each $n\in \mathbb{N}^+$, put $T_n=\tau_n$, $\vartheta^{(n)}=\vartheta^{\tau_n}$, $S^{(n)}=S^{\tau_n}$ and $X^{(n)}=X^{\tau_n}$. Theorem \ref{fcs-p} shows that $(T_n,\vartheta^{(n)})$ and $(T_n,S^{(n)})$ are FCSs for $\vartheta\in \mathcal{P}^B$ and $S\in \mathcal{S}^B$ respectively, and Theorem \ref{eq-HX} shows that $(T_n,\vartheta^{(n)}.S^{(n)})$ is an FCS for ${\vartheta}_{\bullet}{S}\in \mathcal{S}^B$ satisfying
\begin{equation}\label{self1}
\vartheta^{(n)}.S^{(n)}=\vartheta^{T_n}.S^{T_n}=(\vartheta_{\bullet}S)^{T_n}, \quad n\in \mathbb{N}^+.
\end{equation}
From the assumption \eqref{ConditionWealth}, it is easy to see that  $X^{(n)}$ is the investor's wealth relative to his strategy $\vartheta^{(n)}$ in the financial market $({S}^{(n)},\mathbb{F})$.  Since  \eqref{self1} implies
\[
X^{(n)}=X^{T_n}=x_0+\vartheta^{(n)}.S^{(n)}, \quad n\in \mathbb{N}^+,
\]
$\vartheta^{(n)}$ is self-financing in the financial market $({S}^{(n)},\mathbb{F})$.  Therefore, ${\vartheta}$ is self-financing in the financial market $({S},\mathbb{F},B)$.

$(2)$ The proof is analogous to that of $(1)$.
\end{proof}

\begin{definition}\label{portfolio-problem}
Let the investor's wealth ${X}$ and strategy ${\vartheta}\in L({S},\mathbb{F},B)$ be given by (\ref{wealth}), $(T_n,S^{(n)})$ be an FCS for $S\in \mathcal{S}^B$, and $\varphi$ be a utility function, for instance, a logarithmic utility function:
\[
\varphi(x)=\ln x,\quad x>0.
\]
An admissible strategy $\pi$ in the financial market $({S},\mathbb{F},B)$ is said to be optimal, if there exists an FCS $(T_n,\pi^{(n)})$ for ${\pi}\in \mathcal{P}^B$ such that for each $n\in \mathbb{N}^+$, $\pi^{(n)}$ is the optimal strategy for the following portfolio problem in the financial market $({S}^{(n)},\mathbb{F})$:
\begin{equation}\label{optimal-problem}
\left\{
\begin{aligned}
&\pi^{(n)}=\arg\sup\left\{\mathbb{E}\left(\varphi({X}^{(n)}_{T_n})\right): {\vartheta\in l_0({S}^{(n)},\mathbb{F})}\right\},\\
&s.t.\quad {X}^{(n)}={x}_0 +\vartheta.{S}^{(n)}\geq 0.
\end{aligned}
\right.
\end{equation}
\end{definition}

\subsection{A simple example}

We study a simple financial market on a PSIT. Assume that a risky asset with default is traded in the financial market and that an investor aims to maximize the expected value of a utility function.

Let the stochastic time-horizon $B$ of the investor is given by \eqref{time}, where $\tau>0$ is a predictable time  representing the time when a default occurs in the credit risk setting, and where $T>0$ is a constant representing the deterministic terminal time for the investor. The stock price $S$ is a semimartingale on $B$ given by
\begin{equation}\label{price}
  S=s_0\mathfrak{I}_B+S_{\bullet}Z,
\end{equation}
where $s_0>0$ is a positive constant, the process $Z$ on $B$ is defined by
\begin{equation}\label{driven}
\left\{
  \begin{aligned}
  Z&=\left(\sum\limits_{n=1}^{{+\infty}}
      Z^{(n)}I_{\rrbracket{T_{n-1},T_n}
      \rrbracket}\right)\mathfrak{I}_B,\quad T_0=0,\quad T_n=\tau_n\wedge T,\; n\in \mathbb{N}^+,\\
  Z^{(n+1)}&=Y^{(n+1)}+(Z^{(n)}-Y^{(n+1)})^{T_{n}},\quad Z^{(1)}=Y^{(1)},\; n\in \mathbb{N}^+,\\
  Y^{(n)}_t&=\mu_nt+\sigma_nW^{(n)}_t,\quad t\in\mathbb{R}^+,\; n\in \mathbb{N}^+,
  \end{aligned}
\right.
\end{equation}
$(\tau_n)$ is a sequence of stopping times announcing $\tau$, $(W^{(n)})$ is a sequence of standard $\mathbb{F}$-Brownian motions, and for each $n\in \mathbb{N}^+$, $\mu_n\in \mathbb{R}$ and $\sigma_n>0$ are constants. Obviously, $(T_n)$ is an FS for $B$. Example \ref{general-X} shows $Z\in\mathcal{S}^B$ with the FCS $(T_n,Z^{(n)})$, and Corollary \ref{expZ} guarantees the existence of the process $S$ of the stock price. We denote such a financial market by the triplet $(S,\mathbb{F},B)$.

The FCS $(T_n,Z^{(n)})$ for $Z\in\mathcal{S}^B$ in \eqref{driven} can be viewed as a switching adjustment of the stock price, and such a switching adjustment is analogous to a switching control in the theory of optimal switching (see, e.g., Subsection 5.2 in \cite{Pham}). Empirical evidence (see, e.g. \cite{Chava,Carr}) shows that default-risk has an effect on stock returns and volatilities.  Hence, with the default time approaching, the default-risk tends to change and the stock features may switch to different values. The sequence $(T_n)$ of stopping times represents the condition of ``when to switch", and for each $n\in \mathbb{N}^+$, the semimartingale $Z^{(n)}$ captures the new stock features at time $T_{n-1}$ until time $T_n$ or the adjustment of ``where to switch". More importantly, the switching adjustment even allows the stock price to switch the driven processes (i.e., Brownian motions) of the the aggregate risk.

The construction of stock price $S$ shows that the default time $\tau$ is closely associated with the information of stock price, and then the conditions of ``when to switch" (i.e., the sequence $(T_n)$) are a part of stock features. Consequently, we can describe the relation between the conditions of ``when to switch" and the information of stock features through the following assumption: for each $n\in \mathbb{N}^+$, the $\mathbb{F}$-stopping time $T_n$ is not necessarily an $\mathbb{F}^{(n)}$-stopping time, but it is an $\mathbb{F}^{(n+1)}$-stopping time, where $\mathbb{F}^{(n)}=(\mathcal{F}^{(n)}_t,t\geq 0)$ is defined by
\[
  \mathcal{F}^{(n)}_t=\sigma\left\{W_s^{(i)}: 1\leq i\leq n,\;0\leq s\leq t\right\}\vee \mathcal{N},
\]
and $\mathcal{N}$ is the set of $\mathbb{P}$-null sets. Such an assumption means that the information of the current condition $T_n$ of ``when to switch" should be included in the information of the stock features before next condition $T_{n+1}$.

The essential difference between the financial market $(S,\mathbb{F},B)$ and classic financial markets is that the stock price $S$ does not indicate any information outside $B$.
In fact, the stock price $S$ in former is only defined on $B$ and driven by an infinite number of Brownian Motions while the stock prices in latter are essentially defined on $\llbracket{0,+\infty}\llbracket$ and driven by a finite number of Brownian Motions.  On the other hand, the financial market $(S,\mathbb{F},B)$ can degenerate into classic financial markets. If there is not any default time, i.e., $\tau=+\infty$ and $T_n=T$ for each $n\in \mathbb{N}^+$, then the stock price $S$ degenerates into the geometric Brownian motion
\[
\widetilde{S}_t=s_0\exp\left\{\left(\mu_1-\frac{\sigma_1^2}{2}\right)t+\sigma_1W_t^{(1)}\right\}.
\]
The stock price $\widetilde{S}$ is generally adopted in financial researches (e.g., the well-known Black-Scholes model \cite{Black}), and the portfolio problem of Definition \ref{portfolio-problem} becomes classic portfolio allocation (see, e.g., Subsection 2.2.1 in \cite{Pham}).

\begin{proposition}\label{NA}
Suppose that $\mu_n=\mu$ and $\sigma_n=\sigma$ for each $n\in \mathbb{N}^+$, and that the sequence $(W^{(n)})$ satisfies
\[
\langle W^{(i)},W^{(j)}\rangle_t=\rho_{ij}t,\quad i,\;j\in\mathbb{N}^+,\quad t\in \mathbb{R}^+,
\]
where $\rho_{ij}\in [-1,1]$ is a constant. Put $A=\widetilde{A}\mathfrak{I}_B$ with $\widetilde{A}_t=t$, and define the following process on $B$:
\begin{equation}
\left\{
\begin{aligned}
w&=\left(\sum\limits_{n=1}^{{+\infty}}
      w^{(n)}I_{\rrbracket{T_{n-1},T_n}
      \rrbracket}\right)\mathfrak{I}_B,\\
w^{(n+1)}&=W^{(n+1)}+(w^{(n)}-W^{(n+1)})^{T_{n}},\quad w^{(1)}=W^{(1)},\; n\in \mathbb{N}^+.
\end{aligned}
\right.
\end{equation}
Then we have the following statements:
\begin{itemize}
  \item [$(1)$] The financial market $({S},\mathbb{F},B)$ satisfies NA.
  \item [$(2)$] Suppose that, for each $n\in \mathbb{N}^+$,
\[
F^{(n)}_t=\mathbb{P}\left(\tau_n\leq t\left|\mathcal{F}_t^{(n)}\right.\right)
\]
is an increasing absolutely continuous process w.r.t. Lebesgue measure, with a density denoted by $f^{(n)}$, i.e., $F^{(n)}_t=\int_0^tf^{(n)}_sds$.
If $\varphi$ is the logarithmic utility function, then the optimal strategy $\pi$ in the financial market $({S},\mathbb{F},B)$ is given by
\begin{equation}\label{optimal-solution}
\pi=\frac{x_0\mu}{\sigma^2}\exp\left(\frac{\mu^2}{2\sigma^2}A+\frac{\mu}{\sigma}w\right).
\end{equation}
\end{itemize}
\end{proposition}
\begin{proof}
By induction, we deduce $\langle w^{(n)},w^{(n)}\rangle_t=t$ and $Z^{(n)}_t=\mu t+\sigma w^{(n)}_t$ for $t\in \mathbb{R}^+$. L\'{e}vy theorem (see, e.g., Theorem 3.16 in \cite{Karatzas-Shreve}) shows that $w^{(n)}$ is also a standard Brownian motion. Example \ref{gen-M} shows that $(T_n,w^{(n)})$ is an FCS for $w\in(\mathcal{M}_\mathrm{loc})^B$.
For each $n\in \mathbb{N}^+$, define the process $S^{(n)}$ as the geometric Brownian motion
\[
S^{(n)}_t=s_0\exp\left(\left(\mu-\frac{\sigma^2}{2}\right) t+\sigma w^{(n)}_t\right),
\]
or equivalently the SDE
\[
dS^{(n)}_t=S^{(n)}_t(\mu dt+\sigma dw^{(n)}_t), \quad S^{(n)}_0=s_0,\; t\in \mathbb{R}^+.
\]
From Corollary \ref{expZ} and \eqref{price}, we deduce that for each $n\in \mathbb{N}^+$,
\begin{align*}
S^{T_{n}}
&=s_0\exp\left\{Z^{T_n}-\frac{1}{2}\langle Z^c\rangle^{T_n}\right\}\\
&=s_0\exp\left\{(Z^{(n)})^{T_n}-\frac{1}{2}\langle (Z^{(n)})^c\rangle^{T_n}\right\}\\
&=s_0\exp\left(\left(\mu-\frac{\sigma^2}{2}\right) \widetilde{A}^{T_n}+\sigma (w^{(n)})^{T_n}\right)\\
&=(S^{(n)})^{T_n},
\end{align*}
where in the second equality we use the fact that $(T_n,Z^{(n)})$ and $(T_n,\langle (Z^{(n)})^c\rangle)$ are FCSs for $Z\in \mathcal{S}^B$ and $\langle Z^c\rangle\in(\mathcal{A}^{+}_{\mathrm{loc}}\cap \mathcal{C})^B$ (Corollary \ref{qr-pc} and Theorem \ref{XTc}), i.e., $ZI_{\llbracket{0,T_n}\rrbracket}=Z^{(n)}I_{\llbracket{0,T_n}\rrbracket}$ and $\langle Z^c\rangle I_{\llbracket{0,T_n}\rrbracket}=\langle (Z^{(n)})^c\rangle I_{\llbracket{0,T_n}\rrbracket}$ for each $n\in \mathbb{N}^+$.
Then the relations $S^{(n)}\in \mathcal{S}$ and
\[
SI_{\llbracket{0,T_n}\rrbracket}=S^{T_{n}}I_{\llbracket{0,T_n}\rrbracket}
=(S^{(n)})^{T_n}I_{\llbracket{0,T_n}\rrbracket}
=S^{(n)}I_{\llbracket{0,T_n}\rrbracket},\quad n\in \mathbb{N}^+
\]
imply that $(T_n,S^{(n)})$ is an FCS for $S\in\mathcal{S}^B$. And it is easy to see that $(T_n,(S^{(n)})^T)$ is also an FCS for $S\in\mathcal{S}^B$.

$(1)$ It is well-known that the financial market $((S^{(n)})^T,\mathbb{F})$ satisfies NA (see, e.g., Theorem 12.1.8 in \cite{Oksendal}) for each $n\in \mathbb{N}^+$, and hence $({S},\mathbb{F},B)$ satisfies NA.

$(2)$ For each $n\in \mathbb{N}^+$, from \cite{Blanchet}, the optimal strategy $\pi^{(n)}$ of \eqref{optimal-problem} in $(S^{(n)},\mathbb{F})$  is given by
\[
\pi^{(n)}_t=\frac{x_0\mu}{\sigma^2}\exp\left(\frac{\mu^2}{2\sigma^2}t+\frac{\mu}{\sigma}w^{(n)}_t\right),
\]
which implies \eqref{optimal-solution} easily.
\end{proof}

\section{Concluding remarks}\label{section7}\noindent
In this paper, we focus on various classes of processes on PSITs, and use them to investigate three kinds of stochastic integrals on PSITs and their fundamental properties. Analogous to processes on PSITs, there are two features of stochastic integrals on PSITs: (1) they are defined only on PSITs, and their values outside PSITs do not matter; (2) they can be characterized by classic stochastic integrals.

In addition to PSITs, optional sets of interval type (in short: OSITs) can be also studied. Actually, an OSIT can be expressed by \eqref{B} with $T_F$ just being a stopping time (see Appendix), and various classes of  processes on OSITs are well defined in Definition 8.19 of \cite{He}. Except for results relative to FSs, our investigation of processes on PSITs in Section \ref{section2} and L-S integrals on PSITs in Section \ref{section3} can be easily extended into those on OSITs. On the other hand, stochastic integrals on PSITs of predictable process w.r.t. local martingales and stochastic integrals on PSITs of predictable process w.r.t. semimartingales can not be directly extended into those on OSITs, because for two local martingales on an OSIT, their quadratic covariation in the manner of Definition \ref{de-quad} may not be unique (see the example in Appendix).
Therefore, our future work is to investigate stochastic integrals on OSITs of predictable process w.r.t. local martingales and stochastic integrals on OSITs of predictable process w.r.t. semimartingales.

\section*{Appendix}
From Theorem 8.17 in \cite{He}, $B$ is an OSIT if and only if
$I_B=I_FI_{\llbracket{0,T}\llbracket}+I_{F^c}I_{\llbracket{0,T}\rrbracket}$, or equivalently,
      \begin{equation*}
         B=\llbracket{0,T_F}\llbracket\;\cap\;\llbracket{0,T_{F^c}}\rrbracket,
      \end{equation*}
where $T$ is a stopping time, $F\in \mathcal{F}_{T}$, and $T_F>0$.

Let $T$ be random variable with a unit exponential law, and $\mathbb{F}=(\mathcal{F}_t,t\geq 0)$ be the natural filtration of the process $I_B$ with $B=\llbracket{0,T}\llbracket$. From Example 6.2.5 in \cite{Cohen} and Lemma 2.1 in \cite{Aksamit}, $T$ is an $\mathbb{F}$-totally inaccessible time with $T>0$ and $\mathbb{P}(T<{+\infty})>0$. It is obvious that $B$ is an OSIT, and in the manner of Definition \ref{processB}, we can also define the process $M=A^p\mathfrak{I}_B$ on $B$, where $A=I_{\llbracket{T,{+\infty}}\llbracket}$ and $A^p$ is the compensator of $A$. Then $M\in(\mathcal{M}_{\mathrm{loc}})^B$ with an FCS $(T_n=T,M^{(n)}=A^p-A)$, and by Proposition 2.4 in \cite{Aksamit}, $A^p_t=T\wedge t$. On the one hand, we have $M^{2}-M^{2}=0\in(\mathcal{M}_{\mathrm{loc},0})^B$ with $M^{2}\in\mathcal{V}^B$ and $\Delta M^{ 2}=(\Delta M)^{2}=0$. On the other hand, from $A^p-A\in\mathcal{V}$, we deduce that $(A^p-A)^c=0$ and
\[
[A^p-A]=\sum_{s\leq\cdot}(\Delta(A^p-A)_s)^{ 2}=\sum_{s\leq\cdot}(\Delta A_s)^{2}=\sum_{s\leq\cdot}\Delta A_s=A.
\]
Then for each $n\in \mathbb{N}^+$, the relations
\[
M^{ 2}I_{B\llbracket{0,T_n}\rrbracket}=((M^{(n)})^{ 2}-A)I_{B\llbracket{0,T_n}\rrbracket} \quad \text{and} \quad (M^{(n)})^{2}-A=(M^{(n)})^{2}-[M^{(n)}]\in\mathcal{M}_{\mathrm{loc},0}
\]
imply $M^{ 2}\in(\mathcal{M}_{\mathrm{loc},0})^B$, which yields $M^{ 2}-0\mathfrak{I}_B\in(\mathcal{M}_{\mathrm{loc},0})^B$ and $\Delta (0\mathfrak{I}_B)=(\Delta M)^{2}=0$. Therefore, both $M^{2}$ and $0\mathfrak{I}_B$ ($M^{ 2}\neq 0\mathfrak{I}_B$) can be chosen as the process $V\in \mathcal{V}^B$ such that $M^{2}-V\in (\mathcal{M}_{\mathrm{loc},0})^B$ and $\Delta V=(\Delta M)^2$.

\end{document}